\documentclass{article}
\usepackage{theorem}
\usepackage{amsmath,amssymb}  
\usepackage{bm}  

\newtheorem{theorem}{Theorem}
\newtheorem{proposition}[theorem]{Proposition}
\newtheorem{lemma}[theorem]{Lemma}
\newtheorem{cor}[theorem]{Corollary}

{\theorembodyfont{\normalfont\rmfamily}
\newtheorem{definition}[theorem]{Definition}

\newtheorem{remarque}[theorem]{Remark}

} \makeatletter 

\makeatletter
\def\operatorname#1{\mathop{\operator@font #1}\nolimits}%
\makeatother
\newcommand{\R}{\mathbb{R}}
\newcommand{\C}{\mathbb{C}}

\newcommand{\Z}{\mathbb{Z}}

\newcommand{\half}{{\frac{1}{2}}}
\newcommand{\thalf}{{\tfrac{1}{2}}}

\newcommand{\Imc}{{\mathcal{I}}m\ }

\newcommand{\Imag}{\mathcal{I}m\,}
\newcommand{\Id}{\operatorname{Id}}

\newcommand{\Sp}{\operatorname{Sp}}
\newcommand{\Or}{\operatorname{O}}
\newcommand{\card}{\operatorname{Card}}

\newcommand{\Gr}{\operatorname{Gr}}
\newcommand{\End}{\operatorname{End}}
\newcommand{\Mat}{\operatorname{Mat}}
\newcommand{\sign}{\operatorname{Sign}}
\newcommand{\tr}{^{\tau}\!}

\newcommand{\Ker}{\operatorname{Ker}}
\newcommand{\im}{\operatorname{Im}}
\newcommand{\mat}{\left(\begin{array}{cc}F & B \\C & D \end{array}\right)}

\newenvironment{proof}[1][{}]{{ \textsc{Proof{#1}:~}}}{{\hspace*{\fill}$\square$\\}}

\def\adots{\mathinner{\mkern2mu\raise 1pt\hbox{.}\mkern 3mu\raise
4pt\hbox{.}\mkern2mu\raise 8pt\hbox{{.}}}}

\title{The Conley-Zehnder index for a path of symplectic matrices}
\author{Jean Gutt \\
	\small\hbox{\parbox[t]{1.9in}{\begin{center}
	D{\'e}partement de Math{\'e}matique \\
	Universit{\'e} Libre de Bruxelles \\
	Campus Plaine, C. P. 218 \\
	Boulevard du Triomphe \\
	B-1050 Bruxelles \\
	Belgium\\
	jeangutt\char64ulb.ac.be\end{center}}}
	\hbox{\parbox[t]{.6in}{\begin{center}\rm and \end{center}}}
	\hbox{\parbox[t]{1.9in}{\begin{center}
	Universit\'e de Strasbourg\\
	IRMA\\
	7 rue Ren\'e Descartes\\
	67000 Strasbourg\\
	France\\
	{gutt\char64math.unistra.fr}\end{center}}}
	}
\date{} 
\begin{document} 

\maketitle

\begin{abstract}
We give here a self contained and elementary introduction to the Conley-Zehnder index for a path of symplectic matrices.
We start from the definition of the index as the degree of a map into the circle for a path starting at the 
identity and ending at a matrix for which $1$ is not an eigenvalue.
We prove some properties which characterize this index
using  normal forms for symplectic matrices obtained from geometrical considerations.
We explore the relations to Robbin-Salamon index
for paths of Lagrangians. We give an axiomatic  characterization of
the generalization of the Conley-Zehnder index for any continuous path of symplectic matrices
defined by Robbin and Salamon.
\end{abstract}

\section*{Introduction}
One can find in the literature different definitions of the index of a path of
symplectic matrices; the aim of this paper is to clarify the relations between those definitions
and to give a self contained presentation of the Conley-Zehnder index and its properties.
Along the way, we are led to establish normal forms for symplectic matrices via elementary geometrical methods.
We also give an axiomatic characterization of the generalization of the Conley-Zehnder index
defined by Robbin and Salamon for any path of symplectic matrices.\\

The first index for some paths of symplectic matrices was introduced by Arnold \cite{Arnold} who was considering loops of symplectic matrices.
The Conley-Zehnder index of a path \cite{Conley} of symplectic matrices is an integer
associated to each path of symplectic matrices which starts from the identity
and ends at a matrix which does not admit $1$ as an eigenvalue.

This index is useful in the definition of some homologies, in particular
Floer homology and contact homology (see for instance Salamon \cite{Salamon} and Bourgeois \cite{Bourgeois}).
These homologies  give invariants for corresponding geometrical structures
and are defined as  generalizations of Morse theory. The vector spaces of
the complexes defining those homologies are spanned by
 critical points of a functional defined on an infinite dimensional
space, typically a space of loops with values in a manifold.
To define the grading of those vector spaces, one has to associate
an integer to a critical point of the functional, typically a loop which is a solution
of some ordinary differential equation.
In classical variational calculus, one uses Morse index; but here
one cannot use the classical Morse theory since the Hessian has infinite dimensional stable and unstable manifolds.
The way it is done is to associate to those special loops a path
of symplectic matrices and the number associated to the loop
is  obtained from the Conley-Zehnder index of the corresponding path of matrices.
Links between the Conley-Zehnder index and the Morse index have been obtained by Viterbo \cite{Viterbo} for cotangent bundles and for $C^2$ small hamiltonians by Salamon and Zehnder \cite{SalamonZehnder}. Links between the Conley-Zehnder index and Leray index, including generalized
index and axiomatic characterization, have been obtained by de Gosson \cite{DeGosson} (see also the references therein).
\\

We have chosen to start here with the definition of the Conley-Zehnder index given in Salamon \cite{Salamon}, as
 the degree of a map into the circle for an admissible path, i.e. a path $\psi:[0,1]\rightarrow  \Sp(\R^{2n},\Omega_0)$ starting at the 
identity ($\psi(0)=\Id$) and ending at a matrix for which $1$ is not an eigenvalue.

An important result used in this definition is the fact that semisimple matrices are dense in the symplectic group.
This we prove, using  normal forms for symplectic matrices.
Normal forms of symplectic matrices can be found in the literature (for instance in Long \cite{Long}) ; we give here
a construction using elementary geometrical methods; the normal forms  we obtain
agree with those of Long, with some more details which are useful in our characterization of the generalized
index.

We include proofs of the properties of the Conley-Zehnder index given in \cite{Salamon}: its naturality (i.e. invariance under conjugation by a path of symplectic matrices), the homotopy property
(i.e. the index is constant under continuous deformations of admissible paths of symplectic matrices), the zero property (i.e.  it vanishes on an admissible path of matrices such that $\psi(s)$ has no eigenvalue on the circle for $s>0$), the product property (i.e. it is additive with respect to the decomposition of the symplectic vector space as a symplectic direct sum of  two symplectic subspaces), the loop property (i.e. the index of the path $\psi'$ obtained by multiplying the path $\psi$  by a loop of symplectic matrices $\phi$ is the sum of the index of $\psi$ and twice the Maslov
index of $\phi$), the signature property (i.e. the index of a path $\psi(t)=\exp tJ_0S$ is equal to half the signature of $S$ if $S$
is a symmetric non degenerate matrix with all eigenvalues of absolute value $<2\pi$), the determinant property $\bigl($i.e. the parity of the Conley-Zehnder index depends only on $\psi(1)$, precisely 
$(-1)^{n- \mu_{CZ}(\psi)} = \operatorname{sign}\det(\Id - \psi(1))\bigr)$ and the inverse property
$\bigl($i.e. $\mu_{CZ}(\psi^{-1}) = \mu_{CZ}(\psi\tr) = -\mu_{CZ}(\psi)\bigr)$.\\
We show that the homotopy, loop and signature properties characterize the Conley-Zehnder index on admissible paths.\\

In \cite{RobbinSalamon}, Robbin and Salamon define a Maslov-type index with half integer values for a continuous path of Lagrangians in a symplectic vector space $(\R^{2n},\Omega_0).$
This index depends on the choice of a reference Lagrangian.  
Robbin and Salamon generalize the Conley-Zehnder index to any continuous path of symplectic matrices as the index of the path of Lagrangians in $(\R^{2n}\times \R^{2n},\overline{\Omega} = -\Omega_0 \oplus \Omega_0)$ given by the graphs of the matrices, with reference Lagrangian given by the diagonal.
We include the proof that this new index, which we call the Robbin-Salamon index, does  satisfy the homotopy, loop and signature properties on the set of admissible paths.
We show that this  index is invariant under conjugation by a path of symplectic matrices.
We prove that the Robbin-Salamon index is characterized by its invariance under homotopies with fixed end points, its additivity under catenation of paths,   the fact that it vanishes on any paths where the dimension of the eigenspace of eigenvalue $1$ is constant and its value on paths  $\psi(t)=\exp tJ_0S$ with  $S$
symmetric with all eigenvalues of absolute value $<2\pi$.

Another index associated by Robbin and Salamon to a path of symplectic matrices is the index of the path of Lagrangians in $(\R^{2n},\Omega_0)$ given by the image
of a fixed Lagrangian under the matrices.
We show that this index does not coincide with the previous one.
Nevertheless, we show that both indices coincide on symplectic shears.\\

Although we have included proofs of many known  results  in this paper, we hope that this presentation may  help a reader who is first introduced to the subject.\\

{\bf{Thanks}}{ This work is an extension of part of my Master thesis presented in May 2010 at the Universit\'e Libre de Bruxelles.
I thank my supervisors, Fr\'ed\'eric Bourgeois and Alexandru Oancea who introduced me to this subject, encouraged me to write this text and suggested many improvements.}

\tableofcontents

\section{The symplectic and the unitary groups}

Consider a real vector space $V$ of dimension $2n$ with a non degenerate skewsymmetric bilinear form $\Omega$ and a compatible complex structure $j$  (i.e. $j : V\rightarrow V$ is linear,
$j^2=-\Id, \Omega(jv,jw)=\Omega(v,w)~\forall v,w\in V$ and the symmetric bilinear form $g$ defined by $g(v,w):=\Omega(v,jw)$ is positive definite).\\
We choose a basis of $V$ (identifying thus $V$ with $\R^{2n})$ in which the  matrix associated to $\Omega$ is $\Omega_0: =
\left( \begin{array}{cc}
	0 & \Id \\
	-\Id & 0
\end{array} \right)$
and the matrix associated to $j$ is $ J_0 :=
\left( \begin{array}{cc}
	0 & -\Id \\
	\Id & 0
\end{array} \right) .$
The matrix associated to $g$ in this basis is $\Id.$
We obtain such a basis $\{\, e_1,\ldots,e_n,f_1,\ldots,f_n\,\}$  by induction on $n$,  choosing a vector $e_1$ such that $g(e_1,e_1)=1$ and letting $f_1=je_1.$
We consider then $V'$ the  subspace which is orthogonal to the space generated by $\{ e_1, f_1\};$ it is symplectic and of dimension $2n-2.$
\begin{definition}
	The symplectic group $\Sp(V,\Omega)$ is the set of linear transformations of $V$ which preserve $\Omega.$
	\begin{eqnarray}
		\Sp(\R^{2n},\Omega_0)&=& \left\{ {A} \in \Mat (2n\times 2n,\R) \ \vert \ {A}\tr \ \Omega_0 {A} = \Omega_0 \right\} \\
		&=& \left\{ \mat \left\vert 
		\begin{array}{l}
			F, B, C, D  \in \Mat (n\times n,\R)\\
			F\tr C 
			\textrm{~and~}  B\tr D\ \textrm{~are symmetric}\\
			F\tr D -  C\tr  B = \Id
		\end{array}\right. \right\}
	\end{eqnarray}
where ${A}\tr$ denotes the transpose of the matrix $A$.\\
	In particular $\Sp(\R^{2},\Omega_0)=Sl(2,\R)= \left\{ \, {A} \in \Mat (2\times 2,\R)\, \vert \, \det(A)=1\,\right\}.$
	
	The orthogonal group $\Or(V,g)$ is the set of linear transformations of $V$ which preserves $g.$
	We have $\Or\left(\R^{2n}\right) = \left\{ {A} \in \End \left( \R^{2n}\right)\ \vert \ {A}\tr {A} = \Id \right\}$
	\begin{equation}
		\Or\left(\R^{2n}\right) = \left\{ \mat\left\vert
		\begin{array}{ccc}
			F\tr F +  C\tr C & = & \Id \\ 
			B\tr B+D\tr D & = & \Id \\ 
			F\tr B+ C\tr D & = & 0
		\end{array}\right. \right\}.
	\end{equation}
\end{definition}
\begin{definition}
	The real vector space $V$ of dimension $2n$ with the complex structure $j$ is identified with a complex vector space $W$ of dimension $n$: $\left( \R^{2n},J_0\right) \cong \C^n$
	identifying the point $\left(x,y\right):=\left(x_1,\ldots,x_n,y_1,\ldots,y_n\right)\in \R^{2n}$ with the point
	\begin{equation*}
		z:=\left(z_1=x_1+iy_1,\ldots,z_n=x_n+iy_n\right) \in \C^n
	\end{equation*}
	and the map $j$ corresponds  to the multiplication by $i:$
	\begin{equation*}
		J_0
		\left(\begin{array}{c}
			x\\
			y
		\end{array}\right) =
		\left( \begin{array}{cc}
			0&-\Id \\
			\Id&0
		\end{array} \right)
		\left( \begin{array}{c}
			x \\
			y
		\end{array} \right) =
		\left( \begin{array}{r}
			-y \\
			x
		\end{array} \right)\Leftrightarrow i(x+iy)=-y+ix.
	\end{equation*}
	The Hermitian product on $\C^n, \left<( z_1,\ldots,z_n),(z'_1,\ldots,z'_n)\right> = \overline{z_1}z'_1+\cdots+\overline{z_n}z'_n,$
	corresponds to the bilinear form $h$ on $\R^{2n}$ determined by the the metric and the $2$-form:
	\begin{eqnarray*}
		h \bigl( (x,y), (x',y') \bigr) & = & (x_1-i y_1)(x'_1+iy'_1)+\cdots+(x_n-i y_n)(x'_n+iy'_n)\\
		& =  & (x_1x'_1 + y_1 y'_1 + \cdots + x_n x'_n + y_n y'_n )\\
		& & \qquad \qquad+i (x_1y'_1-y_1x'_1 + \cdots + x_n y'_n - y_n x'_n )\\
		& = & g \bigl( (x,y), (x',y') \bigr) + i \Omega \bigl( (x,y), (x',y') \bigr).
	\end{eqnarray*}
\end{definition}
A $\C$-linear map $ \tilde{{A}}$ on $\C^n,(\tilde{{A}} : \C^n \rightarrow \C^n ~\textrm{with}~\tilde{A}\circ i= i\circ \tilde{A})$ corresponds to a  $\R$-linear map $A$ on
$\R^{2n}~A : \R^{2n} \rightarrow \R^{2n}$ such that $AJ_0=J_0A.$
Thus
\begin{lemma}
	\begin{equation*}
		A = \mat \textrm{defines a }\C{\textrm{-linear map }} \tilde{A}{\textrm{  iff }}  D=F  \textrm { and } C=-B;
	\end{equation*}
	in this case $\tilde{A}(z) = Fx+By+i(Cx+Dy) = (F+iC)(x+iy)= (F+iC)z.$
\end{lemma}
\begin{definition}
	The unitary group $U(n)$ is the set of linear transformations of $\C^n$ which preserve the hermitian product.
	\begin{equation*}
		\textrm{U}\left(n\right) = \left\{ \tilde{\mathcal{A}} : \C^n \rightarrow \C^n \,\vert \,\C\textrm{-linear and}\ \overline{\tilde{\mathcal{A}}}\tr \tilde{\mathcal{A}}=\Id \right\}
	\end{equation*}
\end{definition}
\begin{proposition}
	Considering $\textrm{U}\left(n\right)$ as a set of transformation of $\R^{2n}$, we have:
	\begin{equation*}
		\textrm{U}\left(n\right) = \textrm{O}\left(\R^{2n}\right) \cap \textrm{Sp}\left(\R^{2n},\Omega_0\right).
	\end{equation*}
\end{proposition}
\begin{proof}
	If $A \in \End(\R^{2n})$ corresponds to a unitary transformation, we have
	\begin{equation*}
		g(Au,Av)+i\Omega (Au,Av) = g(u,v)+i\Omega (u,v)~\forall u,v\in \R^{2n}
	\end{equation*}
	thus $A\in  \textrm{O}\left(\R^{2n}\right) \cap \textrm{Sp}\left(\R^{2n},\Omega_0\right).$
	On the other hand, if $A \in \textrm{O}\left(\R^{2n}\right) \cap \textrm{Sp}\left(\R^{2n},\Omega_0\right),$ because $ g(u,v) = \Omega (u,jv)$ we have
	\begin{equation*}
		\left\{\begin{array}{cc}g(Au,Av) & =  g(u,v)  \\
		\Omega (Au,Av) & =  \Omega(u,v)
		 \end{array}\right.  \forall u,v\Rightarrow \Omega(Au,jAv) = \Omega(u,jv) = \Omega (Au,Ajv).
	\end{equation*}
	This implies, since $\Omega$ is non degenerate, that $jAv=Ajv \ \forall v$ and thus $J_0A=AJ_0.$ Thus $A$ corresponds to a $\C$-linear transformation of $\C^n$ which is clearly unitary.
\end{proof}
\begin{definition}
	If $A \in \End(\R^{2n})$ corresponds to a $\C$-linear transformation $\tilde{A},$  $A=
	\left(\begin{array}{cc}
		B & -C\\
		C & B
	\end{array}\right),$
	and we define
	\begin{equation*}
		{ \det}_{\C}A := \det\tilde{A}=\det(B+iC).
	\end{equation*}
\end{definition}
\subsection{Topology of $\Sp(\R^{2n},\Omega_0)$}\label{preliminaries}
We show in this paragraph that the fundamental group of the symplectic group is the group $\Z.$
The arguments are essentially taken from \cite{Abbondandolo} section $1.3.3 .$\\
We use the symbol $ \operatorname{diag} (a_1,\ldots,a_{2n})$ to denote a diagonal matrix
with entries $a_1,\ldots,a_{2n}$ on the diagonal.
\begin{theorem}
	Every element $A\in \Sp(\R^{2n},\Omega_0)$ admits a unique polar decomposition
	\begin{equation*}
		A=OP
	\end{equation*}
	where $O$ is a symplectic and orthogonal  (thus unitary) matrix and where the matrix $P=\left(A\tr A \right)^{\half}$ is a symplectic   positive definite symmetric matrix.
	We have $P=\exp \half S$ where $S$ is a symmetric matrix belonging to the symplectic Lie algebra
	\begin{equation}
		\textrm{sp}(\R^{2n},\Omega_0)=\{ \Lambda \in \Mat (2n\times 2n,\R) \,\vert\,\Lambda^\tau \Omega_0+\Omega_0\Lambda=0\}.
	\end{equation}
	The group $\Sp(\R^{2n},\Omega_0)$ is homeomorphic to the topological product of the unitary group $U(n)$ and of the vector space of real symmetric matrices $2n\times 2n$ in the
	symplectic Lie algebra $\textrm{sp}(\R^{2n},\Omega_0).$
\end{theorem}
\begin{proof}
	The matrix $ A\tr A$ is clearly symmetric and positive definite (because the scalar product $v.A\tr A v$ with $v\in \R^{2n}$ is zero iff $A v.A v=0$ iff $A v=0$ iff $v=0$ because $A$ is invertible).
	We can thus find an orthogonal matrix $K$ such that $KA\tr A K^{-1}=\operatorname{diag} (a_1,\ldots,a_{2n})$ with all $a_i>0.$
	We define the symmetric matrices $P=K^{-1} \operatorname{diag} (\sqrt{a_1},\ldots,\sqrt{a_{2n}})K$ and $S=K^{-1}\operatorname{diag} (\log a_1,\ldots,\log a_{2n})K.$
	The matrix $S$ is the unique symmetric matrix such that $\exp S= A\tr A$ and $P$ is the unique symmetric positive definite matrix such that $P^2=A\tr A.$
	
	\noindent [Uniqueness comes from the fact  that if $S$ is symmetric, then $S$ is diagonalizable and the eigenvectors of $S$ are the eigenvectors of $\exp S.$
	Thus if $S$ and $S'$ are symmetric and such that $\exp S=\exp S',$ they have the same eigenvectors and are simultaneously diagonalizable.
	Thus $\exp S \exp -S'= \exp(S-S')= K\operatorname (b_1,\ldots,b_{2n}) K^{-1}=\Id$ thus $S-S'=0.$
	The uniqueness of $P$ is shown in a similar way because $P$ and $P^2$ have the same eigenvectors.]
	
	We define $O=AP^{-1};$ we have $O\tr O=P^{-1}A\tr A P^{-1}=\Id$ thus $O$ is orthogonal.
	
	This decomposition is unique; indeed, if $A=O'P',$ we have
	\begin{equation*}
		A\tr A =P'(O')\tr (O') P'=(P')^2
	\end{equation*}
	thus $P=P'$ and $O=O'.$
	
	The matrices $P$ and $O$ are symplectic.
	Indeed a matrix $B$ is symplectic iff $B\tr \Omega_0 B =\Omega_0$ thus iff $B=(\Omega_0)^{-1}(B^\tau)^{-1}\Omega_0$ and  we have
	\begin{eqnarray*}
		OP&=&(\Omega_0)^{-1}\bigl((OP)\tr\bigr)^{-1}\Omega_0=(\Omega_0)^{-1}(O\tr)^{-1}(P\tr)^{-1}\Omega_0\\
		&=&(\Omega_0)^{-1}(O\tr)^{-1}\Omega_0\, (\Omega_0)^{-1}(P\tr)^{-1}\Omega_0=O' P'
	\end{eqnarray*}
	with $O'= (\Omega_0)^{-1}(O\tr)^{-1}\Omega_0$ and $P'=(\Omega_0)^{-1}(P\tr)^{-1}\Omega_0$.  Thus, by uniqueness of the decomposition, $O=O'$ and $P=P'$.  Hence
	$P=(\Omega_0)^{-1}(P\tr)^{-1}\Omega_0$ and $P$ is symplectic, and similarly for $O.$
	
	We have $P=\exp \half S=(\Omega_0)^{-1}\exp(-\half S\tr) \Omega_0=\exp\bigl(-\half (\Omega_0)^{-1}S\tr \Omega_0\bigr)$ thus 
	$S=-\Omega_0^{-1} S\tr \Omega_0$ and $S\in \textrm{sp}(\R^{2n},\Omega_0).$
	
	We have associated to every matrix $A\in \Sp(\R^{2n},\Omega_0)$  a unique element $O\in U(n)$ and a unique real symmetric matrix $S$ $2n\times 2n$ belonging to the
	symplectic Lie algebra $\textrm{sp}(\R^{2n},\Omega_0)$ in such a way  that $A= O\exp(\half S).$
	
	Reciprocally, if $O\in U(n)$ and if $S$ is a real symmetric $2n\times 2n$ matrix  in the symplectic Lie algebra $\textrm{sp}(\R^{2n},\Omega_0),$ then $\exp tS$ belongs to
	$\Sp(\R^{2n},\Omega_0)$ for all $t$ and $A= O\exp(\half S)$ is in $\Sp(\R^{2n},\Omega_0).$
	
	Hence there is a  homeomorphism between $\Sp(\R^{2n},\Omega_0)$ and the product of $U(n)$ by the vector space of real symmetric $2n\times 2n$ matrices belonging to the symplectic
	Lie algebra $\textrm{sp}(\R^{2n},\Omega_0).$
\end{proof}
\begin{lemma}
	The group $U(n)$ is homeomorphic to the cartesian product of $S^1$ and of the group $SU(n)$ of unitary matrices of determinant $1.$
	
	The group $SU(n)$ is simply connected.
\end{lemma}
\begin{proof}
	The homeomorphism between $S^1\times SU(n)$ and $U(n)$ is given by:
	\begin{equation*}
		S^1\times SU(n)\rightarrow U(n) :~(e^{i\theta},U)\mapsto \textrm{diag} (e^{i\theta},1,\ldots,1)U.
	\end{equation*}
	Its inverse associates to $O\in U(n)$ the element of $S^1$ defined by its determinant $e^{i\theta}=\det O$ and the element $U= \textrm{diag} (e^{-i\theta},1,\ldots,1)O$ which is unitary and of
	determinant $1.$
	
	To show that $SU(n)$ is simply connected, we use the action of $SU(n)$ on the sphere $S^{2n-1} = \{ \ z \in \C^n\ \vert \ \vert z\vert^2=1\ \}$.
	This action is clearly transitive and the isotropy group at $(1,0,\ldots,0)$ is isomorphic to $SU(n-1)$. This yields a fibration
	\begin{equation*}
		\raisebox{.2ex}{$SU(n)$}/\raisebox{-.2ex}{$SU(n-1)$}=S^{2n-1}.
	\end{equation*}
	The  long exact sequence in homotopy corresponding to this fibration gives
	\begin{equation*}
		\ldots\rightarrow\pi_2(S^{2n-1})\rightarrow \pi_1 \bigl( SU(n-1) \bigr) \rightarrow \pi_1 \bigl( SU(n) \bigr) \rightarrow \pi_1(S^{2n-1}).
	\end{equation*}
	If $n\ge 2,$ we have $\pi_2(S^{2n-1})=0$ and $\pi_1(S^{2n-1})=0$ thus
	\begin{equation*}
		\pi_1 \bigl(SU(n) \bigr) \simeq \pi_1 \bigl( SU(n-1) \bigr) \simeq \ldots \simeq \pi_1\bigl( SU(1) \bigr) = \pi_1\bigl( \{1\} \bigr)=0.
	\end{equation*}
\end{proof}
\begin{cor}
	The symplectic group $\Sp(\R^{2n},\Omega_0)$ is homeomorphic to the topological product of the circle $S^1,$ of the group $SU(n)$ and of the vector space of real symplectic
	$2n\times 2n$ matrices  in the symplectic Lie algebra $\textrm{sp}(\R^{2n},\Omega_0).$
	The fundamental group of $\Sp(\R^{2n},\Omega_0)$ is isomorphic to $\Z.$
	Every continuous map from $\Sp(\R^{2n},\Omega_0)$ to $S^1$ which coincides  with the determinant on $U(n)$ induces an isomorphism of the fundamental groups.
\end{cor}
Those results follow directly from the two previous lemmas and from the fact that $\pi_1(S^1)=\Z.$
We shall construct in the next section a continuous map $\rho$  from $\Sp(\R^{2n},\Omega_0)$ to $S^1$ which coincides with the determinant on $U(n)$ (inducing thus an isomorphism of the
fundamental groups).

\section{The  rotation map $\rho$ }
\begin{theorem}[\cite{SalamonZehnder}]\label{existingrho}
	There exists a unique family of continuous maps  
	\begin{equation*}
		\rho : \Sp(\R^{2n},\Omega_0) \rightarrow \textrm{S}^1
	\end{equation*}
	(one for each integer $n\ge 1$) with the following properties:
	\begin{enumerate}
		\item\label{rho1}[determinant]   $\rho$ coincides with  $\det_\C$ on the unitary subgroup 
			\begin{equation*}
				\rho(A) = {\det}_\C A\textrm{ if } A \in \textrm{Sp}(2n) \cap \textrm{O}(2n) = \textrm{U}(n);
			\end{equation*}
		\item\label{rho2}[invariance] $\rho$ is invariant under conjugation :
			\begin{equation*}
				\rho(kAk^{-1})=\rho(A)\ \forall k \in \textrm{Sp}(\R^{2n},\Omega_0);
			\end{equation*}
		\item\label{rho3}[normalisation] $\rho(A)=\pm1$ for matrices which  have no eigenvalue on the unit circle;
		\item\label{rho4}[multiplicativity] $\rho$ behaves multiplicatively with respect to direct sums : if   $(\R^{2n},\Omega_0) = (\R^{2m},\Omega_0)\oplus(\R^{2(n-m)},\Omega_0) ,$ and if
			$A=
			\left(\begin{array}{c|c}
				A' & 0 \\
				\hline 0 & A''
			\end{array}\right)$
			with $A'\in \Sp(\R^{2m},\Omega_0)$ and $A''\in \Sp(\R^{2(n-m)},\Omega_0),$ then 
			\begin{equation*}
				\rho(A)=\rho(A')\rho(A'').
			\end{equation*}
	\end{enumerate}
\end{theorem}
\subsection{Construction of  $\rho$ in dimension $2$}

We use as before the identification $\R^2 \cong \C;$ the matrix of the rotation in  $\R^2$ by an angle $\varphi$ is identified with the multiplication by  $e^{i\varphi}$ in $\C$ :
\begin{equation*}
	\left(\begin{array}{cc}
		\cos \varphi & -\sin \varphi  \\
		\sin \varphi & \cos \varphi
	\end{array}\right) 
	\leftrightarrow e^{i\varphi} \textrm{~ thus~} \rho \Biggl(
	\left(\begin{array}{cc}
		\cos \varphi & -	\sin \varphi \\
		\sin \varphi & \cos \varphi
	\end{array}\right)
	\Biggr)=e^{i\varphi}.
\end{equation*}
This gives in particular: $\rho\left(\Id\right)=1$ and  $\rho\left(-\Id\right)= -1 .$\\

In dimension $2$ we have
\begin{equation*}
	\Sp(\R^2,\Omega_0)=\textrm{Sl}(2,\R) = \left\{ 
	\left(\begin{array}{cc}
		a+d & -b+c \\
		b+c & a-d
	\end{array}\right)
	\vert \ a^2+b^2-c^2-d^2 = 1\right\}.
\end{equation*}
The eigenvalues of $A \in \Sp(\R^2,\Omega_0)$ are the roots of $\det (A-\lambda \Id),$ thus the solutions of $\lambda^2 -2a\lambda+1=0;$ so the eigenvalues are $\lambda = a \pm \sqrt{a^2-1}.$
There are three possible cases:
\begin{enumerate}
	\item{$a^2 > 1$\ :\ } In this case the two eigenvalues are real and distinct (their product is equal to $1$ but they differ from $1$ and $-1);$ so $\exists k \in \textrm{Sl}(2,\R)$ such that
		$kAk^{-1} = 
		\left(\begin{array}{cc}
			\lambda_1 & 0 \\
			0 & \frac{1}{\lambda_1}
		\end{array}\right).$
		By the invariance condition $\rho(A)=\rho (kAk^{-1}).$ By continuity of $\rho$ and by the normalisation condition, since we have seen that  $\rho\left(\Id\right)=1$ and  $\rho\left(-\Id\right)= -1$, we have
		\begin{equation*} 
			\rho(A)=1 ~\textrm{~if~} a>1 \textrm{~and~}  \rho(A)=-1 \textrm{~if~} a<-1.
		\end{equation*}
	\item $a^2<1$\ :\ There are no real eigenvalues; in the complexified space  $\C^2,$ the two eigenvalues are complex conjugate and their product is $1;$
		denote them by $~e^{i\varphi} , e^{-i\varphi}.$
		
		If $z=v-iw ~~ (z_1=v_1-iw_1,z_2=v_2-iw_2)$ is an eigenvector for $A$ in $\C^2$ of eigenvalue $e^{i\varphi},$ the vector $\overline{z}=v+iw=(\overline{z_1}, \overline{z_2})$ is an
		eigenvector of eigenvalue $e^{-i\varphi}:$\\
		$\left(\begin{array}{cc}
			a' & b' \\
			c' & d'
		\end{array}\right)
		\left(\begin{array}{c}
			z_1 \\
			z_2
		\end{array}\right)
		=
		\left(\begin{array}{c}
			e^{i\varphi}z_1 \\
			e^{i\varphi}z_2
		\end{array}\right)
		\ ,\
		\left(\begin{array}{cc}
			a' & b' \\
			c' & d'
		\end{array}\right)
		\left(\begin{array}{c}
			\overline{z_1} \\
			\overline{z_2}
		\end{array}\right)
		=
		\left(\begin{array}{c}
			e^{-i\varphi}\overline{z_1}\\
			e^{-i	\varphi}\overline{z_2}
		\end{array}\right).$
		Hence we have
		\begin{eqnarray*}
			Av & = & A \bigl( \thalf (z+\overline{z}) \bigr)\\
			& = & \thalf (\cos \varphi + i\sin \varphi ) (v-iw) + \thalf(\cos\varphi-i\sin\varphi)(v+iw)\\
			& = & \cos\varphi v  + \sin\varphi w\\
			Aw & = & A \bigl( \tfrac{i}{2} (z - \overline{z} ) \bigr)\\
			&=&\tfrac{i}{2}(\cos\varphi +i\sin\varphi)(v-iw)-\tfrac{i}{2}(\cos \varphi - i\sin \varphi)(v+iw)\\
 			&=& -\sin\varphi v + \cos \varphi w.
 		\end{eqnarray*}
		so that in the basis  $\{v,w\}$ of $\R^2$ the matrix associated to $A$ is  given by 
		\begin{equation*}
			\left(\begin{array}{cc}
				\cos\varphi & -\sin\varphi \\
				\sin\varphi & \cos\varphi
			\end{array}\right).
		\end{equation*}
		\begin{remarque}
			Denote by $\lambda$ ( and call Krein positive) the eigenvalue $e^{i\varphi}$ or $e^{-i\varphi}$  for which $\Omega_0(v_\lambda,w_\lambda)>0$ if
			$z_\lambda=v_\lambda-iw_\lambda$ is an eigenvector of $A$ in  $\C^2$ of eigenvalue $\lambda.$
			Another eigenvector corresponding to the same eigenvalue $\lambda$ has the form
			$z'=\alpha z_\lambda = (\tilde{a}+i\tilde{b})(v_\lambda-iw_\lambda) = (\tilde{a}v_\lambda+\tilde{b}w_\lambda)-i(-\tilde{b}v_\lambda+\tilde{a}w_\lambda)$ so that
			$\Omega_0(\tilde{a}v_\lambda+\tilde{b}w_\lambda,-\tilde{b}v_\lambda+\tilde{a}w_\lambda)=(\tilde{a}^2+\tilde{b}^2)\Omega_0(v_\lambda,w_\lambda)$ and the sign of
			$\Omega_0(v,w)$ does not depend on the choice of the eigenvector $z=v-iw.$
			On the other hand  $\overline{z}=v-(-iw)$ is an eigenvector corresponding to the conjugate eigenvalue and will correspond clearly to the opposite sign.
			Remark also that $\Omega_0(z,\overline{z})= 2i \Omega_0(v,w)$ so that $\Imc\Omega_0(z_\lambda,\overline{z_\lambda})>0$.
		\end{remarque}
		
		In $\R^{2}$, $\Omega_0$ is 
		 the matrix $
		\left( \begin{array}{cc}
			0 & 1 \\
			-1& 0
		\end{array} \right) .$ We have $\det\left(\begin{array}{cc}
				v_1 & w_1 \\
				v_2 & w_2
			\end{array}\right)=\Omega(v,w)$.
		Defining
		\begin{equation*}
			k=\frac{1}{\sqrt{\Omega_0(v_\lambda,w_\lambda)}}
			\left(\begin{array}{cc}
				v_1 & w_1 \\
				v_2 & w_2
			\end{array}\right) ,
			~~\textrm{where } (v_1,v_2) + i(w_1,w_2) = v_{\lambda} + iw_{\lambda} = z_{\lambda},
		\end{equation*}
		the matrix $k$ is of determinant $1$, thus $k\in \Sp(\R^2,\Omega_0),$
		and we have 
		$ \ k
		\left(\begin{array}{c}
			1 \\
			0
		\end{array}\right)
		=\frac{1}{\sqrt{\Omega_0(v_\lambda,w_\lambda)}}v $ and $k
		\left(\begin{array}{c}
			0 \\
			1
		\end{array}\right)
		=\frac{1}{\sqrt{\Omega_0(v_\lambda,w_\lambda)}}w$ so that
		\begin{equation*}
			(k^{-1}Ak) =
			\left(\begin{array}{cc}
				\cos\varphi & -\sin\varphi \\
				\sin\varphi & \cos\varphi
			\end{array}\right).
		\end{equation*}
		Thus, if $a^2<1$ and if $e^{i\varphi}$ is an eigenvalue such that $Az = e^{i\varphi}z$ where $z=v-iw$ with $\Omega_0\left(v,w\right) >0,$ we get,
		using the invariance and the determinant conditions for  $\rho$,  
		\begin{equation*}
			\rho(A) = e^{i\varphi}.
		\end{equation*}	
	\item  $a^2=1 : $
		We already know that $\rho(\Id) = 1$ and $\rho(-\Id) = -1.$
		
		Let $v$ be an eigenvector of $A$ of eigenvalue $a$ and suppose that there doesn't exist another eigenvector linearly independent from $v.$
		For any $w$ linearly independent from $v,$ we have
		\begin{equation*}
			\left\{\begin{array}{ccc}
				Av & = & av \\
				Aw & = & \tilde{c}v+aw.
			\end{array}\right. 
		\end{equation*}
		Take $w$ such that $\Omega_0(v,w)=1.$
		Then $k =
		\left(\begin{array}{cc}
			v_1 & w_1 \\
			v_2 & w_2
		\end{array}\right)
		\in \Sp(\R^2,\Omega_0)$ and
		\begin{equation*}
			k^{-1}Ak =
			\left(\begin{array}{cc}
				a & \tilde{c} \\
				0 & a
			\end{array}\right).
		\end{equation*}
		This matrix is the limit, for $t$ tending to $0$,
		of the path of sympletic matrices $\left(\begin{array}{cc}
				e^ta & e^t\tilde{c} \\
				0 &e^{-t} a
			\end{array}\right)$. For $t\neq 0$, these matrices have two distinct real eigenvalues equal to
			$e^ta$ and $e^{-t} a$; so their image under $\rho$ is equal to $a$ (cf case 1).
			By invariance and continuity of $\rho,$ we have
		\begin{equation*}
			 \rho(A) = a.
		\end{equation*}
\end{enumerate}
To summarize:
\begin{proposition}
	Conditions (\ref{rho1}), (\ref{rho2}) et (\ref{rho3}) determine a unique  continuous map
	\begin{equation*}
		\rho : \textrm{Sl}(2,\R)  \rightarrow  \textrm{S}^1: ~ A  \mapsto  \rho(A)
	\end{equation*}
	defined by
	 \begin{equation*}\rho(A)=
	 \left\{ \begin{array}{rl}
 		1 & \textrm{if } \half \operatorname{Tr}(A) =: a\geqslant 1 \\
      		& \quad \textrm{ i.e if the eigenvalues of } A \textrm{ are real positive};\\ 
 		-1 & \textrm{if } a\leqslant -1\\
  		& \quad  \textrm{ i.e if the eigenvalues of } A \textrm{ are real negative};\\ 
 		e^{i\varphi} & \textrm{if } e^{i\varphi} \textrm{ is an eigenvalue of } A \textrm{ such that } \Omega_0(v,w)>0 \\
		&~~~\textrm{~when~} z=v-iw \textrm{ is a corresponding eigenvector}.
	\end{array}\right .
	\end{equation*}
	\end{proposition}
The condition (\ref{rho4}) is of course empty in dimension $2.$

\subsection{Construction of $\rho$ in any dimension.}
We extend  $\Omega_0~\C$-linearly to  $V\otimes_\R\C$ and represent it by the same matrix.
We also extend any $A\in \Sp(V=\R^{2n},\Omega_0)~\C$-linearly to  $V\otimes_\R\C.$ 
If $v_\lambda$ denotes an eigenvector of $A\in \Sp(V=\R^{2n},\Omega_0)$ in $V\otimes_\R\C$ of eigenvalue $\lambda$ then
$\Omega_0(Av_\lambda,Av_\mu)=\Omega_0(\lambda v_\lambda,\mu v_\mu)=\lambda\mu \Omega_0(v_\lambda,v_\mu),$ thus $\Omega_0(v_\lambda,v_\mu)=0$ unless $\mu=\frac{1}{\lambda}.$
Hence the eigenvalues of $A$ arise in ``quadruples''
\begin{equation}
	[\lambda]:=\left\{\lambda,\frac{1}{\lambda},\overline{\lambda},\frac{1}{\overline\lambda}\right\}.
\end{equation}
\subsubsection{$\rho(A)$ for a semisimple element $A$}
\begin{definition}
An element $A$ is semisimple if $V\otimes_\R\C$ is the direct sum of its eigenspaces.
\end{definition}
Denote by $E_\lambda$ the eigenspace corresponding to the eigenvalue $\lambda$ in $V\otimes_\R\C.$
Remark that  if $v= u+iu'$ is in $E_\lambda$ with $u$ and $u'$ in $V$ then
$\overline{v}:=u-iu'$ is in $E_{\overline{\lambda}}$ so that  $E_\lambda \oplus E_{\overline{\lambda}}$ is the complexification of a real subspace of $V$.
The space
\begin{equation}
	W_{[\lambda]}:=E_\lambda\oplus E_{\frac{1}{\lambda}}\oplus E_{\overline\lambda}\oplus E_{\frac{1}{\overline\lambda}}
\end{equation}
is the complexification  of a real  symplectic subspace $V_{[\lambda]}$ and
\begin{equation}
	\R^{2n}=V_{[\lambda_1]}\oplus V_{[\lambda_2]}\oplus\ldots\oplus V _{[\lambda_K]}
\end{equation}
where the direct sum is symplectic orthogonal and where $[\lambda_1],\ldots [\lambda_K]$ are the distinct quadruples exhausting the eigenvalues of $A$.
Hence, by multiplicativity of $\rho$ 
\begin{equation}
	\rho(A)=\rho(A_{[\lambda_1]})\cdot\rho(A_{[\lambda_2]})\cdot\ldots \cdot \rho(A_{[\lambda_K]})
\end{equation}
where $A_{[\lambda_i]}$ is the restriction of $A$ to $V_{[\lambda_i]}.$
\begin{itemize}
	\item	If $\lambda= \pm 1,$ we have  $A_{[\lambda]}=\pm\Id;$ thus by the  determinant condition, we have
		\begin{equation}\label{lambda+-1}
			\rho(A_{[\lambda]})=\rho(\pm\Id)=
			\left\{\begin{array}{ll}
				1&~\textrm {if } \lambda=1\\
				-1^{\half\dim  V_{[-1]}}&~\textrm {if } \lambda=-1
			\end{array}\right\}
			=\left(\frac{\lambda}{\vert \lambda \vert}\right)^{\half\dim V_{[\lambda]}}.
		\end{equation}
	\item If $\lambda\in \R\setminus\{ \pm 1\},$  no eigenvalue is on $S^1;$ by the normalisation condition and continuity, we have
		\begin{equation}\label{lambdainR}
			\rho(A_{[\lambda]})=
			\left\{\begin{array}{ll}
				1&~\textrm {if } \lambda>0\\
				-1^{\half\dim  V_{[\lambda]}}&~\textrm {if } \lambda<0
			\end{array}\right\}
			=\left(\frac{\lambda}{\vert \lambda \vert}\right)^{\half\dim V_{[\lambda]}}.
		\end{equation}
	\item If $\lambda\in \C\setminus S^1\cup \R,$ no eigenvalue is on $S^1 ;$ by the normalisation condition and continuity, we have
		\begin{equation}\label{lambdainC}
			\rho(A_{[\lambda]})=1.
		\end{equation}
		Indeed, we can  bring continuously $\lambda,\frac{1}{\lambda},\overline{\lambda},\frac{1}{\overline\lambda}$  simultaneously all to $1$ (or all to $-1)$ and in both cases
		$\rho(A_{[\lambda]})=1$ because $\dim V_{[\lambda]} = 4\dim_{\C}E_{\lambda} .$
	\item If $\lambda=e^{i\varphi}\in  S^1\setminus \{\pm 1\},  W_{[\lambda]}:=E_\lambda\oplus  E_{\overline\lambda};$ we define
		\begin{equation}
			Q: E_\lambda \times E_\lambda \rightarrow \R :~(z,z') \mapsto Q(z,z'):=\Imc \Omega_0(z,\overline{z'});
		\end{equation} 
		where $\Imc a$ denotes the imaginary part of a complex number $a .$
		It is a nondegenerate symmetric $2$-form on the vector space $E_\lambda$ viewed as a real vector space.
		It is indeed symmetric because
		\begin{equation*}
			\Imc \Omega_0(z,\overline{z'}) = -\Imc \Omega_0(\overline{z'},z) = \Imc \overline{\Omega_0(\overline{z'},z)}=\Imc \Omega_0(z',\overline{z})
		\end{equation*}
		and it is nondegenerate because 
		\begin{equation*}
			Q(z,z') = \Imc \Omega_0(z,\overline{z'}) = \frac{1}{2i}\bigl(\Omega_0(z,\overline{z'}) - \Omega_0(\overline{z},z')\bigr)
		\end{equation*}
		so that $Q(z,z') =0 ~\forall z \in E_{\lambda}$ iff	 $\Omega_0(z,\overline{z'}) - \Omega_0(\overline{z},z') = 0 \ \forall z \in E_{\lambda}.$
		Replacing $z$ by $iz$ this implies $i\Omega_0(z,\overline{z'}) + i\Omega_0(\overline{z},z') = 0 \ \forall z \in E_{\lambda}$ 
		hence $\Omega_0(z,\overline{z'}) = 0 \ \forall z \in E_{\lambda}$ and this implies $ \overline{z'}=0 .$
		
		We can thus find a vector $z_1\in E_\lambda$ such that $Q(z_1,z_1)=2a_1\ne 0$.
		Writing $z_1=u_1-iv_1,$ the subspace of $V_{[\lambda]}$ generated by $u_1$ and $v_1$ is
		symplectic and stable by $A ;$ we have $\Omega_0 (u_1,v_1)=a_1$ and the restriction $A_1$ of $A$ to this subspace in the basis $\{ u_1,v_1\}$ has the form
		\begin{equation*}
			A_1=
			\left(\begin{array}{cc}
				\cos\varphi & -\sin\varphi \\
				\sin\varphi & \cos\varphi
			\end{array}\right).
		\end{equation*}
		If $a_1>0,$ we have as before (cf the case where $a^2 < 1$), $\rho(A_1)=e^{i\varphi}.$
		If $a_1<0,$ we permute the vectors $u_1$ and $v_1$ and we have $\rho(A_1)=e^{-i\varphi}.$
		Indeed $v_1 - iu_1 = -i(u_1+iv_1) = -i\overline{(u_1-iv_1)}$ which is of eigenvalue $\overline{\lambda} = e^{-i\varphi} .$
		In conclusion as $V_{[\lambda]}$ is the direct sum of $\langle u_1,v_1\rangle$ and its symplectic orthogonal, we have
		\begin{equation}\label{a1>0}
			\rho(A_{[\lambda]})=e^{\frac{i}{2}\varphi \sign(Q)}
		\end{equation}
		where $\sign(Q)$ denotes the signature of $Q$ (the number of positive eigenvalues minus the number of negative eigenvalues of $Q$).
		The $\half$ factor comes from the fact that $Q(iz,iz')=Q(z,z')$ and the vectors $z$ and $iz$ define the same real vector space $\textrm{Span}\{u_1,v_1\}$ of $V_{[\lambda]}.$
\end{itemize}
The map $\rho$ is  continuous as can be viewed by considering  possible variations of the ``quadruples'' of eigenvalues; for example, in a continuous variation of the matrix $A$, a quadruple of eigenvalues corresponding to
$\lambda\notin S^1\cup\R$ can only degenerate  into a pair of real eigenvalues counted twice  or  into a pair of eigenvalues on the circle counted twice but with opposite signs.
In this last case, $[\lambda] = \left\{ \lambda, \frac{1}{\lambda},\overline{\lambda},\frac{1}{\overline{\lambda}} \right\}$ degenerates into $\{ e^{i\varphi},e^{-i\varphi},e^{-i\varphi},e^{i\varphi}\}$ and
the image by $\rho$ of the corresponding matrix is $1.$
Indeed, if $\lambda=re^{i\varphi}\notin S^1\cup \R$ is an eigenvalue and if $z=u-iv$ is an eigenvector of $A$ of eigenvalue $\lambda$ and $z'=u'+iv'$ an eigenvector of eigenvalue $1/\lambda$
such that $\Omega_0(z,z')=2,$ then  $ \Omega_0(z,\overline{z'})=0; \Omega_0(z,\overline{z})=0; \Omega_0(\overline{z},z')=0; \Omega_0(\overline{z},\overline{z'})=2;
\Omega_0(z',\overline{z'})=0,$ 
 so that the $2$-form $\Omega$ and the matrix $A$ are written, in the basis $\{ u,v,u',v'\},$
\begin{equation*}
	\Omega_0=
	\left(\begin{array}{cccc}
		0&0&1&0\\
		0&0&0&1\\
		-1&0&0&0\\
		0&-1&0&0
	\end{array}\right)
	~A=
	\left(\begin{array}{cccc}
		r\cos\varphi&-r\sin\varphi&0&0\\
		r\sin\varphi&r\cos\varphi&0&0\\
		0&0&\frac{1}{r}\cos\varphi&-\frac{1}{r}\sin\varphi\\
		0&0&\frac{1}{r}\sin \varphi&\frac{1}{r}\cos\varphi
	\end{array}\right).
\end{equation*}
and the limit when $r$ equals $1$ give $\rho(A) = e^{i\varphi}e^{-i\varphi}=1$ because in the basis $\{ z, \overline{z'},iz, i\overline{z'}\}$ of $E_{e^{i\varphi}}$ the matrix of $Q$ is
$\left( \begin{array}{cccc}
	0&0&0&-2\\
	0&0&2&0\\
	0&2&0&0\\
	-2&0&0&0
\end{array}\right) .$
\subsubsection{$\rho(A)$ for any $A$}

\begin{lemma}\label{lem:orthokerim}
	Consider $A\in\Sp(V,\Omega)$ and let $\lambda$ be an eigenvalue of $A$ in $V\otimes_\R\C.$
	Then $\Ker(A-\lambda\Id)^j$ is the symplectic orthogonal of $\im(A- \tfrac{1}{\lambda}\Id)^j.$
\end{lemma}
\begin{proof}
	\begin{eqnarray*}
		\Omega \bigl( (A - \lambda \Id ) u, Av \bigr) & = & \Omega(Au,Av)-\lambda\Omega(u,Av)
		=\Omega(u,v)-\lambda\Omega(u,Av)\\
		&=&-\lambda\Omega \Bigl( u, \bigl(A-\tfrac{1}{\lambda} \Id \bigr) v \Bigr)
	\end{eqnarray*}
	and by induction
	\begin{equation}\label{eq:omegaA}
		\Omega \bigl( (A-\lambda\Id)^ju, A^jv \bigr) = (-\lambda)^j\Omega \Bigl( u, \bigl( A-\tfrac{1}{\lambda} \Id \bigr)^jv \Bigr).
	\end{equation}
	The result follows from the fact that $A$ is invertible.
\end{proof}
\begin{cor}\label{lem:orthovalpro}
	If $E_\lambda$ denotes the generalized eigenspace of eigenvalue $\lambda,$ i.e $E_\lambda := \bigl\{ v\in V\otimes_\R\C \ \vert \ (A-\lambda\Id)^jv=0 \textrm{ for an integer } j >0 \bigr\}$, we have
	\begin{equation*}
		\Omega(E_\lambda,E_\mu)=0 \quad \textrm{ when }~\lambda\mu\ne 1 .
	\end{equation*}
\end{cor}
Indeed the symplectic orthogonal of $E_\lambda=\cup_j \Ker (A-\lambda\Id)^j$ is the intersection of the $\im(A-\tfrac{1}{\lambda}\Id)^j.$ By Jordan normal form, this intersection is
the sum of the generalized eigenspaces corresponding to the eigenvalues which are not  $\frac{1}{\lambda}.$\\

\noindent Lemma \ref{lem:orthokerim}  will be used in the next section to establish  normal forms of symplectic matrices.
Its corollary is well-known and usually proven by induction.\\

Recall that $\dim_\C E_\lambda=m_\lambda$ is the algebraic multiplicity of $\lambda$, that is the exponent of $t-\lambda$ in the factorisation of $\det(A-t\Id)$.

Note that if $v= u+iu'$ is in $\Ker (A-\lambda\Id)^j$ with $u$ and $u'$ in $V$ then
$\overline{v}=u-iu'$ is in $\Ker (A-\overline{\lambda}\Id)^j$ so that  $E_\lambda \oplus E_{\overline{\lambda}}$ is the complexification of a real subspace of $V$.
From this remark and lemma \ref{lem:orthovalpro} the space
\begin{equation}
	W_{[\lambda]}:=E_\lambda\oplus E_{\frac{1}{\lambda}}\oplus E_{\overline\lambda}\oplus E_{\frac{1}{\overline\lambda}}
\end{equation}
is again the complexification of a real and symplectic subspace $V_{[\lambda]}$ and
\begin{equation}
	\R^{2n}=V_{[\lambda_1]}\oplus V_{[\lambda_2]}\oplus\ldots\oplus V _{[\lambda_K]}.
\end{equation}
We have again
\begin{equation}
	\rho(A)=\rho(A_{[\lambda_1]})\cdot\rho(A_{[\lambda_2]})\cdot\ldots \cdot \rho(A_{[\lambda_K]})
\end{equation}
where $A_{[\lambda_i]}$ is the restriction of $A$ to $V_{[\lambda_i]}.$

Every symplectic matrix can be approached as closely as we want by a semisimple symplectic matrix.
To be complete we give a proof of this property in the next section.
By the continuity hypothesis, the map $\rho$ is thus necessarilly defined as follows:
\begin{theorem}[\cite{SalamonZehnder, Audin}]
	Let $A\in \Sp(\R^{2n},\Omega)$.
	We consider the eigenvalues $\{ \lambda_i\}$ of $A.$
	For an eigenvalue $\lambda=e^{i\varphi}\in S^1\setminus\{\pm 1\},$ we consider the number $m^+(\lambda)$ of positive eigenvalues of the symmetric  non degenerate $2$-form $Q$
	defined on the generalized eigenspace $E_\lambda$ by
	\begin{equation*}
		Q: E_\lambda \times E_\lambda \rightarrow \R :~ (z,z') \mapsto Q(z,z'):=\Imc \Omega_0(z,\overline{z'}).
	\end{equation*}
	Then
	\begin{equation}\label{eqrho}
		\rho(A)= (-1)^{\half m^-} \prod_{\lambda\in S^1\setminus\{\pm1\}}\lambda^{\half m^+(\lambda)}
	\end{equation}
	where $m^-$ is the sum of the algebraic multiplicities $m_\lambda=\dim_\C E_\lambda$ of the real negative eigenvalues.
\end{theorem}
\begin{proof}
	We have seen that it is necessarilly of that form.
	The first term in formula \eqref{eqrho}, $(-1)^{\half m^-}$, corresponds to the conditions \eqref{lambda+-1}, \eqref{lambdainR} and \eqref{lambdainC}.
	The second term comes from the fact that if $\lambda = e^{i\varphi}$  and of $A_{[\lambda]}$ is the restriction of $A$ to $V_{[\lambda]}$ then $\rho(A_{[\lambda]}) = e^{\half i\varphi\sign Q} .$
	In formula \eqref{eqrho}, we count one term for $\lambda$ and one term for $\overline{\lambda}$.
	Remark that $\Imc \Omega_0(z,\overline{z'})=-\Imc \Omega_0(\overline{z},{z'})$ so that the number of negative eigenvalues of $Q$ on $E_\lambda \times E_\lambda$
	is equal to the number of positive eigenvalues of $Q$ on $E_{\overline{\lambda}}\times E_{\overline{\lambda}}.$ 
	Thus the signature of $Q$ on $E_\lambda \times E_\lambda$ is equal to $m^+(\lambda)-m^+(\overline{\lambda})$.
	Hence $e^{\half i\varphi\sign Q} =\lambda^{\half m^+(\lambda)}{\overline{\lambda}}^{\half m^+(\overline{\lambda})}$.
	
	This map $\rho$ satisfies clearly the hypotheses of continuity, invariance, normalization and multiplicativity.
	It is an extension of the map determinant on $U(n)$ because every element of $U(n)$ is semisimple.
\end{proof}

\begin{definition}
	Let $\lambda$ be an eigenvalue of $A \in \Sp(\R^{2n},\Omega_0)$ with $\vert \lambda \vert = 1$.
	When $\lambda =1 (\textrm{ or } -1)$, the corresponding generalized eigenspace is of even dimension $2m_1$ (or $2m_{-1})$ and we count $m_1$ times the eigenvalue $1$
	(or $m_{-1}$ times the eigenvalue $-1$) as Krein positive.
	For the pair of eigenvalues $\lambda = e^{i\varphi}$ et $\bar{\lambda} = e^{-i\varphi}$, when the quadratic form
	\begin{equation*}
		Q:E_{\lambda}\times E_{\lambda} \longrightarrow \C :~ (v,w) \longmapsto \Imag\Omega(v,\bar{w})
	\end{equation*}
	is of signature $(2r,2s)$ we count $r$ times the eigenvalue $e^{i\varphi}$ and $s$ times the eigenvalue $e^{-i\varphi}$ as Krein positive.
\end{definition}
\begin{remarque}[\cite{SalamonZehnder}]
	Another expression of $\rho$ on a matrix $A \in \Sp(\R^{2n},\Omega_0)$ is obtained as follows: we consider the eigenvalues of $A,$ $\lambda_1,\ldots,\lambda_{2n},$ repeated
	accordingly to their algebraic multiplicities.
	We say that an eigenvalue $\lambda_i$ is of the first kind if $\vert \lambda_i\vert<1$ or if $\vert \lambda_i\vert =1$ and if it is ``positive in the sense of Krein''.
	Then
	\begin{equation}\label{eq:rhopremespece}
		\rho(A) = \prod_{\lambda_i \textrm{ of first kind}} \frac{\lambda_i}{\vert \lambda_i\vert}.
	\end{equation}
\end{remarque}
\begin{proposition}\label{prop:rhocarre}
	The map $\rho$ is not a group homomorphism but we always have
	\begin{equation}
		\rho(A^N) = \bigl( \rho(A) \bigr)^N \quad \quad \forall A\in \Sp(\R^{2n},\Omega_0) \textrm{ and }\forall N \in \Z .
	\end{equation}
\end{proposition}
This results directly from the construction of $\rho$ and from the fact that the eigenvalues of $A^N$ are equal to the $N$th powers of the eigenvalues of $A.$

\section{Normal forms and density of  semisimple elements in the symplectic group}
\begin{definition}
	A \emph{symplectic basis} of a symplectic vector space of dimension $2m$ is a basis $\{ e_1, \ldots ,e_{2m} \}$ in which the matrix associated to the symplectic form is $\Omega_0 = 
	\left(\begin{array}{cc}\
		0 & \Id\\
		-\Id & 0
	\end{array}\right)$.
\end{definition}
We consider the standard symplectic $2n-$dimensional real vector space $(\R^{2n},\Omega_0)$ and the group of its linear symplectic transformations
\[
	\Sp(\R^{2n},\Omega_0) = \left\{ \, A:\R^{2n} \rightarrow \R^{2n}\,\vert\, A \mbox{ linear and }\Omega_0(Au,Av)=\Omega_0(u,v) \, \forall u,v \,\right\}.
\]
We  show that semisimple symplectic matrices with distinct eigenvalues are dense in the set of all symplectic matrices:
we  approach an element $A\in \Sp(\R^{2n},\Omega_0)$ by symplectic matrices which are diagonalizable on $\C^{2n}$.
For this, we shall determine normal forms for symplectic matrices.

As we have seen before, if $\lambda$ is an eigenvalue of $A$ (on $\C^{2n}$) then so are $\overline\lambda, \frac{1}{\lambda}$ and $\frac{1}{\overline\lambda}$ and we  denote by
$\left[ \lambda\right]$ the set  $\{ \lambda, \overline\lambda, \frac{1}{\lambda},\frac{1}{\overline\lambda}\}$ and by $\left[\lambda_1\right], \ldots,\left[ \lambda_K\right] $
the distinct such sets exhausting the eigenvalues of $A$.

If $\lambda$ is an eigenvalue of $A$ and  $E_\lambda$ the corresponding generalized eigenspace, 
\begin{equation}
	W_{[\lambda]}:=E_\lambda\oplus E_{\frac{1}{\lambda}}\oplus E_{\overline\lambda}\oplus E_{\frac{1}{\overline\lambda}}
\end{equation}
is the complexification of a real symplectic subspace $V_{[\lambda]}\subset \R^{2n}$ and 
\begin{equation}
	\R^{2n}=V_{[\lambda_1]}\oplus V_{[\lambda_2]}\oplus\ldots\oplus V _{[\lambda_K]}.
\end{equation}
Since $A$ stabilizes each $V_{[\lambda_i]}$,  it is enough to prove the property for the restriction of $A$ to $V_{[\lambda]}$.

Let $(V,\Omega)$ be a symplectic vector space and let $A \in \Sp(V,\Omega)$.
We want to construct  a symplectic basis of $V$ in which $A$ has a ``simple'' form.
Assume that $V=V_1\oplus V_2$ where $V_1$ and $V_2$ are $\Omega$-orthogonal vector subspaces invariant under $A$.
Suppose that $\{e_1,\ldots,e_{2k}\}$ is a symplectic basis of $V_1$ in which the matrix associated to $A\vert_{V_1}$ is
$\left(\begin{array}{cc}
	A'_1 & A'_2\\
	A'_3 & A'_4
\end{array}\right)$.
Suppose also that $\{f_1,\ldots,f_{2l}\}$ is a symplectic basis of $V_2$ in which the matrix associated to $A\vert_{V_2}$ is
$\left(\begin{array}{cc}
	A''_1 & A''_2\\
	A''_3 & A''_4
\end{array}\right)$.
Then $\{e_1,\ldots,e_k,f_1,\ldots,f_l,e_{k+1},\ldots,e_{2k},f_{l+1},\ldots,f_{2l}\}$ is a symplectic basis of $V$ and the matrix associated to $A$ in this basis is
$$
	\left(\begin{array}{cccc}
		A'_1 & 0 & A'_2 & 0\\
		0 & A''_1 & 0 & A''_2\\
		A'_3 & 0 & A'_4 & 0\\
		0 & A''_3 & 0 & A''_4
	\end{array}\right).
$$
The notation $A'\diamond A''$ is used in Long \cite{Long} for this matrix.  It is ``a direct sum of matrices with obvious identifications".
We call it the \emph{symplectic direct sum} of the matrices $A'$ and $A''$.\\

In general, one cannot find a symplectic basis of the complexified vector space for which the matrix associated to $A$ has Jordan normal form.
Normal forms for  symplectic matrices have been described in particular in \cite{Long}.
The presentation of  normal forms for symplectic matrices that we give here is short and elementary; it is based on the following lemmas.

\subsection{Two technical lemmas}\label{appendicesemisimple}
Let $(V,\Omega)$ be a real symplectic vector space. Consider $A\in\Sp(V,\Omega)$ and let $\lambda$ be an eigenvalue of $A$ in $V\otimes_\R\C.$
We have seen in lemma  \ref{lem:orthokerim}
that the space  $\Ker(A-\lambda\Id)^j$ is the symplectic orthogonal for $\Omega$ of $\im\bigl(A-\tfrac{1}{\lambda}\Id\bigr)^j,$ so that if $E_\lambda$ denotes the generalized eigenspace of eigenvalue $\lambda,$ we have
\begin{equation*}
	\Omega(E_\lambda,E_\mu)=0 \quad \textrm{ if }~\lambda\mu\ne 1 .
\end{equation*}
We had seen precisely that (equation \eqref{eq:omegaA})
\begin{equation*}
	\Omega \bigl( (A-\lambda\Id)^ju,A^jv \bigr) = (-\lambda)^j \Omega \Bigl( u, \bigl( A-\tfrac{1}{\lambda}\Id \bigr)^jv \Bigr).
\end{equation*}
	
Let $p\ge 0$ be the largest integer such  that the restriction to $E_\lambda$ of $(A-\lambda\Id)^{p}$ is not identically zero [and the restriction to  $E_\lambda$ of $(A-\lambda\Id)^{p+1}$ is zero]. Since $A$ is real,
this integer $p$ is the same for $\overline{\lambda}$.

By lemma \ref{lem:orthokerim}, $\Ker (A-\lambda\Id)^j$ is the symplectic orthogonal of $\im\bigl(A-\tfrac{1}{\lambda}\Id\bigr)^j$ for all $j,$ thus $\dim\Ker (A-\lambda\Id)^j=\dim \Ker \bigl(A-\tfrac{1}{\lambda}\Id\bigr)^j$;
hence the integer $p$ is the same for $\lambda$ and $\frac{1}{\lambda}$.
\begin{lemma}\label{lem:tech}
	For any positive integer $j$, the bilinear map
	\begin{equation*}
		\widetilde{Q} : \raisebox{.2ex}{$E_\lambda$}\,/\raisebox{-.2ex}{$\Ker(A-\lambda\Id)^j$}\times \raisebox{.2ex}{$E_{\frac{1}{\lambda}}$}/\raisebox{-.2ex}{$\Ker\bigl(A-{\tfrac{1}{\lambda}} \Id\bigr)^j$}\rightarrow \C 
	\end{equation*}
	\begin{equation}
		([v],[w])\mapsto  \widetilde{Q} ([v],[w]):=\Omega \bigl( (A-\lambda\Id)^jv,w \bigr) \qquad v\in E_\lambda, w\in E_{\frac{1}{\lambda}}
	\end{equation}
	is well defined and non degenerate.
	In the formula $[v]$ denotes the  class containing $v$ in the appropriate quotient.
\end{lemma}
\begin{proof}
 The fact that $\widetilde{Q} $ is well defined follows from equation \eqref{eq:omegaA};
 indeed, for any integer $j$, we have 
 \begin{equation}
	\Omega \bigl( (A-\lambda\Id)^j u,v \bigr) = (-\lambda)^j \Omega \Bigl( A^j u, \bigl( A-\tfrac{1}{\lambda}\Id \bigr)^jv \Bigr).
\end{equation}
The map is non degenerate because $\widetilde{Q} ([v],[w])=0 ~\forall w$ iff $(A-\lambda\Id)^j v=0$ since $\Omega$ is a non degenerate pairing between
$E_\lambda$ and $E_{\frac{1}{\lambda}}$ thus iff $[v]=0.$
	Similarly, $ \widetilde{Q} ([v],[w])=0 ~\forall v $ if and only if $w$ is $\Omega$-orthogonal to $\im(A-\lambda\Id)^j,$ thus iff
	$w\in \Ker(A-\tfrac{1}{\lambda}\Id)^j$ hence $[w]=0.$
\end{proof}
\begin{lemma}
For any $v,w\in V$, any $\lambda\in \C\setminus\{0\}$ and any integers $i\ge 0,~j>0$ we have:
\begin{eqnarray}
 \Omega \Bigl( (A-\lambda\Id)^{i}v,\bigl(A-\tfrac{1}{\lambda}\Id\bigr)^{j}w \Bigr)&=& -\frac{1}{\lambda}
  \Omega \Bigl( (A-\lambda\Id)^{i+1}v,\bigl(A-\tfrac{1}{\lambda}\Id\bigr)^{j}w \Bigr)\label{eq:Aij}\\
  &&-\frac{1}{\lambda^2}  \Omega \Bigl( (A-\lambda\Id)^{i+1}v,\bigl(A-\tfrac{1}{\lambda}\Id\bigr)^{j-1}w \Bigr).\nonumber
\end{eqnarray}
In particular, if $\lambda$ is an eigenvalue of $A$, if $p\ge 0$ is the largest integer such  that the restriction to $E_\lambda$ of $(A-\lambda\Id)^{p}$ is not identically zero  and if $v,w$ belong to $E_\lambda$, we have  for any integer $k\ge 0$:
	\begin{equation}\label{eq:p}
		\Omega \bigl( (A-\lambda\Id)^{p+k}v,w \bigr) = (-\lambda^2)^j \Omega \Bigl( (A-\lambda\Id)^{p+k-j}v, \bigl(A-{\tfrac{1}{\lambda}}\Id \bigr)^jw \Bigr)
	\end{equation}
so that
	\begin{equation}\label{eq:Qbiendef}
		\Omega \bigl( (A-\lambda\Id)^pv,w \bigr)=(-\lambda^2)^p \Omega \Bigl(v, \bigl( A - {\tfrac{1}{\lambda}}\Id \bigr)^pw \Bigr)
	\end{equation}
and
	\begin{equation}\label{eq:zero}
		\Omega \Bigl( (A-\lambda\Id)^{k}v, \bigl( A-{\tfrac{1}{\lambda}}\Id \bigr)^jw\Bigr)=0 ~\textrm{if }~k+j>p.
	\end{equation}
\end{lemma}	
\begin{proof}
We have:
\begin{eqnarray*}
	&&\Omega\Bigl((A-\lambda\Id)^{i} v,\bigl(A-\tfrac{1}{\lambda}\Id\bigr)^jw\Bigr)\\
	&&\ \ \ =-\frac{1}{\lambda}\Omega\Bigl(\bigl(A-{\lambda}\Id-A\bigr)(A-\lambda\Id)^{i} v,\bigl(A-\tfrac{1}{\lambda}\Id\bigr)^jw\Bigr)\\
	&&\ \ \ =-\frac{1}{\lambda}\Omega\Bigl((A-\lambda\Id)^{i+1} v,\bigl(A-\tfrac{1}{\lambda}\Id\bigr)^jw\Bigr)\\
	&&\ \ \ \quad+\frac{1}\lambda \Omega\Bigl(A(A-\lambda\Id)^{i} v,\bigl(A-\tfrac{1}{\lambda}\Id\bigr)\bigl(A-\tfrac{1}{\lambda}\Id\bigr)^{j-1}w\Bigr)\\
	&&\ \ \ =-\frac{1}{\lambda}\Omega\Bigl((A-\lambda\Id)^{i+1} v,\bigl(A-\tfrac{1}{\lambda}\Id\bigr)^jw\Bigr)\\
	&&\ \ \  \quad+\frac{1}{\lambda}\Omega\Bigl((A-\lambda\Id)^{i} v,\bigl(A-\tfrac{1}{\lambda}\Id\bigr)^{j-1}w\Bigr)\\
	&&\ \ \ \quad-\frac{1}{\lambda^2}\Omega\Bigl(A(A-\lambda\Id)^{i} v,\bigl(A-\tfrac{1}{\lambda}\Id\bigr)^{j-1}w\Bigr)
\end{eqnarray*}
and  formula \eqref{eq:Aij} follows.\\
Observe now that, for any integer $k\ge 0$ by (\ref{eq:omegaA}) and the definition of $p$
	\[
		(-\lambda)^{j+1} \Omega \Bigl( (A-\lambda\Id)^{p+k-j}v, \bigl(A-{\tfrac{1}{\lambda}}\Id \bigr)^{j+1} w \Bigr)=\Omega \bigl( (A-\lambda\Id)^{p+k+1}v,A^jw \bigr)=0.
	\]
	Hence, applying  formula \eqref{eq:Aij} with		
an induction on $j\le p+k$, we get formula \eqref{eq:p}. The other formulas follow readily.
	\end{proof}
\subsection{Normal forms  for  $A_{\vert V_{[\lambda]}}$}\label{subsecnormalforms}

We construct here a symplectic basis of $W_{[\lambda]}$ (and of $V_{[\lambda]}$) adapted to $A$ for a given eigenvalue $\lambda$ of $A$.
We shall decompose $W_{[\lambda]}$ (and $V_{[\lambda]}$) into a direct sum of symplectic subspaces stable by $A$.
On any subspace of  $W_{[\lambda]}$,  the only eigenvalues of $A$ are of the form $\lambda,1/\lambda,\overline{\lambda}$ and $1/\overline{\lambda}$. 
We shall assume as above that $(A-\lambda\Id)^{p+1}=0$ and $(A-\lambda\Id)^{p}\ne 0$ on the generalized eigenspace $E_\lambda;$ recall that we have the same integer $p$ for the other eigenvalues of the ``quadruple''.
We shall distinguish three cases; first $\lambda\notin S^1$ then $\lambda=\pm 1$ and
finally $\lambda\in S^1\setminus \pm 1.$

\subsubsection{Case 1:   $A_{\vert V_{[\lambda]}}$ for $\lambda\notin S^1 .$}

Choose an element  $v\in E_\lambda$ and an element $w\in E_{\frac{1}{\lambda}}$ such that 
$$
	\widetilde{Q} ([v],[w])=\Omega \bigl( (A-\lambda\Id)^pv,w \bigr)\ne 0.
	$$
Consider the smallest subspace $E^v_\lambda$ of $E_\lambda$ stable by $A$ and containing $v$; it is of dimension $p+1$ and is generated by 
\begin{equation*}
	\left\{e_1:=(A-\lambda\Id)^pv,\ldots, e_i:=(A-\lambda\Id)^{p+1-i}v,\ldots, e_{p+1}:=v\right\}.
\end{equation*}
Observe that $Ae_i=(A-\lambda\Id)e_i+\lambda e_i$ so that $Ae_i=\lambda e_i+ e_{i-1}$ for $i>1$ and $Ae_1=e_1$.\\
Similarly consider the smallest subspace $E^w_{\frac{1}{\lambda}}$ of $E_{\frac{1}{\lambda}}$ stable by $A$ and containing $w$; it is also of dimension $p+1$  and is generated by
\begin{equation*}
	\left\{e'_1:=w, \ldots, e'_{j}:=\left(A-\tfrac{1}{\lambda}\Id\right)^{j-1}w,\ldots e'_{p+1}:=\left(A-\tfrac{1}{\lambda}\Id\right)^pw\right\} .
\end{equation*}
One has
\begin{itemize}
	\item $\Omega(e_i,e_j)=0$ and $\Omega(e'_i,e'_j)=0$ because $\Omega(E_\lambda,E_\mu)=0$ if $\lambda\mu\ne 1$
	\item $\Omega(e_i,e'_j)=0$ if $i<j$  because the equation (\ref{eq:zero}) implies that
		\begin{equation*}
			\Omega \Bigl( (A-\lambda\Id)^{p+1-i}v, \bigl( A-{\tfrac{1}{\lambda}}\Id \bigr)^{j-1}w \Bigr)=0 \textrm{ if } p+j-i>p
		\end{equation*}
	\item $\Omega(e_i,e'_i)=\bigl(\frac{-1}{\lambda^2}\bigr)^{i-1}\Omega \Bigl(\bigl( A-{\frac{1}{\lambda}}\Id\bigr)^pv,w\Bigr)$ by equation (\ref{eq:p}) and is non zero by the choice of $v,w .$
\end{itemize}
The matrix associated to $\Omega$	in the basis  $\{ e_1,\ldots,e_{p+1}, e'_1,\ldots,e'_{p+1}\}$ is thus of the form
$$
\left(\begin{array}{ccc}\begin{array}{ccc}
	0&&0\\
	&\ddots&\\
	0&&0\\
\end{array} &\vline&\begin{array}{ccc}
	\overline{\ast}&&0\\
	&\ddots&\\
	\ast&&\overline{\ast}\\
\end{array}\\\hline
\begin{array}{ccc}
	\overline{\ast}&&0\\
	&\ddots&\\
	\ast&&\overline{\ast}
\end{array}&\vline&\begin{array}{ccc}
	0&&0\\
	&\ddots&\\
	0&&0\\
\end{array}\end{array}
\right)$$
with non vanishing $\overline{\ast}$.
Hence $\Omega$ is non degenerate on  $E^v_\lambda\oplus E^w_{\frac{1}{\lambda}}$ which is thus a symplectic subspace stable by $A$.
Remark that the symplectic orthogonal $\bigl(E^v_\lambda\oplus E^w_{\frac{1}{\lambda}}\bigr)^\perp $ to this subspace in $E_\lambda\oplus E_{\frac{1}{\lambda}}$
is also symplectic, stable under $A$ and that $(A-\lambda\Id)^{p+1}=0$ on $\bigl(E^v_\lambda\oplus E^w_{\frac{1}{\lambda}}\bigr)^\perp\cap E_\lambda$, so that the integer $p'$ relative to this subspace is $\le p$.
One obtains by induction a decomposition of $E_\lambda\oplus  E_{\frac{1}{\lambda}}$ into a sum of $A$-stable $\Omega$-orthogonal subspaces of the form
$E^{v^j}_\lambda\oplus E^{w^j}_{\frac{1}{\lambda}}.$\\

We now construct a  basis $\left\{e_1,\ldots, e_{p+1},f_1,\ldots,f_{p+1}\right\}$ of $E^v_\lambda\oplus E^w_{\frac{1}{\lambda}}$ in which the symplectic form has the standard form
$\Omega_0=
\left( \begin{array}{cc}
	0 & \Id \\
	-\Id & 0
\end{array} \right)$
and which gives a normal form for $A$.
From the remark above, this will induce a normal form for $A$ on $E_\lambda\oplus  E_{\frac{1}{\lambda}}$.
If $\lambda$ is real, we  take $v, w $ in  the real generalized eigenspaces $E^{\R}_\lambda$ and $E^{\R}_{\frac{1}{\lambda}}$ and we obtain a symplectic basis of the real symplectic vector space stable by $A$,
$E^{\R v}_\lambda\oplus E^{\R w}_{\frac{1}{\lambda}}$. 
If $\lambda$ is not real, one considers the basis of $E^{\overline{v}}_{\overline\lambda}\oplus E^{\overline{w}}_{\frac{1}{\overline{\lambda}}}$ defined by the conjugate vectors
$\{\overline{e_1},\ldots,\overline{e_{p+1}},\overline{f_1},\ldots, \overline{f_{p+1}}\}.$ 
and this  yields a conjugate normal form on $E_{\overline{\lambda}}\oplus  E_{\frac{1}{\overline{\lambda}}}$ hence a normal form on $W_{[\lambda]} $ and this will induce a real normal form on
$V_{[\lambda]}$.\\

We choose $v$ and $w$ such that $\Omega \Bigl( v,\bigl(A-{\frac{1}{\lambda}}\Id\bigr)^pw \Bigr) = 1.$
We obtain a symplectic basis of $E^v_\lambda\oplus E^w_{\frac{1}{\lambda}}$ by taking the $e_i=(A-\lambda\Id)^{p+1-i}v$ and applying an analogue of Gram-Schmidt procedure to the $e'_i$ thus setting
\subitem $f_{p+1}= e'_{p+1} = \bigl( A-{\frac{1}{\lambda}}\Id \bigr)^pw$
\subitem $f_{p}=\frac{1}{\Omega(e_p,e'_p)} \bigl( e'_{p}-\Omega(e_{p+1}, e'_{p})f_{p+1} \bigr)$
\subitem  and by decreasing induction on $j,$
\subitem $f_j=\frac{1}{\Omega(e_j,e'_j)} \bigl( e'_{j}-\sum_{k>j}\Omega(e_{k}, e'_{j})f_{k}\bigr),$\\
so that any $f_j$ is a linear combination of the $e'_k$ with $k\ge j$.
	
In the basis $\{e_1,\ldots,e_{p+1},f_1,\ldots, f_{p+1}\}$ the matrix associated to $A$ is
\begin{equation*}
	\left(\begin{array}{cc}
		J(\lambda,p+1)&0\\
		0&B
	\end{array}\right)
	\end{equation*}
where
\begin{equation} \label{eq:J}
	J(\lambda,m)=
	\left(\begin{array}{ccccccc}
		\lambda & 1 & 0 & 0 & \ldots & 0 & 0\\
		0 & \lambda & 1 & 0 & \ldots & 0 & 0\\
		0 & 0 & \lambda & 1 & \ldots & 0 & 0\\
		\vdots & \vdots & \vdots & \ddots & \ddots & \vdots & \vdots\\
		0 & 0 & 0 & \ldots & \lambda & 1 & 0\\
		0 & 0 & 0 & \ldots & 0 & \lambda & 1\\
		0 & 0 & 0 & \ldots & 0 & 0 & \lambda
	\end{array}\right)
\end{equation}
 is the elementary $m\times m$ Jordan matrix associated to $\lambda$.
 Since $A$ is symplectic, $B$ is the transpose of the inverse of  $J(\lambda,p+1)$, $B=  \bigl(J(\lambda,p+1)^{-1}\bigr)^{\tau}$.\\This is the normal form  for $A$ restricted to $E^v_\lambda\oplus E^w_{\frac{1}{\lambda}}$. \\
If $\lambda=re^{i\phi}\notin \R$ we consider the symplectic basis $\{e_1,\ldots,e_{p+1},f_1,\ldots, f_{p+1}\}$ of $E^v_\lambda\oplus E^w_{\frac{1}{\lambda}}$ as above and the conjugate symplectic basis $\{\overline{e_1},\ldots,\overline{e_{p+1}},\overline{f_1},\ldots, \overline{f_{p+1}}\}$ of $E^{\overline{v}}_{\overline{\lambda}}\oplus E^{\overline{w}}_{\frac{1}{\overline\lambda}}.$ Writing $e_j=\frac{1}{\sqrt{2}}(u_j-iv_j)$ and 
$f_j=\frac{1}{\sqrt{2}}(w_j+ix_j)$
for all $1\le j\le p+1$ with the vectors $u_j,v_j,w_j, x_j$ in the real vector space $V$, we get a symplectic basis $\{u_1,\ldots,u_{p+1},v_1,\ldots, v_{p+1},w_1,\ldots,w_{p+1},x_1,\ldots, x_{p+1}\}$ of the real subspace of $V$ whose complexification is $E^v_\lambda\oplus E^w_{\frac{1}{\lambda}}\oplus E^{\overline{v}}_{\overline{\lambda}}\oplus E^{\overline{w}}_{\frac{1}{\overline\lambda}}$. In this basis, the matrix associated to $A$ is
\begin{equation*}
	\left(\begin{array}{cc}
		J_\R\bigr(re^{i\phi},2(p+1)\bigl)&0\\
		0&\Bigr(J_\R\bigr(re^{i\phi},2(p+1)\bigl)^{-1}\Bigl)^\tau
	\end{array}\right)
\end{equation*}
where $J_\R(re^{i\phi},2m)$ is the $2m\times 2m$ matrix written in terms of $2\times 2$ matrices as
\begin{equation}\label{eq:JR}
	\left(\begin{array}{ccccccc}
		R(re^{i\phi}) & \Id & 0 & 0 & \ldots & 0 & 0\\
		0 &R(re^{i\phi}) & \Id & 0 & \ldots & 0 & 0\\
		0 & 0 & R(re^{i\phi})& \Id & \ldots & 0 & 0\\
		\vdots & \vdots & \vdots & \ddots & \ddots & \vdots & \vdots\\
		0 & 0 & 0 & \ldots & R(re^{i\phi}) &  \Id & 0\\
		0 & 0 & 0 & \ldots & 0 &R(re^{i\phi}) &  \Id \\
		0 & 0 & 0 & \ldots & 0 & 0 & R(re^{i\phi})
	\end{array}\right)
\end{equation}
 with $R(re^{i\phi})=\left(\begin{array}{cc}
		r\cos \phi&-r\sin \phi\\
		r\sin \phi&r\cos \phi
	\end{array}\right)$.
By induction, we get
\begin{theorem}[Normal form for $A_{\vert V_{[\lambda]}}$ for $\lambda\notin S^1.$]
Let $\lambda\notin S^1$ be an eigenvalue of $A$. Denote $k:=\dim_\C \Ker (A-\lambda\Id)$ (on $V^\C$) and $p$  the smallest integer so that $(A-\lambda\Id)^{p+1}$ is identically zero on the generalized eigenspace $E_\lambda$.\\
If $\lambda\neq \pm1$ is a real eigenvalue of $A$,
there exists a symplectic basis of $V_{[\lambda]}$
in which the matrix associated to the restriction of $A$ to $V_{[\lambda]}$ is a symplectic direct sum of $k$
matrices of the form 
$$
        \left(\begin{array}{cc}
		J(\lambda,p_j+1)&0\\
		0& \bigl(J(\lambda,p_j+1)^{-1}\bigr)^{\tau}
	\end{array}\right)
$$
	 with $p=p_1\ge p_2\ge \dots \ge p_k.$  
If $\lambda=re^{i\phi}\notin(S^1\cup \R)$ is a complex eigenvalue of $A$, there exists a symplectic basis of $V_{[\lambda]}$ in which the matrix associated to the restriction of $A$ to $V_{[\lambda]}$ is a symplectic direct sum of $k$
matrices of the form
$$
	\left(\begin{array}{cc}
		J_\R\bigl(re^{i\phi},2(p_j+1)\bigr)&0\\
		0&\Bigl(J_\R\bigl(re^{i\phi},2(p_j+1)\bigr)^{-1}\Bigr)^\tau
	\end{array}\right)
$$
with $p=p_1\ge p_2\ge \dots \ge p_k.$
\end{theorem}

\subsubsection{Case 2:  $A_{\vert V_{[\lambda]}}$ for $\lambda=\pm 1.$}

In this situation $[\lambda]=\{ \lambda\}$ and  $V_{[\lambda]}$ is  the generalized real eigenspace of eigenvalue $\lambda$, still denoted --with a slight abuse of notation-- $E_\lambda$.
We consider $\widetilde{Q} :\raisebox{.2ex}{$E_\lambda$}/\raisebox{-.2ex}{$\Ker(A-\lambda\Id)^p$}\times 	\raisebox{.2ex}{$E_\lambda$}/\raisebox{-.2ex}{$\Ker(A-\lambda\Id)^p$}\rightarrow \R$
the non degenerate form defined by $\widetilde{Q} ([v],[w])=\Omega \bigl( (A-\lambda\Id)^p v,w \bigr).$
We see directly from equation (\ref{eq:Qbiendef}) that $\widetilde{Q}$ is symmetric if $p$ is odd and antisymmetric if $p$ is even.\\
	
{\bf{If $p=2k-1$ is odd}}, we choose $v\in E_\lambda$ such that
\begin{equation*}
	\widetilde{Q}([v],[v])=\Omega \bigl( (A-\lambda\Id)^p v,v \bigr) \ne 0
\end{equation*}
and consider the smallest subspace $E^v_\lambda$ of $E_\lambda$ stable by $A$ and containing $v$; it is  generated by
\begin{equation*}
	\bigl\{ e_1:=(A-\lambda\Id)^pv,\ldots, e_i:=(A-\lambda\Id)^{p+1-i}v,\ldots, e_{p+1=2k}:= v \bigr\}.
\end{equation*}
We have
\begin{itemize}
	\item $\Omega(e_i,e_j)=0$ if $i+j\le p+1(=2k)$ by equation (\ref{eq:zero}) because\\ 
		$\Omega \bigl((A-\lambda\Id)^{p+1-i}v,(A-\lambda\Id)^{p+1-j}v \bigr) = 0$ if $p+1-j+p+1-i\ge p+1$
	\item $\Omega(e_i,e_{p+2-i})\ne0;$ by equation(\ref{eq:p}) because\\
		$\Omega \bigl( (A-\lambda\Id)^{p+1-i}v,(A-\lambda\Id)^{i-1}v \bigr) = (-1)^{i-1}  \Omega \bigl( (A-\lambda\Id)^{p}v,v \bigr) \ne 0$ by the choice of $v .$
\end{itemize}

\noindent Hence $E^v_\lambda$ is a symplectic subspace because,  in the  basis defined by the $e_i $'s, $\Omega$ has the triangular form
$\left(\begin{array}{ccc}
	0&&\overline{\ast}\\
	&\adots&\\
	\overline{\ast}&&\ast
\end{array}\right)$
and has a non-zero determinant.\\
We can choose $v$ so that $\Omega \bigl((A-\lambda\Id)^{k}v,(A-\lambda\Id)^{k-1}v \bigr) = d=\pm 1$ by rescaling the vector.
One may further assume that 
$$
T_{i,j}(v):=\Omega \bigl((A-\lambda\Id)^{i}v,(A-\lambda\Id)^{j}v \bigr) = 0\quad \textrm{for all }~0\le i,j\le k-1.
$$
Indeed, by formula \eqref{eq:Aij}  we have $T_{i,j}(v)=-\lambda T_{i+1,j}(v)-T_{i+1,j-1}(v),~ T_{i,i}(v)=0$ and we proceed by decreasing induction (observing that $T_{i,j}(v)=-T_{j,i}(v)$) as follows:
\begin{itemize}
\item if $T_{{k-2},{k-1}}(v)= \alpha_1$, we replace $v$ by $v-\frac{\alpha_1}{2d}(A-\lambda\Id)^{2}v$; it spans the same subspace
and the quantities $T_{i,j}(v)$ do not vary for $i+j\ge 2k-1$ but now $T_{{k-2},{k-1}}(v)= 0$, hence $T_{{k-3},{k-1}}(v)=-\lambda T_{{k-2},{k-1}}(v)-T_{{k-2},{k-2}}(v)= 0$;
\item if $T_{{k-3},{k-2}}(v)= \alpha_2$, we replace $v$ by $v+\frac{\alpha_2}{2d}(A-\lambda\Id)^{4}v$; it spans the same subspace
and the quantities $T_{i,j}(v)$ do not vary for $i+j\ge 2k-1$, for $k-2\le i \le j\le k-1$ and for $i=k-3,\, j= k-1$; but
now $T_{{k-3},{k-2}}(v)= 0$, hence also 
$T_{{k-4},{k-2}}(v)= 0$, $T_{{k-4},{k-1}}(v)= 0$;
\item if $T_{{k-4},{k-3}}(v)= \alpha_3$, we replace $v$ by $v-\frac{\alpha_3}{2d}(A-\lambda\Id)^{6}v$; it spans the same subspace
and the quantities $T_{i,j}(v)$ do not vary for $i+j\ge 2k-1$, for $k-3\le i \le j\le k-1$ and for $i=k-4, \, k-2\le j\le k-1$;  now $T_{{k-4},{k-3}}(v)= 0$, hence also 
$T_{{k-5},{k-3}}(v)= 0,~T_{{k-5},{k-2}}(v)= 0,~T_{{k-5},{k-1}}(v)= 0$;
\item assume by induction on increasing $r$  that all $T_{i,j}(v)$ vanish  for $k-r\le i \le j\le k-1$ and for $i=k-(r+1), \, k-(r-1)\le j\le k-1$;  if  $T_{{k-(r+1)},{k-r}}(v)= \alpha_r$, , we replace $v$ by $v+(-1)^r\frac{\alpha_r}{2d}(A-\lambda\Id)^{2r}v$; it spans the same subspace
and the quantities $T_{i,j}(v)$ do not vary for $i+j\ge 2k-1$, for $k-r\le i \le j\le k-1$ and for $i=k-(r+1), \, k-(r-1)\ j\le k-1$;  now 
$T_{{k-(r+1)},{k-r}}(v)= 0$, hence also 
$T_{{k-(r+2)},j(v)}$ for $j\ge k-(r)$;
\item we proceed by induction until all $T_{i,j}$ vanish for $1\le i \le  j\le  k-1$ and for $i=0,\, 1\le j \le k-1$; if  $T_{0,1}(v)=\alpha_{k-1}$,
we replace $v$ by $v+(-1)^{k-1}\frac{\alpha_{k-1}}{2d}(A-\lambda\Id)^{2k-2}v$; it spans the same subspace
and the quantities $T_{i,j}(v)$ do not vary for $i+j\ge 2k-1$, for $1\le i \le j\le k-1$ and for $i=0, \, 2\le j\le k-1$;  now $T_{{0},{1}}(v)= 0$
so that all $T_{i,j}(v)$ vanish for $0\le i,j\le k-1$.
\end{itemize}
We extend $\{e_1,\ldots,e_{k}\}$ into a symplectic basis $\{e_1,\ldots,e_{k}, f_1,\ldots, f_{k}\}$ of our subspace, using again an analogue of the Gram-Schmidt procedure and the fact that $\Omega(e_{k+i},e_{k+j})=T_{k-i,k-j}(v)=0$ for all $1\le i,j\le k$:
\subitem $f'_{k}=  e_{k+1}$ and $f_k:=\tfrac{1}{\Omega(e_k,f'_k)}f'_k=\tfrac{1}{d}e_{k+1}=\tfrac{1}{d}(A-\lambda\Id)^{k-1}v$;
\subitem $f'_{k-1}= e_{k+2}-\Omega(e_{k}, e_{k+2})f_{k}+\Omega(f_{k}, e_{k+2})e_{k}=e_{k+2}-\Omega(e_{k}, e_{k+2})f_{k}$ and $f_{k-1}:=\frac{1}{\Omega(e_{k-1},f'_{k-1})}f'_{k-1}$ is a linear combinaison of $e_{k+1}$ and $e_{k+2}$;
\subitem  and by induction on $j$ 
\subitem $f'_{k-j}=e_{k+j+1}-\sum_{r>k-j}\Omega(e_{r}, e_{k+j+1})f_{r}+\sum_{r>k-j}\Omega(f_{r}, e_{k+j+1})e_{r}=e_{k+j+1}-\sum_{r>k-j}\Omega(e_{r}, e_{k+j+1})f_{r}$ and
$f_{k-j}:=\frac{1}{\Omega(e_{k-j},f'_{k-j})}f'_{k-j}$, \\
so that  $f_{k-j}$ is a linear combination of the $e_i$'s for $k+1\le i \le k+j+1$ .\\

Since $Ae_{j}=\lambda e_{j}+e_{j-1}$ for all $j>1$, the matrix associated to $A$ in the basis constructed above is
\begin{equation*}
	A'=
	\left(\begin{array}{cc}
		J(\lambda,k)&C\\
		0&B
	\end{array}\right)
\end{equation*}
with $C$ identically zero except for the last lign and with the coefficient $C^k_k=\tfrac{1}{d}=d$.
Since it is symplectic we have $J(\lambda,k)^{\tau}B=\Id$ so that $B = \bigl( J(\lambda,k)^{-1}\bigr)^{\tau}$ and
$CJ(\lambda,k)^{\tau}$ is symmetric, hence diagonal of the form $\operatorname{diag} \bigl(0,\ldots,0,d\bigr)$
so that $C$ is of the form
$C(k,d,\lambda)=\operatorname{diag} \bigl(0,\ldots,0,d\bigr)\bigl( J(\lambda,k)^{-1}\bigr)^{\tau}$ so
\begin{equation}\label{eq:C}
	C(k,d,\lambda):=
	\left(\begin{array}{ccccc}
		0&\ldots &0&0&0\\
		\vdots&&&&\vdots\\
		0&\ldots &0&0&0\\
	         (-\lambda)^{k-1} d&\ldots&(-\lambda)^2d&(-\lambda)d&d
	\end{array}\right).
\end{equation}
The matrix
$\left(\begin{array}{cc}
	J(\lambda,k)&C(k,d,\lambda)\\
	0& \bigl( J(\lambda,k)^{-1}\bigr)^{\tau}
\end{array}\right)$
with $d=\pm 1$ is the normal form of $A$ restricted to $E^v_\lambda$.\\

{\bf{If $p=2k$ is even}}, we choose $v$ and $w$ in $E_\lambda$ such that
\begin{equation*}
	\widetilde{Q}([v],[w])=\Omega\bigl((A-\lambda\Id)^p v,w\bigr)=1
\end{equation*}
and we consider the smallest subspace $E^v_\lambda\oplus E^w_\lambda$ of $E_\lambda$ stable by $A$ and containing $v$ and $w.$
It is of dimension $4k+2.$
Remark that $\Omega\bigl((A-\lambda\Id)^p v,v\bigr)=0$ since $\Omega\bigl((A-\lambda\Id)^p v,w\bigr)=(-1)^p \Omega\Bigl( v, \bigl(A-\frac{1}{\lambda}\Id\bigr)^pw\Bigr)$
for any $v,w\in E_{\lambda}$ and $\lambda=\frac{1}{\lambda}$.
So $\Omega\Bigl((A-\lambda\Id)^{p-r} v, \bigl(A-\frac{1}{\lambda}\Id\bigr)^{r}v\Bigr)=0$ for all $0\le r\le p$.\\

We can assume inductively on decreasing $J$ that  for all $j\ge J$ and for all $0\le r\le j$ we have 
$\Omega\Bigl((A-\lambda\Id)^{j-r} v,\bigl(A-\frac{1}{\lambda}\Id\bigr)^{r}v\Bigr)=0$.
Then, rewritting formula \ref{eq:Aij} permuting $\lambda$ and $\frac{1}{\lambda}$ we have
\begin{eqnarray*}
	&&\Omega\Bigl((A-\lambda\Id)^{J-1-s} v,\bigl(A-\tfrac{1}{\lambda}\Id\bigr)^sv\Bigr)\\
	&&\ \ \ =-\lambda  \Omega\Bigl((A-\lambda\Id)^{J-1-s} v,\bigl(A-\tfrac{1}{\lambda}\Id\bigr)^{s+1}v\Bigr)\\
	&&\ \ \ \quad -\lambda^2\Omega\Bigl((A-\lambda\Id)^{J-2-s} v,\bigl(A-\tfrac{1}{\lambda}\Id\bigr)^{s+1}v\Bigr).
\end{eqnarray*}
and the first term on the righthand side of this equation vanishes by induction.
If $\Omega\bigl((A-\lambda\Id)^{J-1} v,v\bigr)=d\neq 0$, $J$ is even and we replace $v$ by 
$$
	v'=v+\tfrac{d}{2}(A-\lambda\Id)^{p-J+1}w;
$$
we have 
$v'\in E^v_\lambda\oplus E^w_\lambda$ and  $\Omega\bigl((A-\lambda\Id)^p v',w\bigr)=1$. Furthermore\\ 
$\Omega\bigl((A-\lambda\Id)^{j-r} v',(A-\frac{1}{\lambda}\Id)^{r}v'\bigr)=0$ for all $0\le r\le j$ and $j\ge J$ and now
\begin{eqnarray*}
	&&\Omega\bigl((A-\lambda\Id)^{J-1} v',v'\bigr)\\
	&&\ \ \ =\Omega\Bigl((A-\lambda\Id)^{J-1} \bigl(v+\tfrac{d}{2}(A-\lambda\Id)^{p-J+1}w\bigr),v+\tfrac{d}{2}(A-\lambda\Id)^{p-J+1}w\Bigr)\\
	&&\ \ \ = \Omega\bigl((A-\lambda\Id)^{J-1} v,v\bigr)+\tfrac{d}{2}\Omega\bigl((A-\lambda\Id)^p w,v\bigr)\\
	&&\quad\quad  +\tfrac{d}{2}\Omega\bigl((A-\lambda\Id)^{J-1} v,(A-\lambda\Id)^{p-J+1}w\bigr) \\
	&&\quad\quad\quad+(\tfrac{d}{2})^2\Omega\bigl((A-\lambda\Id)^p w,(A-\lambda\Id)^{p-J+1}w\bigr)\\
	&&\ \ \ = d-\frac{d}{2}-\frac{d}{2}=0
\end{eqnarray*}
 so that $\Omega\bigl((A-\lambda\Id)^{J-1-r} v',(A-\frac{1}{\lambda}\Id)^{r}v'\bigr)=0$
for all $0\le r\le J-1$.\\

We assume from now on that we have chosen  $v$ and $w$ in $E_\lambda$ so that\\
$\Omega\bigl((A-\lambda\Id)^p v,w\bigr)=1$ and $\Omega\bigl((A-\lambda\Id)^r v,(A-\frac{1}{\lambda}\Id)^{s}v\bigr)=0$
for all $r,s$.\\ We can proceed similarly with $w$;  we can thus furthermore assume  that \\
$\Omega\Bigl((A-\lambda\Id)^j w,\bigl(A-\frac{1}{\lambda}\Id\bigr)^{k}w\Bigr)=0$ for all $j,k$.

We choose for basis of $E^v_\lambda\oplus E^w_\lambda$:
$\{e_1=(A-\lambda\Id)^pv, \ldots,e_i=(A-\lambda\Id)^{p+1-i}v $ $
\ldots e_{p+1}=v , g_1=(A-\lambda\Id)^pw, \ldots,g_i=(A-\lambda\Id)^{p+1-i}w \ldots g_{p+1}=w\}.$
We have
$\Omega(e_i,e_j)=0,~\forall i,j\le 2k+1, ~\Omega(e_i,g_{p+1-i})\ne0 ~\forall 1\leqslant i \leqslant p+1 ,~
\Omega(e_i,g_j)=0$ when $i+j<p+1$ and $\Omega(g_i,g_j)=0$ for all  $i,j \le 2k+1.$
The matrix associated to $\Omega$ has a triangular form
$\left(\begin{array}{ccc}
0&\vline&\begin{array}{ccc}
	0&&\overline{\ast}\\
	&\adots&\\
	\overline{\ast}&&\ast
\end{array}\\\hline
\begin{array}{ccc}
	0&&\overline{\ast}\\
	&\adots&\\
	\overline{\ast}&&\ast
\end{array}&\vline&0\\
\end{array}\right)$
so that $\Omega$ is  non singular and the subspace $E^v_\lambda\oplus E^w_\lambda$ is symplectic.\\

We extend $\{e_1,\ldots,e_{p+1}\}$ into a symplectic basis of this subspace $\{e_1,\ldots,e_{p+1},$
$f_1,\ldots, f_{p+1}\}$
constructing inductively $f_{p+1}, f_p, \ldots, f_1$
by a  Gram Schmidt procedure:
\subitem $f_{p+1}:=\frac{1}{\Omega(e_{p+1}, g_1)}g_1$; 
\subitem $f_p:=\frac{1}{\Omega(e_{p}, g'_2)}g'_2$ with $g'_2:=g_2-\Omega(e_{p+1}, g_2)f_{p+1}$;
\subitem $f_{p+2-j}:=\frac{1}{\Omega(e_{p+2-j}, g'_j)}g'_j$ with $g'_j:=g_j-\sum_{r<j}\Omega(e_{p+2-r}, g_j)f_{p+2-r};$  so that, inductively, each $f_{p+2-j}$ is a linear combination of the $g_k$'s for $k\le j$.\\

In this basis, the matrix associated to $A$ is
\begin{equation*}
	\left(\begin{array}{cc}
		J(\lambda,p+1)&0\\
		0&B\\
			\end{array}\right).
\end{equation*}
Hence, again, the matrix 
$$
\left(\begin{array}{cc}
		J(\lambda,p+1)&0\\
		0& \bigl( J(\lambda,p+1)^{-1}\bigr)^{\tau}
	\end{array}\right)
$$
is a normal form for $A$ restricted to $E^v_\lambda\oplus E^w_\lambda$.

\begin{theorem}[Normal form for $A_{\vert V_{[\lambda]}}$ for $\lambda=\pm 1.$]\label{normalforms1}
Let $\lambda=\pm 1$ be an eigenvalue of $A$. 
There exists a symplectic basis of $V_{[\lambda]}$
in which the matrix associated to the restriction of $A$ to $V_{[\lambda]}$ is a symplectic direct sum of 
matrices of the form 
$$
\left(\begin{array}{cc}
J(\lambda,r_j)&C(r_j,d_j,\lambda)\\
		0& \bigl( J(\lambda,r_j)^{-1}\bigr)^{\tau}
		\end{array}\right)
$$
where $C(r_j,d_j,\lambda)=\operatorname{diag}\bigl(0,\ldots,0,d_j\bigr)\bigl( J(\lambda,r_j)^{-1}\bigr)^{\tau}$ as in \eqref{eq:C} with $d_j\in \{0,1,-1\}$.
If $d_j=0$ then $r_j$ is odd.
The dimension of the eigenspace of eigenvalue $1$ is given by $2\card \{j \,\vert\, d_j=0\}+\card\{j \,\vert \,d_j\neq0\}$.
\end{theorem}

\subsubsection{Case 3:   $A_{\vert W_{[\lambda]}}$ for $\lambda\in S^1\setminus \pm 1.$}

We consider $\widehat{Q} :\raisebox{.2ex}{$E_\lambda$}/\raisebox{-.2ex}{$\Ker(A-\lambda\Id)^p$}\times \raisebox{.2ex}{$E_{{\lambda}}$}/\raisebox{-.2ex}{$\Ker(A-\lambda\Id)^p$}\rightarrow \C$
the sesquilinear non degenerate form defined by
\begin{equation*}
	\widehat{Q} ([v],[w])=\overline{{\lambda}^p}\Omega \bigl( (A-\lambda\Id)^p v,\overline{w} \bigr).
\end{equation*}
Since $\widehat{Q}$ is non degenerate, we can choose $v\in E_\lambda$ such that $\widehat{Q}([v],[v])\ne 0$ thus $\Omega \bigl( (A-\lambda\Id)^p v,\overline{v} \bigr) \ne 0$ and we consider the smallest subspace
$E^v_\lambda\oplus E^{\overline{v}}_{\overline{\lambda}}$ of $E_\lambda\oplus E_{\overline{\lambda}}$ stable by $A,$ complexification of a real subspace and containing $v$ and $\overline{v}.$
It is thus generated by
\begin{equation*}
	\bigl\{ u_{i}:=(A-\lambda\Id)^{p+1-i}v, v_{i}:=(A-\overline{\lambda}\Id)^{p+1-i}\overline{v}; 1\le i\le p+1 \bigr\}.
\end{equation*}
We have $u_{i}=\overline{v_{i}}$ and, as before
\begin{itemize}
	\item $\Omega(u_{i},u_{j})=0,~\Omega(v_{i}, v_{j})=0$ because $\Omega(E_\lambda,E_\lambda)=0;$
	\item $\Omega(u_{i},v_{k})=0$ if $p+1-k+p+1-i\ge p+1$ i.e. $i+k\le p+1$ by equation (\ref{eq:zero});
	\item $\Omega(u_{i},v_{k})\ne 0$ if $p+2=i+k$ by equation (\ref{eq:p}) and by the choice of $v.$
\end{itemize}
\noindent We conclude that $E_\lambda\oplus E_{\overline{\lambda}}$ is a symplectic subspace.\\

{\bf {Subcase : p=2k-1}}\\
We consider the basis 
$\{u_1,\ldots,u_{k},v_1,\ldots,v_k, v_{p+1},\ldots v_{k+1}, u_{p+1},\ldots u_{k+1}\}$ 
and we transform it by a Gram-Schmidt method into a symplectic basis composed of pairs of conjugate vectors, extending
$b=\{u_1,\ldots,u_{k}, v_1,\ldots,v_k\}$ on which $\Omega$ identically vanishes. Recall that $u_j=\overline{v_j}$ for all $ j$.
We proceed as follows.
We start by the pair of conjugate vectors $v_{k+1}$ and $u_{k+1}$ which are $\Omega$-orthogonal to every element of $b$ except respectively to ${u_k}$ and to $v_k$ and we set
\begin{equation*}
	v'_{k+1} =\frac{1}{\Omega(u_k,v_{k+1})}(v_{k+1}+rv_k),~u'_{k+1} =\frac{1}{\Omega(v_k,u_{k+1})}(u_{k+1}+\overline{r}u_k)=\overline{v'_{k+1}},
\end{equation*}
where $r\in \C$ is chosen such that $\Omega(u'_{k+1},v'_{k+1})=0,$ which gives
\begin{equation*}
	\overline{r}\Omega(u_k,v_{k+1})+r\Omega(u_{k+1},v_k)=-\Omega(u_{k+1},v_{k+1}).
\end{equation*}
This is equivalent to $2i\Imc \bigl( r\Omega(u_{k+1},v_k) \bigr) = - \Omega(u_{k+1},v_{k+1})$ and may be solved since $\Omega(u_{k+1},v_{k+1})=\Omega(\overline{v_{k+1}},v_{k+1})$ is purely imaginary.

We then proceed by induction on $j$, $(1<j\le k)$, starting from the pair of vectors $v_{k+j}, u_{k+j}$ which are $\Omega$-orthogonal to all $\{u_1,\ldots,u_{k-j},v_1,\ldots,v_{k-j}\}$ and defining a
pair of vectors  which are $\Omega$-orthogonal to all previous ones except $\{u_{k+1-j},v_{k+1-j}\}$:
\begin{equation*}
	v''_{k+j}=v_{k+j}-\sum_{i<j}\Omega(u_{k+1-i},v_{k+j})v'_{k+i}+ \sum_{i<j}\Omega(u'_{k+i},v_{k+j})v_{k+1-i}
\end{equation*}
\begin{equation*}
	u''_{k+j}=\overline{v''_{k+j}}
\end{equation*}
then normalizing and $\Omega$-orthogonalizing them
\begin{equation*}
	v'_{k+j} =\frac{1}{\Omega(u_{k+1-j},v''_{k+j})}(v''_{k+j} +r_j v_{k+1-j}), \textrm{and}
\end{equation*}
\begin{equation*}
	u'_{k+j} =\frac{1}{\Omega(v_{k+1-j},u''_{k+j})}(u''_{k+j} +\overline{r_j }u_{k+1-j})=\overline{v'_{k+j} }
\end{equation*}
with $r_j$ such that $2i\Imc \bigl( \overline{r_j}\Omega(u_{k+1-j},v_{k+j}) \bigr)=-\Omega(u''_{k+j},v''_{k+j}).$\\

 In the symplectic basis
$\{u_1,\ldots,u_{k},v_1,\ldots,v_k, v'_{p+1},\ldots v'_{k+1},u'_{p+1},\ldots u'_{k+1}\},$
the matrix associated to $A$ is of the form
\begin{equation*}
	\left(\begin{array}{cccc}
		J(\lambda,k) & 0 & 0 & C\\
		0 & J(\overline{\lambda},k) & \overline{C} & 0\\
		0 & 0 & \bigl( J(\lambda,k)^{-1}\bigr)^\tau & 0\\
		0 & 0 & 0 & \bigl( J(\overline{\lambda},k )^{-1} \bigr)^\tau
	\end{array}\right)
\end{equation*}
where $C$ is a $k\times k$ matrix such that $C^i_{\,j}=0$ when $j>i+1$. The fact that the matrix is symplectic implies
that $S:=CJ(\overline{\lambda},k)^\tau$ is hermitean; since $S^i_{\,j}=0$ when $j>i+1$, we have, for real $s_i$'s and complex $\alpha_i$'s
\[
	C= \left(\begin{array}{ccccccc}s_1& \alpha_1&0&&&\ldots&0\\
	                                                  \overline{\alpha_1}&s_2&\alpha_2&0&&\ldots&0\\
						 0&\overline{\alpha_2}&s_3&\alpha_3&0&\ldots&0\\
						\vdots &0&\ddots&\ddots&\ddots&0&0\\
						\vdots &\ldots&0&\ddots&\ddots&\ddots&0\\
						0&\ldots&&0&\overline{\alpha_{k-2}}&s_{k-1}&\alpha_{k-1}\\
						0&\ldots&&&0&\overline{\alpha_{k-1}}&s_{k}
						\end{array}\right)                      \bigl( J(\overline{\lambda},k )^{-1} \bigr)^\tau.
\]
						Writing $u_j=\frac{1}{\sqrt 2}(x_{2j-1}-ix_{2j}),~v_j=\overline{u_j}=\frac{1}{\sqrt 2}(x_{2j-1}+ix_{2j})$, as well as 
$u'_{2k+1-j}=\frac{1}{\sqrt 2}(y_{2j-1}-i y_{2j})$ and $v'_{2k+1-j}=\frac{1}{\sqrt 2}(y_{2j-1}+i y_{2j})$ for $1\le j\le k$,
the vectors $x_i,y_i$ all belong to the real subspace of $V$ whose complexification is $E^v_\lambda\oplus  E^{\overline{v}}_{\overline{\lambda}}$ and we get a symplectic basis $$\{x_1,\ldots,x_{p+1},y_1,\ldots,y_{p+1} \}$$ of this real subspace of $V$. In this basis, the matrix associated to $A$ is :  
\begin{equation}\label{eqref:3-1}
	\left(\begin{array}{cc}
		J_\R(e^{i\phi},2k) & C_\R\\
		0 & \bigl(J_\R(e^{i\phi},2k)^{-1}\bigr)^\tau
	\end{array}\right)
\end{equation}
where $J_\R(e^{i\phi},2k)$ is defined as in \eqref{eq:JR}
and where $C_\R$ is the $(p+1)\times (p+1)$ matrix written in terms of two by two matrices $C(i,j)$ as
\begin{equation}\label{eq:CR}
C_\R=\left(\begin{array}{cccc}
		C(1,1)&C(1,2)&\ldots &C(1,k)\\
		C(2,1)&C(2,2)&\ldots &C(2,k)\\
		\vdots&&&\\
		C(k,1)&C(k,2)&\ldots &C(k,k)
	\end{array}\right)
\textrm{ with } C(i,j)=0 \textrm{ when } j>i+1.
\end{equation}


{\bf {Subcase : p=2k}}\\
We consider the basis $
	\{u_1,\ldots,u_{k},v_1,\ldots,v_{k+1}, v_{p+1},\ldots v_{k+2}, u_{p+1},\ldots u_{k+1}\}.$
Remark that $\Omega(v_{k+1}, u_{k+1}) = \Omega \bigl( (A-\overline{\lambda}\Id)^{k}\overline{v},(A-\lambda\Id)^{k}v \bigr)$ is purely imaginary; we choose $v$ so that it is $\pm  i$ and we choose
(between $\lambda$ and $\overline{\lambda})$ so that it is $-i.$
We extend the subset $
	b=\{e_1=u_1,\ldots,e_k=u_{k},e_{k+1}=v_1,\ldots,e_{2k+1}=v_{k+1}\}
$ on which $\Omega$ vanishes
into  a symplectic basis $$\{u_1,\ldots,u_{k},v_1,\ldots,v_{k+1}, v'_{p+1},\ldots v'_{k+2}, u'_{p+1},\ldots u'_{k+1}\}$$
via a Gram-Schmidt process. We set $u'_{k+1}= i u_{k+1}$ and we construct  pairs of ``conjugate'' vectors in the following way.
We start by the pair of conjugate vectors $v_{k+2}=(A-\overline{\lambda}\Id)^{k-1}\overline{v}$ et $u_{k+2}=\overline{v_{k+2}}$ which are $\Omega$-orthogonal to all elements of
$b\setminus \{v_{k+1}\}$ except respectively to ${u_k}$ and to $v_k$ and we set
\begin{equation*}
	v''_{k+2} := \bigl( v_{k+2}+i\Omega(u_{k+1},v_{k+2})v_{k+1}+rv_k \bigr),
\end{equation*}
\begin{equation*}
	u''_{k+2} := \bigl( u_{k+2}-i\Omega(v_{k+1},u_{k+2})u_{k+1}+\overline{r}u_k \bigr) = \overline{v'_{k+2}},
\end{equation*}
where $r$ is chosen to have $\Omega(u''_{k+2},v''_{k+2})=0,$ which gives
\begin{equation*}
	\overline{r}\Omega(u_k,v_{k+2})+r\Omega(u_{k+2},v_k) =
\end{equation*}
\begin{equation*}
	-\Omega \bigl( u_{k+2}-i\Omega(v_{k+1},u_{k+2})u_{k+1},v_{k+2}+i\Omega(u_{k+1},v_{k+2})v_{k+1} \bigr) ;
\end{equation*}
The left hand side is $2i\Imc \bigl( r\Omega(u_{k+2},v_k) \bigr)$ and one can find a solution $r\in \C$ because the right hand side is purely imaginary.

Then we set $$v'_{k+2}:=\frac{1}{\Omega(u_k,v_{k+2})}v''_{k+2} \textrm{ and }u'_{k+2} :=\frac{1}{\Omega(v_k,u_{k+2})}u''_{k+2}=\overline{v'_{k+2}}.$$

We then proceed by induction as in the case where $p$ is odd with increasing $j$ $(1<j\le k)$ starting from the pair $v_{k+1+j}, u_{k+1+j}$ which are $\Omega$-orthogonal to all
$\{u_1,\ldots,u_{k-j},v_1,\ldots,v_{k-j}\}$ and defining a pair of vectors $\Omega$-orthogonal to all previous ones.
\begin{equation*}
	v''_{k+1+j}=v_{k+1+j}+i\Omega(u_{k+1},v_{k+1+j})v_{k+1}-\sum_{i<j}\Omega(u_{k+1-i},v_{k+1+j})v'_{k+1+i}
\end{equation*}
\begin{equation*}
	+ \sum_{i<j}\Omega(u'_{k+1+i},v_{k+1+j})v_{k+1-i},\qquad \textrm{and}
\end{equation*}
\begin{equation*}
	u''_{k+j}=\overline{v''_{k+j}}
\end{equation*}
then normalizing and $\Omega$-orthogonalizing them
\begin{equation*}
	v'_{k+1+j} =\frac{1}{\Omega(u_{k+1-j},v_{k+1+j})}(v''_{k+1+j} +r_j v_{k+1-j}), \textrm{and}
\end{equation*}
\begin{equation*}
	u'_{k+1+j} =\frac{1}{\Omega(v_{k+1-j},u_{k+1+j})}(u''_{k+1+j} +\overline{r_j }u_{k+1-j})=\overline{v'_{k+1+j} }
\end{equation*}
with $r_j$ such that $2i\Imc \bigl( r_j\Omega(u''_{k+1-j},v_{k+1-j}) \bigr) = -\Omega(u''_{k+1-j},v''_{k+1-j}).$

In the basis $\{u_1,\ldots,u_{k},v_1,\ldots,v_{k+1}, v'_{p+1},\ldots v'_{k+2}, u'_{p+1},\ldots u'_{k+1}\}$ the matrix associated to $A$  is
of the form
\begin{equation*}
	\left(\begin{array}{cccc}
		J(\lambda,k)&0&{0}&\left(\begin{array}{cc} D&d\end{array}\right)\\
		0&J(\overline{\lambda},k+1)&\left(\begin{array}{c}\overline{ D}\\d'\end{array}\right)&0\\
		0&0&\bigl(J(\lambda,k)^{-1}\bigr)^\tau&0\\
		0&0&0&\bigl(J(\overline{\lambda},k+1)^{-1}\bigr)^\tau
	\end{array}\right)
\end{equation*}
Where $D$ is a $k\times k$ matrix such that $D^i_{\, j}=0$ when $j>i+1$, where $d^\tau=(0,\ldots,0,i)$
and $d'$ is a $1\times k$ matrix.\\
Writing $u_j=\frac{1}{\sqrt{2}}(x_{2j-1}-ix_{2j}),~v_j=\frac{1}{\sqrt{2}}(x_{2j-1}+ix_{2j})$ for all $j\le k$,\\
 $~\qquad v_{k+1}=\frac{1}{\sqrt 2}(x_{2k+1}+iy_{2k+1}); u'_{k+1}=i\frac{1}{\sqrt{2}}(x_{2k+1}-iy_{2k+1})$\\
$~\qquad u'_{2k+2-j}=\frac{1}{\sqrt 2}(y_{2j-1}-i y_{2j})$ and $v'_{2k+2-j}=\frac{1}{\sqrt 2}(y_{2j-1}+i y_{2j})$ for $1\le j\le k$,\\
the vectors $x_i,y_i$ all belong to the real subspace of $V$ whose complexification is $E^v_\lambda\oplus  E^{\overline{v}}_{\overline{\lambda}}$ and we get a symplectic basis $$\{x_1,\ldots,x_{p+1},y_1,\ldots,y_{p+1} \}$$ of this real subspace of $V$. In this basis, the matrix associated to $A$ is
\begin{equation}\label{eqref:3-2}
	\left(\begin{array}{cccc}
		J_\R(e^{i\phi},2k)&d_{2k-1}&C_\R&d_{2k}\\
		0&\cos \phi &d'_\R& -\sin\phi\\
		0&0&\bigl(J_\R(e^{i\phi},2k)^{-1}\bigr)^\tau&0\\
		0&\sin\phi&d''_\R&\cos\phi
	\end{array}\right)
\end{equation}
where $J_\R(e^{i\phi},2k)$ is defined as in \eqref{eq:JR},
where $C_\R$ is the $(p+1)\times (p+1)$ matrix written in terms of two by two matrices $C(i,j)$ as in \eqref{eq:CR}
with $C(i,j)=0$ for $j>i+1,$  where $d_j=(0,\ldots,0,1,0,\ldots,0)^\tau$ is the column matrix with a $1$ in the $j$th position
and where $d'_\R$ and $d''_\R$ are $1\times 2k$ matrices.
\begin{theorem}[Normal form for $A_{\vert V_{[\lambda]}}$ for $\lambda\in S^1\setminus \pm 1.$] 
Let $\lambda=e^{i\phi}$ be an eigenvalue of $A$ on $S^1$ and not equal to $\pm 1$. Denote $k:=\dim_\C \Ker (A-\lambda\Id)$ (on $V^\C$) and $p$  the smallest integer so that $(A-\lambda\Id)^{p+1}$ is identically zero on the generalized eigenspace $E_\lambda$.\\
There exists a symplectic basis of $V_{[\lambda]}$
in which the matrix associated to the restriction of $A$ to $V_{[\lambda]}$ is a symplectic direct sum of $\half k$
matrices of dimension $2s_j\times 2s_j$ with $s_1=p+1\ge s_2,\ldots \ge s_{\frac{k}{2}}$, which are   of the two following types:

-if $s_j=2k$ is even 
	$\left(\begin{array}{cc}
		J_\R(e^{i\phi},2k)&C_\R\\
		0 & \bigl(J_\R(e^{i\phi},2k)^{-1}\bigr)^\tau
	\end{array}\right)$
	as in \eqref{eqref:3-1}
\vspace{0,5cm}

 -if $s_j=2k+1$ is odd
	$\left(\begin{array}{cccc}
		J_\R(e^{i\phi},2k) & d_{2k} & C_\R & d_{s_j}\\
		0 & \cos \phi & d'_\R & -\sin\phi\\
		0 & 0 & \bigl(J_\R(e^{i\phi},2k)^{-1}\bigr)^\tau & 0\\
		0 & \sin\phi & d''_\R & \cos\phi
	\end{array}\right)$
	as in \eqref{eqref:3-2}.
\end{theorem}

\subsection{Density of semisimple elements in the symplectic group}

\begin{proposition}
	Semisimple symplectic matrices (with distinct eigenvalues) are dense in the group of all symplectic matrices
\end{proposition}
\begin{proof}
To see that semisimple elements are dense in the symplectic group, we show that the restriction of $A$ 
to any $V_{[\lambda]}$ (and $W_{[\lambda]}$) can be approached by semisimple elements.
We have decomposed each $V_{[\lambda]}$ into a sum of symplectic subspaces which are mutually
symplectically orthogonal and stable by $A$; we shall approach the restriction of $A$ to each of those
subspaces by semisimple elements with distinct eigenvalues.\\

For $\lambda\notin S^1\setminus \pm 1$,  when such a   subspace (of $V_{[\lambda]}$ or  $W_{[\lambda]}$ depending whether  $\lambda$ is real or not)
 is spanned by a symplectic basis $\{e_1,\ldots, e_{k}, f_1,\ldots,f_{k}\}$
in which the matrix associated to $A$ has the form $\left(\begin{array}{cc}
J(\lambda,k)&C(k,d,\lambda)\\
		0& \bigl( J(\lambda,k)^{-1}\bigr)^{\tau}
		\end{array}\right)$, we define a symplectic transformation $S$  on that subspace defined in the given basis by the matrix
\begin{equation*}
	\left(\begin{array}{cc}
		\widetilde{S}&0\\
		0 &{\widetilde{S}}^{-1}
	\end{array}\right)
	\textrm{where} ~~ \widetilde{S}=\textrm{diag}(1+\epsilon_1,\ldots,1+\epsilon_{k}).
\end{equation*}
The transformation $AS$ on that subspace is semisimple because it admits the eigenvalues $(1+\epsilon_1)\lambda,\ldots,(1+\epsilon_{k})\lambda$ and their inverses.\\
If $\lambda$ is not real, the transformation $S$ is defined by the conjugate matrix on the conjugate basis.
It is then clear that $S$ induces a real and symplectic transformation of the corresponding subspace of  $V_{[\lambda]} ;$ it is
semisimple because it admits the quadruples of eigenvalues corresponding to $\lambda(1+\epsilon_1),\ldots,\lambda(1+\epsilon_{k}).$\\

In the third case considered, for $\lambda \in S^1\setminus \pm 1$, on a subspace of $W_{[\lambda]}$ spanned by a symplectic basis
$\{u_1,...,u_k,v_1,...,v_{k},v'_{2k},\ldots, v'_{k+1},u'_{2k},\ldots, u'_{k+1}\}$ with $v_j=\overline{u_j}$ for $j\le k$,  $v'_j=\overline{u'_j}$ for $j\ge k+1$ in which
the matrix associated to $A$ is of the form
\begin{equation*}
	\left(\begin{array}{cccc}
		J(\lambda,k) & 0 & 0 & C\\
		0 & J(\overline{\lambda},k) & \overline{C} & 0\\
		0 & 0 & \bigl( J(\lambda,k)^{-1}\bigr)^\tau & 0\\
		0 & 0 & 0 & \bigl( J(\overline{\lambda},k )^{-1} \bigr)^\tau
	\end{array}\right)
\end{equation*}
we define  a transformation $S$ given in the basis considered by the matrix
\begin{equation*}
	\left(\begin{array}{cccc}
		D&0&0&0\\
		0&\overline{D}&0&0\\
		0&0&D^{-1}&0\\
		0&0&0 &{\overline{D}}^{-1}
	\end{array}\right)
	\textrm{where} ~~ D=\textrm{diag}(1+\epsilon_1,\ldots,1+\epsilon_{k}).
\end{equation*}
It induces a real symplectic transformation of the corresponding subspace of $V_{[\lambda]}.$
The transformation $AS$ on that subspace is semisimple because it admits the quadruples of eigenvalues associated to
$(1+\epsilon_1)\lambda,\ldots,(1+\epsilon_{k})\lambda.$ \\

On a subspace of $W_{[\lambda]}$ spanned by a symplectic basis
$\{u_1,\ldots,u_{k},v_1,\ldots,v_{k+1},$
$ v'_{p+1},\ldots v'_{k+2}, u'_{p+1},\ldots u'_{k+1}\}$
with $v_j=\overline{u_j}$ for $j\le k$,  $v'_j=\overline{u'_j}$ for $j\ge k+2$ and $u'_{k+1}=i \overline{v_{k+1}}$ in which 
the matrix associated to $A$ is of the form
$$
	\left(\begin{array}{cccc}
		J(\lambda,k)&0&{0}&\left(\begin{array}{cc} D&d\end{array}\right)\\
		0&J(\overline{\lambda},k+1)&\left(\begin{array}{c}\overline{ D}\\d'\end{array}\right)&0\\
		0&0&\bigl(J(\lambda,k)^{-1}\bigr)^\tau&0\\
		0&0&0&\bigl(J(\overline{\lambda},k+1)^{-1}\bigr)^\tau
	\end{array}\right)
	$$
we define  a transformation $S$ given in that basis  by the matrix
\begin{equation*}
	\left(\begin{array}{cccccc}
		D&0&0&0&0&0\\
		0&\overline{D}&0&0&0&0\\
		0&0&e^{i\epsilon}&0&0&0\\
		0&0&0&D^{-1}&0&0\\
		0&0&0&0 &{\overline{D}}^{-1}&0\\
		0&0&0&0&0&e^{-i\epsilon}
	\end{array}\right)
	\textrm{where} ~~ D=\textrm{diag}(1+\epsilon_1,\ldots,1+\epsilon_{k}).
\end{equation*}
It induces a real symplectic transformation of the corresponding subspace of $V_{[\lambda]}.$\\
The transformation $AS$  is semisimple on that subspace because it admits the quadruples of eigenvalues associated to
$(1+\epsilon_1)\lambda,\ldots,(1+\epsilon_{k})\lambda$  and the double associated to $e^{i\epsilon}\lambda.$
\end{proof}

\section{The set $\Sp^{\star}(\R^{2n},\Omega_0)$}
We would like to associate an integer to a path in the symplectic group starting from the identity and ending at an element which does not admit $1$ as eigenvalue.

This paragraph is a development of the page $1317$ of \cite{SalamonZehnder}.
\begin{definition}
	We denote by $\textrm{Sp}^{\star}(\R^{2n},\Omega_0)$ the set of symplectic matrices which do not have $1$ as eigenvalue:
	\begin{equation*}
		\textrm{Sp}^{\star}(\R^{2n},\Omega_0) := \left\{A\in \textrm{Sp}(\R^{2n},\Omega_0) \vert \det (A-\Id) \neq 0\right\}.
	\end{equation*}
\end{definition}
\begin{remarque}
	The set $\textrm{Sp}^{\star}(\R^{2},\Omega_0)$ has two connected components:\\
	$\left\{A\in \textrm{Sp}(\R^{2},\Omega_0)\vert \det (A-\Id) > 0\right\}$ and $\left\{A\in \textrm{Sp}(\R^{2},\Omega_0)\vert \det (A-\Id) < 0\right\} .$
\end{remarque}
Indeed $\Sp(\R^2,\Omega_0)=Sl(2,\R)$ and we have:
\begin{eqnarray*}
	\det \Biggl(
	\left(\begin{array}{cc}
		a+d&-b+c\\
		b+c & a-d
	\end{array}\right)
	-\Id \Biggr)&=&(a+d-1)(a-d-1)+(b-c)(b+c)\\
	&=&1+a^2+b^2-c^2-d^2-2a=2(1-a)
\end{eqnarray*}
thus $\Sp^{\star}(\R^2,\Omega_0)$ is the disjoint union of the two connected components.
\begin{eqnarray*}
	\Sp^{\star}(\R^2,\Omega_0) &=&
	\left\{ \,\left(\begin{array}{cc}
		a+d&-b+c\\
		b+c&a-d
	\end{array}\right)
	\,\vert\, a<1 ~\textrm{~and~} a^2+b^2-c^2-d^2=1\,\right\} \\
	&& \cup
	\left\{ \, \left(\begin{array}{cc}
		a+d&-b+c\\
		b+c&a-d
	\end{array}\right)\,\vert\, a>1 ~\textrm{~and~} a^2+b^2-c^2-d^2=1\,\right\}
\end{eqnarray*}
The property stated above generalizes for all dimensions:
\begin{theorem}[\cite{SalamonZehnder}]\label{prop:connexite}
	The group $\Sp(\R^{2n},\Omega_0)$ is connected. The subset $\Sp^*(\R^{2n},\Omega_0)$
	has  two connected components defined by
	\begin{equation*}
		\Sp^\pm(\R^{2n},\Omega_0) = \left\{ A \in \Sp(\R^{2n},\Omega_0)\ \vert\ \pm \det(A-\Id) 
		>		0\right\}
	\end{equation*}
	and every loop in  $\Sp^\pm$
	is contractible in $\Sp(\R^{2n},\Omega_0) .$
\end{theorem}
\begin{remarque}
	In the connected component where $\det(A-\Id)>0,$ we have the matrix
	\begin{equation} \label{W+}
		W^+:=-\Id \textrm{~~~and~~~} \rho(-\Id)=\Biggl(\rho
		\left(\begin{array}{cc}
			-1&0\\
			0 & -1
		\end{array}\right)
		\Biggr)^n=(-1)^n;
	\end{equation}
	In the component where $\det(A-\Id)<0$, we have the diagonal matrix
	\begin{eqnarray} \label{W-} 
		W^-&:=&\textrm{diag}(2,-1,\ldots, -1,\half,-1,\ldots,-1)\\ 
		&\textrm{and}&\rho(W^-)=\rho
			\left(\begin{array}{cc}
				2&0\\
				0 & \half
			\end{array}\right)
			\Biggl(\rho
			\left(\begin{array}{cc}
				-1&0\\
				0 & -1
			\end{array}\right)
			\Biggr)^{n-1}= (-1)^{n-1}.
	\end{eqnarray}
\end{remarque}

\begin{proof}[ of the connectedness]
	We have seen in paragraph \ref{appendicesemisimple} that every element $A \in \Sp(\R^{2n},\Omega_0)$  may be connected to an element $A'\in \Sp(\R^{2n},\Omega_0)$ as close
	as we want and with all eigenvalues distinct.
	
	If $A \in \Sp^\pm(\R^{2n},\Omega_0) ,$ the element $A'$ may be chosen in the same set since the map $A\mapsto \det(A-\Id)$ is a continuous map on $\Sp(\R^{2n},\Omega_0) .$
	We now show that  every element $A' \in \Sp^+(\R^{2n},\Omega_0)$ (and similarly every element $A' \in\Sp^-(\R^{2n},\Omega_0))$ with  distinct eigenvalues may be connected to
	an element with all eigenvalues equal to $-1$ (respectively to an element with a pair of eigenvalues $(2,\half)$ of multiplicity $1$ and all other eigenvalues equals to $-1).$
	
	Indeed, we consider a basis (of $\C^{2n}$) in which $A'$ is diagonal and we modify $A'$ on the symplectic subspace
	\begin{equation*}
		V_{\left[\lambda\right]}\otimes_{\R}\C = E_{\lambda}\oplus E_{\frac{1}{\lambda}}\oplus E_{\bar{\lambda}}\oplus E_{\frac{1}{\bar{\lambda}}}
	\end{equation*}
	by considering $\lambda(t) , \frac{1}{\lambda(t)} , \bar{\lambda}(t) ,  \frac{1}{\bar{\lambda}(t)}$ where $\lambda : \left[0,1\right]\rightarrow\C : t\mapsto \lambda(t)$ is defined by:
	\begin{eqnarray}
		\lambda(t) &=& (1-t)\lambda -t ~\textrm{ if  } \lambda \notin S^1\cup \R\\
				&&\textrm{choosing between }\lambda \textrm{ and }\frac{1}{\lambda}\textrm{ so that } \vert\lambda\vert<1\nonumber \\
		\lambda(t) &=& e^{it\pi}e^{i(1-t)\varphi} ~\textrm{ if } \lambda = e^{i\varphi} \in S^1 \setminus \left\{ \pm 1\right\}\\
				&& \textrm{choosing between } \lambda \textrm{ and }\overline{\lambda}\textrm{ so that } 0<\varphi<\pi\nonumber \\
		\lambda(t) &=& (1-t)\lambda -t ~\textrm{if } \lambda \textrm{ is real negative and } \lambda>-1.
	\end{eqnarray}
	For the real positive eigenvalues we consider two pairs of eigenvalues $\left\{ a , \frac{1}{a} , b , \frac{1}{b} \right\}$ with $a>1$ and $b>1 .$
	We bring them  to $\left\{ \frac{a+b}{2} , \frac{2}{a+b} , \frac{a+b}{2} , \frac{2}{a+b}\right\}$ without passing trough $1$ via
	$\left\{ \bigl(1-\frac{t}{2}\bigr)a+\frac{t}{2}b,\frac{1}{\left(1-\frac{t}{2}\right)a+\frac{t}{2}b}, \bigl(1-\frac{t}{2}\bigr)b+\frac{t}{2}a,\frac{1}{\left(1-\frac{t}{2}\right)b+\frac{t}{2}a}\right\} .$
	Then we go to $\left\{ \left(\frac{a+b}{2}\right)e^{i\theta} , \left(\frac{2}{a+b}\right)e^{-i\theta} , \left(\frac{a+b}{2}\right)e^{-i\theta} , \left(\frac{2}{a+b}\right)e^{i\theta}\right\}$ and
	finally we proceed as for $\left\{ \lambda , \frac{1}{\lambda} , \bar{\lambda} , \frac{1}{\bar{\lambda}} \right\}$ with $\lambda \notin S^1 \cup \R .$
	
	A semisimple element of $\Sp^+(\R^{2n},\Omega_0)$ must have an even number of real positive eigenvalues since $\det
	\left(\begin{array}{cc}
		a-1&0\\
		0&\frac{1}{a}-1
	\end{array}\right)
	<0 \textrm{ if }a>1.$
	
	The elements considered  above exhaust all possibilities for $\Sp(\R^{2n},\Omega_0) .$
	We have thus connected every element of $\Sp^+(\R^{2n},\Omega_0)$ to a semisimple element with all eigenvalues equal to $-1$, i.e. to $W^+=-\Id .$
	This implies that $\Sp^+(\R^{2n},\Omega_0)$ is connected.
	
	Note that we have shown also that every matrix in $\Sp(\R^{2n},\Omega_0)$ may be connected to $-\Id$ since we can connect two eigenvalues $1$ on a symplectic space of dimension
	$2,$ to two eigenvalues equals to $-1$ by a rotation
	$\left(\begin{array}{cc}
		\cos t\pi & -\sin t\pi\\
		\sin t\pi & \cos t\pi
	\end{array}\right) .$
	We deduce that $\Sp(\R^{2n},\Omega_0)$ is connected.
	
	In a similar way, since a pair of positive real eigenvalues $\{b,\frac{1}{b}\}$ may be connected to $\{2,\half\}$ without going through $1,$ we have connected every element of
	$\Sp^-(\R^{2n},\Omega_0)$ to a semisimple element $A''$ whose eigenvalues are either $2$ and  $\half$ (each with multiplicity $1$) or $-1.$
	We have thus
	\begin{equation*}
		A'' = kW^-k^{-1} ~~\textrm{ where  } k \in \Sp(\R^{2n},\Omega_0).
	\end{equation*}
	Since $\Sp(\R^{2n},\Omega_0)$ is connected, we can find a path $k(t)$ in the group $\Sp(\R^{2n},\Omega_0)$ which links $k$ to $\Id .$
	So $k(t)W^-k^{-1}(t)$ will link $A''$ to $W-$ staying in $\Sp^-(\R^{2n},\Omega_0)$ as the eigenvalues remained unchanged.
\end{proof}
\begin{remarque}
We have shown above that an element $A\in \Sp^*(\R^{2n},\Omega_0)$ is in $\Sp^+(\R^{2n},\Omega_0)$ if and only if  the number of real positive eigenvalues is a multiple of four.\\
\end{remarque}
\begin{proof}[ of the contractibility of every loop in $\Sp^*(\R^{2n},\Omega_0)$]\\
	Let $\gamma : \left[0,1\right] \rightarrow \Sp^\pm(\R^{2n},\Omega_0) : t \mapsto \gamma(t)$ be a loop; i.e a continuous map such that $\gamma(0) = \gamma(1) .$
	
	Recall that a loop is contractible in $\Sp(\R^{2n},\Omega_0)$ if and only if its image by $\rho$ is contractible.
	Indeed:
	\begin{equation*}
		\rho : \Sp(\R^{2n},\Omega_0) \longrightarrow S^1
	\end{equation*}
	is a continuous map wich induces an isomorphism on the fundamental groups, because $\det_{\C} : U(n)\rightarrow S^1$ induces an isomorphism on the fundamental groups
	(see section \ref{preliminaries}).
	
	Recall (\ref{eq:rhopremespece}) that the value of $\rho$ on a matrix $A \in \Sp(\R^{2n},\Omega_0)$ may be obtained in the following way: we consider the eigenvalues
	$\lambda_1,\ldots,\lambda_{2n}$ repeated according to  their algebraic multiplicity and we look at the eigenvalues $\lambda_1,\ldots,\lambda_{n}$ of first kind (thus all eigenvalues
	such that $\vert \lambda_i\vert<1,$ half of those who equal $1,$ half of those who equal $-1,$ and $r$ times the value $\lambda=e^{i\varphi}\ne \pm 1$ if the quadratic form
	$Q:E_{\lambda}\times E_{\lambda} \longrightarrow \C :~ (v,w) \longmapsto \Imag\Omega(v,\bar{w})$ has signature $(2r,2s)).$
	Then
	\begin{equation*}
		\rho(A) = \prod_{\lambda_i \textrm{ of first kind}} \frac{\lambda_i}{\vert \lambda_i\vert}
	\end{equation*}
	The map  
	\begin{equation*}
		\Sp(\R^{2n},\Omega_0) \rightarrow \Lambda(n) := {\raisebox{.2ex}{$\C^n$}} / {\raisebox{-.2ex}{permutations of the elements}},
	\end{equation*}
	which sends a matrix $A$ on the set
	$(\lambda_1(A),\ldots,\lambda_n(A))$ of its eigenvalues of the first kind is continuous.
	More precisely, given a path $\gamma : \left[0,1\right] \rightarrow \Sp(\R^{2n},\Omega_0),$ we can choose continuous functions $\Lambda_1,\ldots,\Lambda_n : \left[0,1\right] \rightarrow \C$
	such that $\Lambda_1(t),\ldots,\Lambda_n(t)$ are the eigenvalues of first kind of $\gamma(t) .$
	
	If the path $\gamma$ takes its values in $\Sp^\pm(\R^{2n},\Omega_0)$ there is an even number (respectively odd) of the $\frac{\Lambda_i}{\vert \Lambda_i\vert}$ which equal $1$ for every value of $t.$
	We define functions
	\begin{equation*}
		\alpha_i : \left[0,1\right]\longrightarrow \left[0,2\pi\right] ~(i=1\ldots n)~ \textrm{such that }e^{i\alpha_i(t)} = \frac{\Lambda_i(t)}{\vert\Lambda_i(t)\vert}.
	\end{equation*}
	They are uniquely determined if  $\frac{\Lambda_i(t)}{\vert\Lambda_i(t)\vert}\ne 1.$
	Since there are $2k\, (\textrm{resp }2k+1)$ numbers $\frac{\Lambda_i(t)}{\vert\Lambda_i(t)\vert} =1$, we define for $k\,  (\textrm{resp }k+1)$ of them the value $\alpha_i(t) = 2\pi$ and for the other $k$ of
	them the value $\alpha_i(t) =0$; we can thus make the functions $\alpha_i$ continuous.
	It is indeed continuous because if a quadruple $\lambda,\frac{1}{\lambda},\overline{\lambda},\frac{1}{\overline{\lambda}}$ degenerates into two pairs of real positive eigenvalues, the two eigenvalues
	of first kind have an argument equal to $\varphi$ and to $(2\pi-\varphi)$ tending one to zero and the other to $2\pi .$
	
	If we start from a loop $\gamma$ which lies in $\Sp^\pm(\R^{2n},\Omega_0)$, its image by each $\alpha_i$  is a loop $\subset \left[0,2\pi\right]$ and is thus contractible.
	
	Hence the  map $[0,1]\rightarrow \R : t \mapsto \sum_{i=1}^n \alpha_i(t)$ is contractible and thus the image by $\rho$  of the loop
	$[0,1]\rightarrow S^1 : t\mapsto \rho\bigl(\gamma(t)\bigr) = e^{i\sum_{i=1}^n \alpha_i(t)}$ is contractible in $S^1 .$
	
	Thus the loop $\gamma$ is contractible in $\Sp(\R^{2n},\Omega_0) .$
\end{proof}
\section{Conley-Zehnder index for a path of  matrices in $\textrm{SP}(n)$}
\begin{definition}[\cite{Salamon, SalamonZehnder}]
	We consider the set $\textrm{SP}(n)$ of continuous paths of matrices in $\Sp(\R^{2n},\Omega_0)$ linking the matrix $\Id$ to a matrix in $\Sp^\star(\R^{2n},\Omega_0) :$
	\begin{equation*}
		\textrm{SP}(n) := \Biggl\{ \psi : \left[0,1\right] \rightarrow \Sp(\R^{2n},\Omega_0)\, \left\vert
		\begin{array}{l}
			 \psi(0)=\Id \textrm{ and }\\
 			1\textrm{ is not an eigenvalue of } \psi(1)
		\end{array}
		\right. \Biggr\}.
	\end{equation*}
\end{definition}
From the decomposition of $\textrm{Sp}^{\star}(\R^{2n},\Omega_0)$ into two connected components, every path $\psi \in \textrm{SP}(n)$ may be extended to a continuous path
$\widetilde{\psi} : \left[0,2\right] \rightarrow \textrm{Sp}(\R^{2n},\Omega_0)$ such that $\widetilde{\psi}$ coincides with $\psi$ on the interval $\left[0,1\right],$ such that
$\widetilde{\psi}(s) \in \textrm{Sp}^{\star}(\R^{2n},\Omega_0)$ for all $s\geqslant 1$ and  such that the path ends either in $\widetilde{\psi}(2) = -\Id=W^+$ either in
$\widetilde{\psi}(2) = W^- := \operatorname{diag}(2,-1,\ldots,-1,\half,-1,\ldots,-1).$
Remark that $\rho(W^{-} )= (-1)^{n-1}$ and $\rho(-\Id)=(-1)^{n} ; $ thus $\rho^2 \circ \widetilde{\psi} : \left[0,2\right] \rightarrow \textrm{S}^1$ is a loop.
\begin{definition}[\cite{Salamon, SalamonZehnder}]\label{def1}
	The Conley-Zehnder index of a path $\psi \in \textrm{SP}(n)$ is the integer given by the degree of the map $\rho^2 \circ \tilde{\psi} : \left[0,2\right] \rightarrow \textrm{S}	^1:$
	\begin{equation}
		\mu_{\textrm{CZ}}(\psi) := \operatorname{deg}(\rho^2\circ \widetilde{\psi})
	\end{equation}
	where $\widetilde{\psi} : \left[0,2\right] \rightarrow \textrm{Sp}(\R^{2n},\Omega_0)$ is a continuous  extension of the path $\psi$ such that
	\begin{itemize}
		\item  $\widetilde{\psi}$ coincides with $\psi$ on the interval $\left[0,1\right];$
		\item $\widetilde{\psi}(s) \in \textrm{Sp}^{\star}(\R^{2n},\Omega_0)$ for all $s\geqslant 1;$
		\item $\widetilde{\psi}(2) = -\Id=:W^+$ or $\widetilde{\psi}(2) = W^- := \operatorname{diag}\bigl(2,-1,\ldots,-1,\half,-1,\ldots,-1\bigr).$
	\end{itemize}
\end{definition}
\begin{lemma}[\cite{Salamon, SalamonZehnder}]
	The Conley-Zehnder index of a path $\psi \in \textrm{SP}(n)$ does not depend on the choice of the chosen extension $\tilde{\psi}$.
\end{lemma}
\begin{proof}
	If $\widetilde{\psi_1}$ and $\widetilde{\psi_2}$ are two extensions of $\psi$, their restrictions to $[1,2]$,
$\widehat{\psi_1}$ and $\widehat{\psi_2}$,  are paths in $\Sp^*(\R^{2n},\Omega_0)$ starting at $\psi(1)$ and ending at $W^\pm.$
	They end at the same element since this element is in the same connected component of $\Sp^*(\R^{2n},\Omega_0)$ as $\psi(1).$
	The two paths $\widehat{\psi_1} $ and $\widehat{\psi_2}$ are homotopic since any loop in $\Sp^*(\R^{2n},\Omega_0)$ is contractible in $\Sp(\R^{2n},\Omega_0)$. Hence the two paths $\widetilde{\psi_1}$ and $\widetilde{\psi_2}$ are homotopic.
	Their images under $\rho^2 ,$  which is a continuous map, are thus homotopic.
	We have seen that these images are loops; they have thus the same degree.
\end{proof}
\begin{proposition}[\cite{Salamon}]\label{proprietescz}
	The Conley-Zehnder index
	\begin{equation*}
		\mu_{\textrm{CZ}} :   \textrm{SP}(n) \rightarrow \Z :~\psi \mapsto    \mu_{\textrm{CZ}}(\psi)	
	\end{equation*}
	has the following properties:
	\begin{enumerate}
		\item\label{cznaturalite} \emph{(Naturality)} For all path $\phi : \left[0,1\right] \rightarrow \textrm{Sp}(\R^{2n},\Omega_0)$\ we have
		 	\begin{equation*}
				\mu_{\textrm{CZ}}(\phi\psi\phi^{-1}) = \mu_{\textrm{CZ}}(\psi);
			\end{equation*}
		\item\label{czhomotopy} \emph{(Homotopy)} The Conley-Zehnder index is constant on the components of $ \textrm{SP}(n);$
		\item \emph{(Zero)} If $\psi(s)$ has no eigenvalue on the unit circle for $s>0$ then
			\begin{equation*}
				\mu_{\textrm{CZ}}(\psi) =0;
			\end{equation*}
		\item \emph{(Product)} If $n'+n''=n,$ we identify $\Sp(\R^{2n'},\Omega_0)\times \Sp(\R^{2n"},\Omega_0)$ with a subgroup of  $\Sp(\R^{2n},\Omega_0)$ in the obvious way.
			Then 
			\begin{equation*}
				\mu_{\textrm{CZ}}(\psi' \oplus \psi'') =  \mu_{\textrm{CZ}}(\psi') +  \mu_{\textrm{CZ}}(\psi'');
			\end{equation*}
		\item\label{czlacet} \emph{(Loop)} If $\phi : \left[0,1\right] \rightarrow \Sp(\R^{2n},\Omega_0)$ is a loop with $\phi(0) = \phi(1) = \Id ,$ then
			\begin{equation*}
				\mu_{\textrm{CZ}}(\phi\psi) = \mu_{\textrm{CZ}}(\psi) + 2\mu(\phi)
			\end{equation*}
			where $ \mu(\phi)$ is the Maslov index of the loop $\phi$, i.e.  $\mu(\phi) = \operatorname{deg}(\rho \circ \phi);$
		\item\label{czsignature} \emph{(Signature)} If $S=S\tr \in \R^{2n\times 2n}$ is a symmetric non degenerate matrix with all eigenvalues of absolute value $<2\pi \ (\|S\| < 2\pi)$ and if 
			$\psi(t) = \operatorname{exp}(J_0St)$ for $t\in \left[0,1\right],$ then $\mu_{\textrm{CZ}}(\psi) = \half \sign(S)$ $(\textrm{where } \sign(S) \textrm{ is the signature of }S) .$
		\item\label{determinant} \emph{(Determinant)} $(-1)^{n- \mu_{CZ}(\psi)} = \operatorname{sign}\det\bigr(\Id - \psi(1)\bigl)$
		\item\label{inverse} \emph{(Inverse)} $\mu_{CZ}(\psi^{-1}) = \mu_{CZ}(\psi\tr) = -\mu_{CZ}(\psi)$
	\end{enumerate}
\end{proposition}

\begin{proof}[ of the naturality]\\
	Naturality comes from the invariance of $\rho;$
	\begin{equation*}
		\rho(kAk^{-1})=\rho(A) ~\forall A,k\in \Sp(\R^{2n},\Omega_0).
	\end{equation*}
	If $\widetilde{\psi} : [0,2]\rightarrow \Sp(\R^{2n},\Omega_0)$ is an extension of the path $\psi$ as before, we choose as extension of the path $\phi\psi\phi^{-1}$ the path $\widetilde{\psi'} $
	defined by $\widetilde{\psi'} (t)=\phi'(t)\widetilde{\psi}(t)\phi'(t)^{-1}$ for $t\ge 1,$ where $\phi':[1,2] \rightarrow \Sp(\R^{2n},\Omega_0)$ is a path linking $\phi(1)$ to $\Id.$
	Since $\rho\bigl(\phi(t)\psi(t)\phi^{-1}(t) \bigr) = \rho \bigl( \psi(t) \bigr)$ for $t\in [0,1]$ and $\rho\bigl( \widetilde{\psi'} (t) \bigr) = \rho \bigl( \phi'(t)\widetilde{\psi}(t)\phi'(t)^{-1} \bigr) = \rho\bigl( \widetilde{\psi}(t) \bigr)$
	for $t\in [1,2]$ we have $\rho^2\circ \widetilde{\psi}=\rho^2\circ \widetilde{\psi'}$ thus $\mu_{\textrm{CZ}}(\phi\psi\phi^{-1}) = \mu_{\textrm{CZ}}(\psi).$
\end{proof}

\begin{proof}[ of the homotopy property]\\
	Two paths $\psi_0$ and $\psi_1\in  \textrm{SP}(n)$ are in the same component if and only if there exists a continuous map
	\begin{equation*}
		{\psi} : [0,1] \times [0,1]\rightarrow \Sp(\R^{2n},\Omega_0)~(s,t)\mapsto \psi(s,t)=:\psi_s(t)
	\end{equation*}
	such that $\psi_s(0)=\Id ~\forall s$ and $\psi_s(1)\in \Sp^*(\R^{2n},\Omega_0)~\forall s .$
	Thus $\psi_s\in  \textrm{SP}(n)\, \forall s.$
	Consider an extension $\widetilde{\psi_0}$ of $\psi_0$ and define an extension $\widetilde{\psi_s}$ of $\psi_s$ by 
	$$
		\widetilde{\psi_s}(t)=
		\left\{ \begin{array}{ll}
			\psi_s(t) & 0\le t\le 1\\
			\psi_{s(3-2t)}(1) & 1\le t\le \frac{3}{2}\\
			\widetilde{\psi_0}(2t-2) & \frac{3}{2}\le t\le 2 \ .
		\end{array}\right.
	$$
	
	The map $\widetilde{\psi} : [0,1] \times [0,2]\rightarrow \Sp(\R^{2n},\Omega_0)~(s,t)\mapsto \widetilde{\psi}(s,t)=\widetilde{\psi}_s(t)$ is continuous and defines a homotopy between
	$\widetilde{\psi_0}$ and $\widetilde{\psi_1}.$
	Since $\rho^2$ is continuous and since the degree of a map from $[0,2]$ in $S^1$ is invariant by homotopy, we have
	$\mu_{\textrm{CZ}}(\psi_0)= \operatorname{deg}(\rho^2\circ \widetilde{\psi_0})=\operatorname{deg}(\rho^2\circ \widetilde{\psi_1})  =\mu_{\textrm{CZ}}(\psi_1). $
\end{proof}

\begin{proof}[ of the zero property]\\
	If $\psi(s)$ has no eigenvalue on the unit circle, we have $\rho\bigl(\psi(t)\bigr) = 1 ~\forall t\in [0,1]$ by the property of normalisation of $\rho,$ the fact that $\rho$ is continuous and the fact that
	$\rho(\Id)=1.$
	
	We have seen in the proof of  Theorem \ref{prop:connexite} that we can find a path of matrices in $\Sp^*(\R^{2n},\Omega_0)$ which brings back the eigenvalues of
	$\psi(1)$ which are not real and negative to $-1$ by groups of $4$ without going through a point of $S^1\setminus \ \{- 1\}.$
	We thus find an extension $\widetilde{\psi}$ of $\psi$ such that $\rho\bigl(\widetilde{\psi}(t)\bigr) = \rho\bigl({\psi}(1)\bigr) = 1~\forall t\ge 1$ since we have a continuous path of matrices with no eigenvalues on
	$S^1\setminus \{\pm1\}$ and with an even number of pairs of negative eigenvalues [Recall indeed  formula (\ref{eqrho})].
	A fortiori $\rho^2\bigl(\widetilde{\psi}(t)\bigr) = 1~\forall t \in [0,2]$ and thus $\mu_{\textrm{CZ}}(\psi)= \operatorname{deg}(\rho^2\circ \widetilde{\psi})=0.$
\end{proof}

\begin{proof}[ of the product]\\
	If $\psi'$ is a path in $\textrm{SP}(n')$ and $\psi''$ a path in $ \textrm{SP}(n'')$ we can find extensions
	$\widetilde{\psi'} :[0,2]\rightarrow \Sp(\R^{2n'},\Omega_0)$ et $\widetilde{\psi''} :[0,2]\rightarrow \Sp(\R^{2n''},\Omega_0)$ as before.
	With the obvious identifications
	\begin{equation*}
		\widetilde{\psi'}\oplus \widetilde{\psi''} :[0,2]\rightarrow \Sp(\R^{2n},\Omega_0)
	\end{equation*}
	is such that:\\
	${{\bullet}}(\widetilde{\psi'}\oplus \widetilde{\psi''})(t)=({\psi'}\oplus {\psi''})(t) \quad \forall t\in [0,1];$\\
	${{\bullet}}(\widetilde{\psi'}\oplus \widetilde{\psi''})(t) \in  \Sp^*(\R^{2n},\Omega_0) \quad \forall t\in [1,2];$ and\\
	${{\bullet}}(\widetilde{\psi'}\oplus \widetilde{\psi''})(2)$ is equal either to $W^{\pm}$ either to a diagonal matrix $W'$ with twice the value $2$, twice the value $\frac{1}{2}$ and $2n-4$ times
	the value $-1 .$
	
	If it is equal to $W^{\pm},$ then $\widetilde{\psi'}\oplus \widetilde{\psi''}$ is an extension of ${\psi'}\oplus {\psi''}.$
	If it is equal to $W'$ we obtain an extension of ${\psi'}\oplus {\psi''}$ considering $\widetilde{\psi'}\oplus \widetilde{\psi''}$ followed by a path of matrices bringing the two pairs of positive
	eigenvalues back to $-1$ without going trough another point of $S^1$ as before; since the value of $\rho$ does not change along this last path we have in all cases:
	\begin{eqnarray*}
		\mu_{\textrm{CZ}}({\psi'}\oplus {\psi''})&=& \operatorname{deg}\bigl(\rho^2\circ(\widetilde{\psi'}\oplus \widetilde{\psi''})\bigr)=
		\operatorname{deg}\bigl((\rho^2\circ\widetilde{\psi'})\oplus\rho^2\circ \widetilde{\psi''})\bigr)\\
		&&~\qquad\quad\textrm{due to the multiplicativity property of } \rho\\
		&=& \operatorname{deg}\bigl(\rho^2\circ\widetilde{\psi'}\bigr)+\operatorname{deg}\bigl(\rho^2\circ\widetilde{\psi''}\bigr)=\mu_{\textrm{CZ}}({\psi'})+\mu_{\textrm{CZ}}( {\psi''}).
	\end{eqnarray*}
\end{proof}
Before proving the loop property, we indicate a lemma that will be useful in the computations of Conley-Zehnder indices of iterates of closed orbits.
\begin{lemma}\label{lem:prod}
	Let $\varphi$ and $\psi$ be two paths in $\Sp(\R^{2n},\Omega_0)$
	\begin{equation*}
		\varphi , \psi : \left[0,T\right] \longrightarrow \Sp(\R^{2n},\Omega_0)
	\end{equation*}
	such that $\varphi(0) = \psi(0) = \Id .$
	We consider on one hand the path $\psi \varphi$ obtained as the  product of the two paths:
	\begin{equation*}
		\psi \varphi : \left[0,T\right] \longrightarrow \Sp(\R^{2n},\Omega_0) : t \longmapsto \psi(t)\varphi(t)
	\end{equation*}
	and on the other hand the catenation of the path $\varphi$ with the translation
	of the path $\psi$ so that it starts from $\varphi(T) :$
	\begin{equation*}
		\psi \diamond \varphi : \left[0,T\right] \longrightarrow \Sp(\R^{2n},\Omega_0) :
		\left\{\begin{array}{cc}
			\varphi(2t)&t\leqslant\frac{T}{2}\\
			\psi\Bigl(2\bigl(t-\frac{T}{2}\bigr)\Bigr)\varphi(T)&t\geqslant\frac{T}{2}
		\end{array}\right. .
	\end{equation*}
	Then those two paths are homotopic.
\end{lemma}
\begin{proof}
	We consider the homotopy $\chi : \left[0,1\right]\times\left[0,T\right] \longrightarrow \Sp(\R^{2n},\Omega_0)$ defined by
	\begin{equation*}
		\chi_s(t) =
		\left\{\begin{array}{cc}
			\varphi(2t) & t\leqslant\frac{sT}{2}\\
			\psi\left(\frac{2}{2-s}\left(t-\frac{sT}{2}\right)\right)\varphi \left(sT + \frac{2(1-s)}{2-s}\left(t-\frac{sT}{2}\right)\right) & t\geqslant \frac{sT}{2}
		\end{array}\right.
	\end{equation*}
	It is continuous $\left(\chi_s\left(\frac{sT}{2}\right) = \varphi(sT)~ \forall s \right)$ and
	\begin{equation*}
		 \left\{\begin{array}{cccc}
	 		\chi_0(t) &=& \psi(t)\varphi(t) & \forall t \in \left[0,T\right]\\
			\chi_1(t) &= & \left\{\begin{array}{c}
				\varphi(2t)\\
				\psi\Bigl(2\bigl(t-\frac{T}{2}\bigr)\Bigr)\varphi(T)
			\end{array}\right.&
			\left.\begin{array}{c}
				t\leqslant\frac{T}{2}\\
				t\geqslant \frac{T}{2}
			\end{array}\right.
		\end{array}\right.
	\end{equation*}
\end{proof}

\begin{proof}[ of the loop property]\\
	If $\phi : \left[0,1\right] \rightarrow \Sp(\R^{2n},\Omega_0)$ is a loop with $\phi(0) = \phi(1) = \Id ,$ and if $\psi \in \textrm{SP}(n),$ then by the previous lemma, the product path $\phi\psi$ is
	homotopic to the catenation of $\phi$ and $\psi .$
	Thus, by the invariance of the degree by homotopy we have
	\begin{equation*}
		\mu_{\textrm{CZ}}(\widetilde{\phi\psi}) =  \operatorname{deg}\bigl(\rho^2 \circ(\widetilde{ \psi} \circ \phi)\bigr) =
		\operatorname{deg}(\rho^2 \circ\widetilde{\psi})+ \operatorname{deg}(\rho^2 \circ \phi)=\mu_{\textrm{CZ}}(\psi) + 2\mu(\phi)
	\end{equation*}
	where $\mu(\phi) = \operatorname{deg}(\rho \circ \phi).$
\end{proof}

\begin{proof}[ of the signature (\cite{SalamonZehnder})]\\
	Since $S$ is a symmetric matrix, there exists an orthogonal matrix $P$ of determinant equal to $1$ such that $PSP^{-1} =\operatorname{diag}(a_1,\ldots,a_{2n})$ with all $a_i$ non zero
	since $S$ is non degenerate.
	The condition $\|S\| < 2\pi$ implies that $\|J_0S\| < 2\pi$ thus the eigenvalues of $J_0S$ are all smaller in norm than $2\pi$ and $\operatorname{exp}(J_0S)$ doesn't admit $1$
	as an eigenvalue thus is in $\Sp^*(\R^{2n},\Omega_0).$
	We consider a path of orthogonal matrices $P_s$ starting at $P$ and ending at the identity.
	The norm of $P_sSP_s^{-1} $ is always smaller than $2\pi$ and thus $\operatorname{exp}(J_0P_sSP_s^{-1})$ doesn't admit the eigenvalue $1$ for any $s.$
	This shows that the path $\operatorname{exp}(J_0St)$ is in the same connected component of $\operatorname{SP}(n)$ as the path $\operatorname{exp}(J_0PSP^{-1}t)$ and we can thus
	assume, in view of the homotopy property for the computation of the Conley-Zehnder index that $S$ is diagonal.
	
	We have then $J_0S=
	\left(\begin{array}{cc}
		0&\operatorname{diag}(-a_{n+1},\ldots,-a_{2n})\\
		\operatorname{diag}(a_{1},\ldots,a_{n})&0
	\end{array}\right)$ with all the $\vert a_i\vert <2\pi.$
	Since $\psi(t) = \operatorname{exp}(J_0St),$ we can decompose $(\R^{2n},\Omega_0)$ in a sum of $n$ symplectic $2$-planes and the Conley-Zehnder index of the path $\psi$ is the sum of the
	indices of the paths $\psi_i(t):= \exp t
	\left(\begin{array}{cc}
		0&-a_{n+i}\\
		a_i&0
	\end{array}\right)$
	by the  product property.
	Those paths are
	\begin{equation*}
		\psi_i(t):=
		\left(\begin{array}{cc}
			\cos(\sqrt{a_ia_{n+i}}t)&-\sqrt{\frac{a_{n+i}}{a_i}}\sin(\sqrt{a_ia_{n+i}}t)\\
			\sqrt{\frac{a_i}{a_{n+i}}}\sin(\sqrt{a_ia_{n+i}}t)&  \cos(\sqrt{a_ia_{n+i}}t)
		\end{array}\right) 
	\end{equation*} 
	if $a_ia_{n+i}>0$ and $a_i>0;$
	\begin{equation*}
		\psi_i(t):=
		\left(\begin{array}{cc}
			\cos(\sqrt{a_ia_{n+i}}t)&\sqrt{\frac{a_{n+i}}{a_i}}\sin(\sqrt{a_ia_{n+i}}t)\\
			-\sqrt{\frac{a_i}{a_{n+i}}}\sin(\sqrt{a_ia_{n+i}}t)&  \cos(\sqrt{a_ia_{n+i}}t)
		\end{array}\right) 
	\end{equation*} 
	if $a_ia_{n+i}>0$ and $a_i<0;$
	\begin{equation*}
		\psi_i(t):=
		\left(\begin{array}{cc}
			\cosh(\sqrt{-a_ia_{n+i}}t)&\pm\sqrt{\frac{-a_{n+i}}{a_i}}\sinh(\sqrt{-a_ia_{n+i}}t)\\
			\pm\sqrt{\frac{-a_i}{a_{n+i}}}\sinh(\sqrt{-a_ia_{n+i}}t)&  \cosh(\sqrt{-a_ia_{n+i}}t)
		\end{array}\right)
	\end{equation*}
	if $a_ia_{n+i}<0.$
	In the third case, there are no eigenvalues on the circle and the contribution to the Conley-Zehnder index is thus zero by the zero property.
	In the two first cases, the eigenvalues are on the circle, they are equal to $\cos(\sqrt{a_ia_{n+i}}t)\pm \sin(\sqrt{a_ia_{n+i}}t);$ we can change the eigenvalues, staying in the same component of
	$\operatorname{SP}(n)$, to have $a_i=a_{n+i};$ we obtain then the path $\psi_i'(t)$ which consists in turning in the $2$-plane; we have $\rho(\psi'_i(t))=e^{i\sqrt{a_ia_{n+i}}t}$ in the
	first case and $\rho(\psi'_i(t))=e^{-i\sqrt{a_ia_{n+i}}t}$ in the second case.
	An extension $\widetilde{\psi'_i}$ consists in going further or coming  back turning in the $2$-plane to reach the matrix $-\Id .$
	Thus $\rho(\widetilde{\psi'_i})$ varies from $e^0$ to $e^{i\pi}$ in the first case ($a_i>0$) and from $e^0$ to $e^{-i\pi}$ in the second case ($a_i<0$).
	Hence $\mu_{\textrm{CZ}}(\psi_i)=\operatorname{deg}(\rho^2\circ \widetilde{\psi'_i})=\textrm{sign}(a_i).$
	Thus 
	\begin{eqnarray*}
		\mu_{\textrm{CZ}}(\psi)& =&\card \{i\le n\, \vert \,a_i>0~a_ia_{n+i}>0\}\\
		&&~~\qquad\qquad-\card \{i\le n\, \vert \,a_i<0~a_ia_{n+i}>0\}\\
		&=& \half \sign(S).
	\end{eqnarray*}
\end{proof}

\begin{proof}[ of the determinant property (\cite{SalamonZehnder})]\\
	If $\det\bigl(\Id - \psi(1)\bigr) > 0$, it is in the connected component of $W^{+}$ so $\widetilde{\psi}(2) = W^{+}$ for a prolongation $\widetilde{\psi}$
	and $\rho(W^{+}) = (-1)^{n}$.
	If $\det\bigl(\Id - \psi(1)\bigr) < 0$, it is in the connected component of $W^{-}$ so $\widetilde{\psi}(2) = W^{-}$ for a prolongation $\widetilde{\psi}$
	and $\rho(W^{-}) = (-1)^{n-1}$.
	Since the degree of the map $\rho^{2} \circ \widetilde{\psi}$ is even when $\rho\bigl(\widetilde{\psi}(2)\bigr)=1$ and odd when $\rho\bigl(\widetilde{\psi}(2)\bigr)=-1$,
	we have $(-1)^{n- \mu_{CZ}(\psi)} = \operatorname{sign}\det\bigr(\Id - \psi(1)\bigl)$.
\end{proof}

\begin{proof}[ of the inverse property]\\
	If $\psi \in SP(n)$ we define $\psi^{-1}$ and $\psi\tr \in SP(n)$ by
	$$
		\psi^{-1}(t) = \bigl(\psi(t)\bigr)^{-1} \hspace{1cm} \psi\tr(t) = \bigl(\psi(t)\bigr)\tr
	$$
	Since, for any symplectic matrix $A\tr=\left(\begin{array}{cc}O&\Id\\ 
	-\Id & 0\end{array}\right)A^{-1}\left(\begin{array}{cc}O&-\Id\\ 
	\Id & 0\end{array}\right)$, those two paths have the same
	Conley-Zehnder index in view of the naturality. Furthermore,
	we have seen in proposition \ref{prop:rhocarre} that $\rho(A^{-1}) = \bigl(\rho(A)\bigr)^{-1}$ for any symplectic matrix $A$.
	Since the inverse of a prolongation $\widetilde{\psi}$ is a prolongation of the inverse, we have
	\begin{eqnarray*}
		\mu_{CZ}(\psi^{-1}) & = & \deg(\rho^2\circ \widetilde{\psi^{-1}}) = \deg(\rho^2\circ \widetilde{\psi}^{-1}) = \deg\bigl((\rho\circ \widetilde{\psi})^{-2}\bigr)\\
		& = & -\deg\bigl((\rho\circ \widetilde{\psi})^{2}\bigr) = -\deg(\rho^2\circ \widetilde{\psi})\\
		& = & -\mu_{CZ}(\psi)
	\end{eqnarray*}
\end{proof}
\begin{proposition}\label{caract}
	The properties \ref{czhomotopy}, \ref{czlacet} and \ref{czsignature} of homotopy, loop and signature characterize the Conley-Zehnder index.
\end{proposition}
\begin{proof}
	Assume $\mu' : SP(n) \rightarrow \Z$ is a map satisfying those properties.
	Let $\psi : [0,1] \rightarrow \Sp(\R^{2n},\Omega_0)$ be an element of $SP(n)$ (i.e $\psi(0) = \Id ,\ \psi(1) \in \Sp^*(\R^{2n},\Omega_0)).$
	Then $\psi$ is in the same component of $SP(n)$ as its prolongation $\tilde{\psi} : [0,2] \rightarrow \Sp(\R^{2n},\Omega_0)$ with $\tilde{\psi}(s) \in \Sp^*(\R^{2n},\Omega_0) \ \forall s \ge 1$
	and $\tilde{\psi}(2)$ either equal to $W^+$,
	either equal to $W^- $.
	So $\mu'(\psi) = \mu'(\tilde{\psi}).$\\

Observe that 	$W^+ = \exp \pi (J_0 S^+)$ with $S^+ = \Id$ and $W^- = \exp \pi (J_0 S^-)$ with
$$
	S^- =
	\left(\begin{array}{cccc}
		0 & 0 & -\frac{\log 2}{\pi} & 0\\
		0 & \Id_{n-1} & 0 & 0\\
		-\frac{\log 2}{\pi} & 0 & 0 & 0\\
		0 & 0 & 0 & \Id_{n-1}
	\end{array}\right).
$$
          The catenation of $\tilde{\psi}$ and $\psi_2^-$ (the path $\psi_2$ in the reverse order, i.e followed from end to beginning) when
	$\psi_2 : [0,1] \rightarrow \Sp(\R^{2n},\Omega_0) \ t \mapsto \exp t \pi J_0 S^{\pm}$ is a loop $\phi .$
	Hence $\tilde{\psi}$ is homotopic to the catenation of $\phi$ and $\psi_2 ,$ which, by lemma \ref{lem:prod}, is homotopic to the product $\phi \psi_2 .$
	
	We thus have $\mu'(\psi) = \mu'(\phi \psi_2) .$
	By the loop condition $\mu'(\phi \psi_2) = \mu'(\psi_2) + 2\mu(\phi)$ and by the signature condition $\mu'(\psi_2) = \half \sign (S^{\pm}).$
	Thus 
$$
        \mu'(\psi) = 2\mu(\phi) + \half \sign (S^{\pm}).
$$ 
          Since the same is true for $\mu_{CZ}(\psi)$, this proves uniqueness.
\end{proof}
Remark that we have only used the signature property to know the value of the Conley-Zehnder index on the paths $\psi_{2\pm}:\ t\in [0,1]  \mapsto  \
\exp t \pi J_0 S^{\pm}$.
Hence we  have : 
\begin{proposition}\label{compdef}
Let $\psi\in \textrm{SP}(n)$ be a  continuous path of matrices in $\Sp(\R^{2n},\Omega_0)$ linking the matrix $\Id$ to a matrix in $\Sp^\star(\R^{2n},\Omega_0)$
and let  $\widetilde{\psi} : \left[0,2\right] \rightarrow \Sp(\R^{2n},\Omega_0)$ be an extension such that $\widetilde{\psi}$ coincides with $\psi$ on the interval $\left[0,1\right],$ such that
$\widetilde{\psi}(s) \in \Sp^{\star}(\R^{2n},\Omega_0)$ for all $s\geqslant 1$ and  such that the path ends either in $\widetilde{\psi}(2) = -\Id=W^+$ either in
$\widetilde{\psi}(2) = W^- := \operatorname{diag}(2,-1,\ldots,-1,\half,-1,\ldots,-1).$
The Conley-Zehnder index of $\psi$ is equal to the integer given by the degree of the map ${\tilde{\rho}}^2 \circ \tilde{\psi} : \left[0,2\right] \rightarrow \textrm{S}^1:$
	\begin{equation}
		\mu_{\textrm{CZ}}(\psi) := \operatorname{deg}({\tilde{\rho}}^2\circ \widetilde{\psi})
	\end{equation}
	for any continuous map ${\tilde{\rho}}: \Sp(\R^{2n},\Omega_0)\rightarrow S^1$ which coincide with the (complex) determinant $\det_\C$ on $\textrm{U}\left(n\right) = \textrm{O}\left(\R^{2n}\right) \cap \textrm{Sp}\left(\R^{2n},\Omega_0\right)$, such that ${\tilde{\rho}}(W^-)=\pm 1$, and such that
	\[
		\operatorname{deg}\, ({\tilde{\rho}}^2\circ \psi_{2-})= n-1  \quad \textrm{ for }\,  \psi_{2-} : t \in [0,1]\mapsto \exp t \pi J_0 S^-.
	\]		\end{proposition}
\begin{proof}
This is a direct consequence of the fact that the map defined by $ \operatorname{deg}({\tilde{\rho}}^2\circ \widetilde{\psi})$
has the homotopy property, the loop property (since any loop is homotopic to a loop of unitary matrices where $\rho$ and $\det_{\C}$
coincide) and we have added what we need of the signature property to characterize the Conley-Zehnder index. Indeed  $\half \sign S^-=n-1, S^+=\Id_{2n}, \half \sign S^+=n$ and\\
$\exp t \pi J_0S^+=\exp t \pi \left(\begin{array}{cc}
		0 & -\Id_n\\
		\Id_n&0
                \end{array}\right)=\left(\begin{array}{cc}
		\cos \pi t \Id_n& -\sin \pi t\Id_n\\
		\sin \pi t \Id_n& \cos \pi t \Id_n
                \end{array}\right)$ is in $\textrm{U}\left(n\right)$ 
so that
  ${\tilde{\rho}}^2\left(\exp t \pi \left(\begin{array}{cc}
		0 & -\Id_n\\
		\Id_n&0
                \end{array}\right)\right)= e^{2\pi int}$ and
$\operatorname{deg}({\tilde{\rho}}^2\circ \psi_{2+})=n$.
\end{proof}
\begin{cor}
The Conley-Zehnder index of a path $\psi\in \textrm{SP}(n)$ is given by
\begin{equation}\label{defdetunit}
		\mu_{\textrm{CZ}}(\psi) := \operatorname{deg}({{\det}_\C}^2\circ U\circ \widetilde{\psi})
\end{equation}
	where $U:  \Sp(\R^{2n},\Omega_0)\rightarrow \textrm{U}\left(n\right)$ is the projection defined by the polar
	decomposition $U(A=OP)=O=AP^{-1}$ with $P$ the unique symmetric positive definite matrix such that $P^2=A\tr A.$
\end{cor}
\begin{proof} The map $\tilde{\rho}:={{\det}_\C}\circ U$ satisfies all the  properties stated in proposition \ref{compdef};
it is indeed continuous, coincides obviously with $\det_{\C}$ on $\textrm{U}\left(n\right)$ and we have that 
  $\exp t \pi \left(\begin{array}{cc}
		0 & -\frac{\log 2}{\pi}\\
		-\frac{\log 2}{\pi}&0
                \end{array}\right) =\left(\begin{array}{cc}
		2^t&  0\\
	            0& 2^{-t} 
                \end{array}\right)$   is a positive symmetic matrix so that
      $ U(\exp t \pi J_0 S^-)=  \left(\begin{array}{cccc}
		1&0&0&0\\
		0&\cos \pi t \Id_{\tiny{n-1}}& 0&-\sin \pi t\Id_{\tiny{n-1}}\\
		0&0&1&0\\
		0&\sin \pi t \Id_{\tiny{n-1}}& 0& \cos \pi t \Id_{\tiny{n-1}}\\
		 \end{array}\right)$; hence
		  ${\det}_{\C}^2\circ U(\exp t \pi J_0 S^-)=e^{2\pi i(n-1)t}$
		  and $\operatorname{deg}({{\det}_\C}^2\circ U\circ \psi_{2-})=n-1$.
		  \end{proof}
		  
Formula \eqref{defdetunit} is the definition of the Conley-Zehnder index used in \cite{DeGosson, HWZ}.
Another formula is obtained using the parametrization of the symplectic group introduced in \cite{refs:RobRaw}:
\begin{cor}
The Conley-Zehnder index of a path $\psi\in \textrm{SP}(n)$ is given by
\begin{equation}\label{newdef}
		\mu_{\textrm{CZ}}(\psi) := \operatorname{deg}(\hat\rho^2\circ \widetilde{\psi})
\end{equation}
	where $\hat\rho:  \Sp(\R^{2n},\Omega_0)\rightarrow S^1$ is the normalized complex determinant
	of the $\C$-linear part of the matrix:
	\begin{equation}
\hat\rho(g)=\frac{{\det}_{\C} \left(\half(g-J_0gJ_0)\right)}{\left\vert{\det}_{\C} \left(\half(g-J_0gJ_0)\right)\right\vert}.	
	\end{equation}
\end{cor}
\begin{proof}
Remark that for any $g \in \Sp(\R^{2n},\Omega_0)$ the element $C_g:=\half(g-J_0gJ_0)$, which clearly defines
a complex linear endomorphism of $\C^{n}$ since it commutes with $J_0$, is always invertible.
Indeed for any non-zero $v \in V$ 
\[
 4\Omega_0 (C_gv, J_0C_gv) \, = \, 2\Omega_0 (v,J_0v) + \Omega_0 (gv, J_0gv) +
\Omega_0 (gJ_0v, J_0gJ_0v) > 0.
\]
If $g\in \textrm{U}\left(n\right)$, then $C_g=g$ so that $\hat\rho(g)={\det}_{\C}(g)$ hence
$\hat\rho$ is a continuous map which coincide with $\det_{\C}$ on $\textrm{U}\left(n\right)$.
Furthermore\\
  $\half \left(\left(\begin{array}{cc}
		2^t&  0\\
	            0& 2^{-t} 
                \end{array}\right) - J_0 \left(\begin{array}{cc}
		2^t&  0\\
	            0& 2^{-t} 
                \end{array}\right)J_0\right)=\half \left(\begin{array}{cc}
		2^t+2^{-t}&  0\\
	            0& 2^t2^{-t} 
                \end{array}\right)$ hence its complex determinant is equal to $2^t+2^{-t}$ and its normalized complex determinant is equal to $1$ 
                so that $\hat\rho(\exp t \pi J_0 S^-)=e^{\pi i (n-1)t}$ and $\operatorname{deg}(\hat\rho^2\circ \psi_{2-})=n-1.$
\end{proof}
\section{Example in dimension $2:$ the index of the path $ \exp t J_0S $}\label{sectionexptJS}

\begin{proposition}
	We consider the path of symplectic matrices in dimension $2$ defined by
	\begin{equation*}
		\psi: [0,T]\rightarrow \Sp(\R^2,\Omega_0) :~ t\mapsto \exp (tJ_0 S)
	\end{equation*}
	where $S$ is a symmetric non degenerate matrix and $\exp (TJ_0 S)\ne \Id.$
	We have
	\begin{equation}
		\mu_{CZ}(\psi)=
		\left\{\begin{array}{ll}
			\left(\half +\left\lfloor \frac{\sqrt{a_1 a_2}T}{2\pi}\right\rfloor\right) \sign S &\textrm{ if } \sign(S)\ne0\\ 
			0&\textrm{ if } \sign(S)=0
		\end{array}\right.
	\end{equation}
	where $a_1$ and $a_2$ are the eigenvalues of $S,$ where $\sign S$ is the signature of $S$ and where $\left\lfloor b\right\rfloor$ denote the greatest integer $\le b.$
\end{proposition}
\begin{proof}
	Since $S$ is symmetric, we diagonalize it in the orthogonal group and we  get a symplectic basis of $\R^2$ in which the matrices read:
	\begin{equation*}
		S=
		\left(\begin{array}{cc}
			a_1 &0 \\
			0 & a_2
		\end{array}\right)
		\qquad\qquad J_0S=
		\left(\begin{array}{cc}
			0&-a_2\\
			a_1 &0
		\end{array}\right)
	\end{equation*}
	thus, as in the former section,
	\begin{equation*}
		\psi(t) =
		\left\{\begin{array}{l}
			\left(\begin{array}{cc}
				\cos\sqrt{a_1a_2}{t} & -\sqrt\frac{a_2}{a_1}\sin\sqrt{a_1a_2}{t} \\
				\sqrt\frac{a_1}{a_2}\sin\sqrt{a_1a_2}{t} & \cos\sqrt{a_1a_2}{t}
			\end{array}\right)
			\textrm{ if } a_1>0 \textrm{ and } a_2>0\\
			\left(\begin{array}{cc}
				\cos\sqrt{a_1a_2}{t} & \sqrt\frac{a_2}{a_1}\sin\sqrt{a_1a_2}{t} \\
				-\sqrt\frac{a_1}{a_2}\sin\sqrt{a_1a_2}{t} & \cos\sqrt{a_1a_2}{t}
			\end{array}\right)
			\textrm{ if } a_1<0 \textrm{ and } a_2<0\\
			\left(\begin{array}{cc}
				\cosh\sqrt{-a_1a_2}{t} & \sqrt\frac{-a_2}{a_1}\sinh\sqrt{-a_1a_2}{t} \\
				\sqrt\frac{-a_1}{a_2}\sinh\sqrt{-a_1a_2}{t} & \cosh\sqrt{-a_1a_2}{t}
			\end{array}\right) \\
			\qquad\qquad\qquad\qquad\qquad\qquad\qquad\qquad\qquad~\textrm{ if } a_1>0 \textrm{ and } a_2<0\\
			\left(\begin{array}{cc}
				\cosh\sqrt{-a_1a_2}{t} &- \sqrt\frac{-a_2}{a_1}\sinh\sqrt{-a_1a_2}{t} \\
				-\sqrt\frac{-a_1}{a_2}\sinh\sqrt{-a_1a_2}{t} & \cosh\sqrt{-a_1a_2}{t}
			\end{array}\right) \\
			\qquad\qquad\qquad\qquad\qquad\qquad\qquad\qquad\qquad~~~\textrm{ if }a_1<0 \textrm{ et } a_2>0\\
		\end{array}\right.	
	\end{equation*}
	In the third and fourth case, $\psi(t)$ has no eigenvalues on the circle and thus
	\begin{equation*}
		\mu_{CZ}(\psi)=0 \textrm{ when } \sign(S)=0.
	\end{equation*}
	In the first two cases, to compute the Conley-Zehnder index of the path $\psi,$ we extend $\psi$ to $\widetilde{\psi}$ and we compute the degree of $\rho^2\circ \widetilde\psi.$
	Note that the eigenvalues of $\psi(t)$ are equal to $\cos\sqrt{a_1a_2}{t} \pm i\sin\sqrt{a_1a_2}{t} ,$ and that an eigenvector in $\C^2$ associated to
	$\cos\sqrt{a_1a_2}{t} + i\sin\sqrt{a_1a_2}{t}$ is given by $z=(\sqrt{\frac{a_2}{a_1}},-i)$ in the first case and by $z=\bigl(\sqrt{\frac{a_2}{a_1}},i\bigr)$ in the second case.
	Since $\Omega_0\left((\sqrt{\frac{a_2}{a_1}},0),(0,1)\right)>0$, $ \, \cos\sqrt{a_1a_2}{t} + i\sin\sqrt{a_1a_2}{t} $ is the eigenvalue of the first kind in the first case and 
	$\cos\sqrt{a_1a_2}{t} - i\sin\sqrt{a_1a_2}{t}$ is the eigenvalue of the first kind in the second case.
	We extend $\psi$ to $\widetilde\psi$ going (or going back) to $-\Id$ without going through $\Id.$
	Since the period in $t$ to go back to the identity with $\psi$ is $\frac{2\pi}{\sqrt{a_1a_2}},$ and since $T=\frac{2\pi}{\sqrt{a_1a_2}}\left(\left\lfloor \frac{\sqrt{a_1a_2}T}{2\pi} \right\rfloor+b\right)$
	with $0<b<1$, the extension $\widetilde\psi$ is homotopic to the path
	\begin{equation*}
		\widetilde{\psi'}:\left[0, \frac{2\pi}{\sqrt{a_1a_2}}\left(\left\lfloor \frac{\sqrt{a_1a_2}T}{2\pi} \right\rfloor+\half\right)\right]\rightarrow \Sp(\R^{2n},\Omega_0)
	\end{equation*}
	\begin{equation*}
		t\mapsto \widetilde{\psi'}(t):=
		\left(\begin{array}{cc}
			\cos\sqrt{a_1a_2}{t} &
				\pm
				\sqrt\frac{a_2}{a_1}\sin\sqrt{a_1a_2}{t}\\
			\pm
			\sqrt\frac{a_1}{a_2}\sin\sqrt{a_1a_2}{t} & \cos\sqrt{a_1a_2}{t}
		\end{array}\right).
	\end{equation*}
	Since $\rho \bigl(\psi'(t)\bigr)=e^{
	\pm
	i\sqrt{a_1a_2}{t} }$, $\rho^2 \bigl(\psi'(t)\bigr)=e^{{
	\pm
	2i\sqrt{a_1a_2}{t} }}$ and the degree of the map $\rho^2\circ \widetilde{\psi'}$ is
	$\pm	
	2\Bigl(\left\lfloor \frac{\sqrt{a_1a_2}T}{2\pi} \right\rfloor+\half \Bigr)$ we have 
	\begin{equation*}
		\mu_{\textrm{CZ}}(\psi) = \biggl(\left\lfloor \frac{\sqrt{a_1a_2}T}{2\pi} \right\rfloor+\half \biggr)\sign(S).
	\end{equation*}
\end{proof}

\section{Generalized definition of the  Conley-Zehnder index}
In \cite{RobbinSalamon}, Robbin and Salamon define a Maslov-type index for a continuous path of Lagrangians in a symplectic vector space $(\R^{2n},\Omega)$ and they give a definition of a generalization of the Conley-Zehnder index
defined for any  path of symplectic matrices.

\subsection{The space of Lagrangians in $(\R^{2n},\Omega)$}
A \emph{Lagrangian} in $(\R^{2n},\Omega)$ is a subspace $V$ of $\R^{2n}$ of dimension $n$ such that $\left. \Omega \right\vert_{V \times V} = 0 .$
Given any Lagrangian $V$ in $\R^{2n} ,$ there exists a Lagrangian $W$ (not unique!) such that $V \oplus W = \R^{2n} .$
With the choice of such a supplementary $W$ any Lagrangian $V'$ in a neighborhood of $V$ (any Lagrangian supplementary to $W$) can be identified to a linear map $\alpha : V \rightarrow W$ through
$V' = \{ v + \alpha (v) \vert v \in V \} ,$ with $\alpha$ such that $\Omega\bigl(\alpha(v),w\bigr) + \Omega\bigl(v,\alpha(w)\bigr) = 0 \ \forall v,w \in V$.
Hence it can be identified to a symmetric bilinear form $\underline{\alpha} : V \times V \rightarrow \R : (v,v') \mapsto \Omega\bigl(v,\alpha(v')\bigr) .$
In particular the tangent space at a point $V$ to the space $\mathcal{L}_n$ of Lagrangians in $(\R^{2n},\Omega)$ can be identified to the space of symmetric bilinear forms on $V .$

If $\Lambda : [a,b] \rightarrow \mathcal{L}_n : t \mapsto \Lambda_t$ is a smooth curve of Lagrangian subspaces, we define $Q(\Lambda_{t_0}, \dot{\Lambda}_{t_0})$
to be the symmetric bilinear form on $\Lambda_{t_0}$ defined by
\begin{equation}
	Q(\Lambda_{t_0}, \dot{\Lambda}_{t_0})(v,v') = \left.\frac{d}{dt} \underline{\alpha}_t (v,v')\right\vert_{t_0} = \left.\frac{d}{dt} \Omega \bigl(v,\alpha_t(v')\bigr) \right\vert_{t_0}
\end{equation}
where $\alpha_t : \Lambda_{t_0} \rightarrow W$ is the map corresponding to $\Lambda_t$ for a decomposition $\R^{2n} = \Lambda_{t_0} \oplus W$ with $W$ Lagrangian.
\begin{proposition}
	The symmetric bilinear form $Q(\Lambda_{t_0}, \dot{\Lambda}_{t_0}) : \Lambda_{t_0} \times \Lambda_{t_0} \rightarrow \R$ is independent of the choice of the supplementary Lagrangian $W$ to $\Lambda_{t_0} .$
\end{proposition}
\begin{proof}
	Indeed, if $\R^{2n} = \Lambda_{t_0} \oplus W = \Lambda_{t_0} \oplus W'$ then
	\begin{equation*}
		W' = \left\{ w + \beta(w) \vert w \in W \right\}
	\end{equation*}
	where $\beta : W \rightarrow \Lambda_{t_0}$ is a linear map such that $\Omega\bigl(\beta(w),w'\big) + \Omega\bigl(w,\beta(w')\bigr) = 0 \ \forall w,w' \in W .$
	If $\alpha_t : \Lambda_{t_0} \rightarrow W$ is the linear map defining $\Lambda_t$
	\begin{equation*}
		\Lambda_t = \left\{ v + \alpha_t(v) \vert v \in \Lambda_{t_0} \right\} = \left\{ \bigl(v - \beta \alpha_t (v)\bigr) + \bigl(\alpha_t (v) + \beta \alpha_t (v) \bigr) \vert v \in \Lambda_{t_0} \right\}
	\end{equation*}
	so that the linear map $\alpha'_t : \Lambda_{t_0} \rightarrow W'$ defining $\Lambda_t$ is given by
	\begin{eqnarray*}
		\alpha'_t \bigl(v - \beta \alpha_t (v) \bigr) = \alpha_t (v) + \beta \alpha_t (v) \ \textrm{or}\\
		\alpha'_t = (\Id + \beta) \circ \alpha_t \circ (\Id - \beta \alpha_t)^{-1}.
	\end{eqnarray*}
	Since $\alpha_{t_0} = 0, \left. \frac{d}{dt} \alpha'_t \right\vert_{t_0} = (\Id + \beta)\circ \left. \frac{d}{dt} \alpha_t \right\vert_{t_0}$ so that
	\begin{eqnarray*}
		\left. \frac{d}{dt} \Omega \bigl(v,\alpha'_t(v')\bigr) \right\vert_{t_0} &=& \left. \frac{d}{dt} \Omega \bigl(v, \alpha_t(v') + \beta \alpha_t(v')\bigr) \right\vert_{t_0} \\
		&=& \left. \frac{d}{dt} \Omega \bigl(v, \alpha_t (v')\bigr) \right\vert_{t_0}.
	\end{eqnarray*}
\end{proof}
\begin{lemma}\label{lem:naturalitedeQ}
	If $\psi \in \Sp(\R^{2n},\Omega_0)$ then
	\begin{equation*}
		Q(\psi \Lambda_{t_{0}}, \psi \dot{\Lambda}_{t_{0}})(\psi v, \psi v') = Q(\Lambda_{t_{0}}, \dot{\Lambda}_{t_{0}})(v,v') \quad \forall v,v' \in \Lambda_{t_{0}} .
	\end{equation*}
\end{lemma}
\begin{proof}
	Write $\R^{2n} = \Lambda_{t_{0}} \oplus W$ with $W$ Lagrangian and let $\alpha_t : \Lambda_{t_{0}} \rightarrow W$ be the linear map corresponding to $\Lambda_t$
	(i.e $\Lambda_t = \{ v+ \alpha_t(v) \vert v \in \Lambda_{t_{0}}\}$).
	
	Similarly, write $\R^{2n} = \psi\Lambda_{t_{0}} \oplus \psi W$ and $\psi W$ is Lagrangian since $\psi \in \Sp(\R^{2n},\Omega_0)$.
	The linear map $\alpha'_t : \psi\Lambda_{t_{0}} \rightarrow \psi W$ corresponding to $\psi \Lambda_t$ is given by	$\alpha'_t = \psi \alpha_t \psi^{-1}$
	since $\psi\Lambda_t = \{\psi v + \psi \alpha_t v \vert v \in \Lambda_{t_{0}}\}$.
	
	Hence
	\begin{eqnarray*}
		Q(\psi\Lambda_{t_{0}}, \psi \dot{\Lambda}_{t_{0}})(\psi v, \psi v') & = & \left.\frac{d}{dt} \Omega_0 (\psi v, \alpha'_t \psi v') \right\vert_{t_{0}} \\
		& = & \left. \frac{d}{dt}\Omega_0 (\psi v, \psi\alpha_t  v') \right\vert_{t_{0}} \\
		& = & \left.\frac{d}{dt} \Omega_0(v, \alpha_t v')\right\vert_{t_{0}} \\
		& = & Q(\Lambda_{t_{0}}, \dot{\Lambda}_{t_{0}})(v,v') .
	\end{eqnarray*}
\end{proof}

\subsection{The Robbin-Salamon index for a path of Lagrangians}
This index defined by Robbin and Salamon is invariant under homotopy with fixed endpoints and is additive for catenation of paths.
The definition depends on the choice of a reference Lagrangian $V \subset (\R^{2n},\Omega)$ and goes as follows.

Consider a smooth path of Lagrangians $\Lambda : [a,b] \rightarrow \mathcal{L}_n$.
A \emph{crossing} for $\Lambda$ is a number $t \in [a,b]$ for which $\dim \Lambda_t \cap V \neq 0 .$
At each crossing time $t\in [a,b]$ one defines the \emph{crossing form}
\begin{equation}
	\Gamma (\Lambda , V , t) = \left. Q \bigl( \Lambda_t , \dot{\Lambda}_t \bigr) \right\vert_{\Lambda_t \cap V} .
\end{equation}
A crossing $t$ is called \emph{regular} if the crossing form $\Gamma (\Lambda , V , t)$ is nondegenerate.
In that case $\Lambda_s \cap V = \{0\}$ for $s\neq t$ in a neighborhood of $t.$
\begin{definition}[\cite{RobbinSalamon}]
	For a curve $\Lambda : [a,b] \rightarrow \mathcal{L}_n$ with only regular crossings the \emph{Robbin-Salamon index} is defined as
	\begin{equation}\label{Maslov}
		\mu_{RS}(\Lambda , V) = \half \sign \Gamma(\Lambda, V, a) + \sum_{{\stackrel{a<t<b}{\mbox{\tiny $t$ crossing}}}} \sign \Gamma(\Lambda, V, t) + \half \sign \Gamma(\Lambda, V, b) .
	\end{equation}
\end{definition}
Robbin and Salamon show (Lemmas $2.1$ and $2.2$  in \cite{RobbinSalamon}) that two paths with only regular crossings which are homotopic with fixed endpoints have the same Robbin-Salamon index
and that every continuous path of Lagrangians is homotopic with fixed endpoints to one having only regular crossings.
These two properties allow to define the Robbin-Salamon index for every continuous path of Lagrangians and this index is clearly invariant under homotopies with fixed endpoints.
It depends on the choice of the reference Lagrangian $V$.
Robbin and Salamon show (\cite{RobbinSalamon}, Theorem $2.3$):
\begin{theorem}[\cite{RobbinSalamon}]\label{proplag}
	The index $\mu_{RS}$ has the following properties:
	\begin{enumerate}
		\item(Naturality) For $\psi \in \Sp(\R^{2n},\Omega) ~~$
			$\mu_{RS}(\psi \Lambda,\psi V ) = \mu_{RS}(\Lambda, V )$.
		\item(Catenation) For $a < c < b,~
			\mu_{RS}(\Lambda, V) = \mu_{RS}(\Lambda_{\vert_{[a,c]}}, V ) + \mu_{RS}(\Lambda_{\vert_{[c,b]}}, V )$.
		\item(Product) If $n'+n''= n$, identify $\mathcal{L}(n')\times \mathcal{L}(n'')$ as a submanifold of $\mathcal{L}(n)$ in the obvious way.
			Then $\mu_{RS}(\Lambda'\oplus\Lambda'' , V'\oplus V'') = \mu_{RS}(\Lambda', V') + \mu_{RS}(\Lambda'', V'')$.
		\item(Localization) If $V = R^n \times \{0\}$ and $\Lambda(t) = \Gr(A(t))$ where $A(t)$ is a path of symmetric matrices, then the Maslov index
			of  $\Lambda$ is given by $\mu_{RS}(\Lambda, V ) = \half \sign A(b) - \half \sign A(a)$.
		\item(Homotopy) Two paths $\Lambda_0, \Lambda_1 : [a, b] \rightarrow \mathcal{L}(n)$ with $\Lambda_0(a) =\Lambda_1(a)$ and $\Lambda_0(b) = \Lambda_1(b)$ are homotopic with fixed endpoints if and only if they
			have the same Maslov index.
		\item(Zero) Every path $\Lambda : [a, b] \rightarrow \Sigma_k(V )$, with $\Sigma_k(V )=\{\,W\in \mathcal{L}(n)\,\vert\, \dim W\cap V=k\,\}$, has Maslov index $\mu_{RS}(\Lambda, V ) = 0$.
	\end{enumerate}
\end{theorem}

\subsection{Generalized Conley-Zehnder index for a path of symplectic matrices}
Consider the symplectic vector space $(\R^{2n}\times\R^{2n} , \overline{\Omega} = -\Omega_0 \times \Omega_0).$
Given any linear map $\psi: \R^{2n} \rightarrow \R^{2n} ,$ its graph
\begin{equation*}
	\Gr\psi = \{ (x , \psi x) \vert x \in \R^{2n} \}
\end{equation*}
is a $2n$-dimensional subspace of $\R^{2n}\times\R^{2n}$ which is Lagrangian
if and only if $\psi$ is symplectic $\bigl(\psi \in \Sp(\R^{2n},\Omega_0)\bigr) .$

A particular Lagrangian is given by the diagonal
\begin{equation}
	\Delta = \Gr\Id = \{ (x,x) \vert x \in \R^{2n} \} .
\end{equation}
Remark that $\Gr (-\psi)$ is a Lagrangian subspace which is always supplementary to $\Gr\psi$ for $\psi \in \Sp(\R^{2n},\Omega_0) .$
In fact $\Gr \phi$ and $\Gr\psi$ are supplementary if and only if $\phi-\psi$ is invertible.
\begin{definition}\label{RS}
	The \emph{Robbin-Salamon index} of a continuous path of symplectic matrices $\psi :  [0,1] \rightarrow \Sp(\R^{2n},\Omega_0) : t\mapsto \psi_t$ is defined as the Robbin-Salamon index of the path of Lagrangians
	in $(\R^{2n}\times\R^{2n} , \overline{\Omega})$, 
	\begin{equation*}
		\Lambda= \Gr\psi : [0,1] \rightarrow \mathcal{L}_{2n} : t \mapsto \Gr\psi_t
	\end{equation*}
	when the fixed Lagrangian is the diagonal $\Delta$:
	\begin{equation}\label{hungry}
		\mu_{RS}(\psi):=\mu_{RS}( \Gr\psi  , \Delta).
	\end{equation}
\end{definition}
Note that this index is defined for any continuous path of symplectic matrices  but can have half integer values.\\

Note that a crossing for a smooth path $\Gr\psi$ is a number $t\in [0,1]$ for which $1$ is an eigenvalue of $\psi_t$ and 
\begin{equation*}
	\Gr\psi_t\cap \Delta=\{\, (x,x)\,\vert\, \psi_tx=x\,\}
\end{equation*}
is in bijection with $\Ker(\psi_t-\Id).$\\

The properties of homotopy, catenation and product of theorem \ref{proplag} imply that \cite{RobbinSalamon}
\begin{itemize}
	\item $\mu_{RS}$ is invariant under homotopies with fixed endpoints,
	\item $\mu_{RS}$ is additive under catenation of paths and
	\item $\mu_{RS}$ has the product property $\mu_{RS}(\psi'\oplus\psi'') = \mu_{RS}(\psi')+\mu_{RS}(\psi'')$ as in proposition \ref{proprietescz}.
\end{itemize}
The zero property of the Robbin-Salamon index of a path of Lagrangians becomes:
\begin{proposition}
If $\psi:[a,b]\rightarrow\Sp(\R^{2n},\Omega)$ is a path of matrices such that $\dim\Ker (\psi(t)-\Id) =k$ for all $t\in [a,b]$
then $\mu_{RS}(\psi)=0$.
\end{proposition}
Indeed, $\Gr\psi_t\cap\Delta = \{ v\in\R^{2n} \vert \psi_tv=v\}$ so
$\dim (\Gr\psi_t \cap \Delta) = k$ if and only if $\dim\Ker (\psi(t)-\Id) =k$.
\begin{proposition}[Naturality]\label{prop:invariance}
	Consider two continuous paths of symplectic matrices $\psi , \phi :  [0,1] \rightarrow \Sp(\R^{2n},\Omega_0)$ and define $\psi' = \phi\psi\phi^{-1}$.
	Then
	\begin{equation*}
		\mu_{RS}(\psi')=\mu_{RS}(\psi)
	\end{equation*}
\end{proposition}
\begin{proof}
	One has
	\begin{eqnarray*}
		\Lambda'_t := \Gr \psi'_t &=& \{ (x,\phi_t\psi_t\phi_t^{-1}x) \,\vert \,x\in \R^{2n}\}\\
		&=& \{ (\phi_ty,\phi_t\psi_ty) \,\vert\, y\in \R^{2n}\}\\
		&=& (\phi_t\times\phi_t) \Gr \psi_t \\
		&=& (\phi_t\times\phi_t)\Lambda_t
	\end{eqnarray*}
	and $(\phi_t\times\phi_t)\Delta = \Delta$.
	Furthermore $(\phi_t\times\phi_t) \in \Sp(\R^{2n}\times\R^{2n} , \overline{\Omega}).$\\	
	Hence $t \in [0,1]$ is a crossing for the path of Lagrangians $\Lambda' = \Gr\psi'$ if and only if
	$\dim \Gr\psi'_t \cap \Delta \neq 0$ if and only if $\dim (\phi_t \times \phi_t)(\Gr\psi_t \cap \Delta) \neq 0$
	if and only if $t$ is a crossing for the path of Lagrangian $\Lambda = \Gr\psi$.
	
	By homotopy with fixed endpoints, we can assume that $\Lambda$ has only regular crossings and
	$\phi$ is locally constant around each crossing $t$ so that
	\begin{equation*}
		\frac{d}{dt}({\phi\psi\phi^{-1}})(t) = \phi_t\dot{\psi}_t\phi^{-1}_t .
	\end{equation*}
	Then at each crossing
	\begin{eqnarray*}
		\Gamma(\Gr\psi',\Delta,t) &=& Q(\Lambda'_t , \dot{\Lambda'}_t)\vert_{\Gr\psi'_t\cap\Delta}\\
		&=& Q((\phi_t \times \phi_t)\Lambda_t , (\phi_t \times \phi_t)\dot{\Lambda}_t)\vert_{(\phi_t \times \phi_t)\Gr\psi_t\cap\Delta}\\
		&=& Q(\Lambda_t , \dot{\Lambda}_t)\vert_{\Gr\psi_t\cap\Delta} \circ (\phi^{-1}_t \times \phi^{-1}_t)\otimes(\phi^{-1}_t \times \phi^{-1}_t)
	\end{eqnarray*}
	in view of Lemma \ref{lem:naturalitedeQ}, so that
	\begin{equation*}
		\sign \Gamma(\Gr\psi',\Delta,t) = \sign \Gamma(\Gr\psi,\Delta,t).
	\end{equation*}
\end{proof}
\begin{definition}
	For any smooth path $\psi$ of symplectic matrices, define a path of symmetric matrices $S$ through
	\begin{equation*}
		\dot{\psi}_t = J_0S_t\psi_t.
	\end{equation*}
	This is  indeed  possible since $\psi_t\in\Sp(\R^{2n},\Omega_0)\, \forall t$, thus $\psi_t ^{-1}\dot{\psi}_t$ is in the Lie algebra $sp(\R^{2n},\Omega_0)$
	and every element of this Lie algebra may be written in the form $J_0S$ with $S$ symmetric.
\end{definition}
The symmetric bilinear form $Q\bigl(\Gr\psi,\frac{d}{dt}{\Gr\psi}\bigr)$ is given as follows.
For any $t_0\in[0,1],$ write $\R^{2n}\times\R^{2n}=\Gr\psi_{t_0}\oplus \Gr(-\psi_{t_0}).$
The linear map $\alpha_t : \Gr\psi_{t_0}\rightarrow \Gr(-\psi_{t_0})$ corresponding to $\Gr\psi_t$ is obtained from:
\begin{equation*}
	(x,\psi_tx)=(y,\psi_{t_0}y) + \alpha_t(y,\psi_{t_0}y)=(y,\psi_{t_0}y)+(\widetilde{\alpha}_t y,-\psi_{t_0}\widetilde{\alpha}_t y)
\end{equation*}
if and only if $(\Id+\widetilde{\alpha}_t )y=x$ and $\psi_{t_0}(\Id-\widetilde{\alpha}_t )y=\psi_tx$, hence $\psi_{t_0}^{-1}\psi_t(\Id +\widetilde{\alpha}_t)=\Id - \widetilde{\alpha}_t$ and
\begin{equation*}
	\widetilde{\alpha}_t=(\Id +\psi_{t_0}^{-1}\psi_t)^{-1}(\Id-\psi_{t_0}^{-1}\psi_t)\quad\qquad \left. \frac{d}{dt}\widetilde{\alpha}_t \right\vert_{t_0}=-\half\psi_{t_0}^{-1}\dot\psi_{t_0}.
\end{equation*}
Thus
\begin{eqnarray*}
	\lefteqn{Q \Bigl( \Gr\psi_{t_0},\frac{d}{dt}\Gr\psi_{t_0} \Bigr) \bigl( ( v,\psi_{t_0}v),(v',\psi_{t_0}v' ) \bigr) }\\
	& = & \left. \frac{d}{dt} \overline{\Omega} \bigl( (v,\psi_{t_0}v), \alpha_t(v',\psi_{t_0}v') \bigr) \right\vert_{t_0}\\
	& = & \left.\frac{d}{dt} \overline{\Omega} \bigl( (v,\psi_{t_0}v), (\widetilde{\alpha}_tv',-\psi_{t_0} \widetilde{\alpha}_tv' ) \bigr) \right\vert_{t_0}\\
	& = & -2\Omega_0 \Bigl( v,\left.\frac{d}{dt} \widetilde{\alpha_t} \right\vert_{t_0}v' \Bigr)\\
	& = & \Omega_0(v,\psi_{t_0}^{-1}\dot\psi_{t_0}v')\\
	& = & \Omega_0(\psi_{t_0}v,J_0S_{t_0}\psi_{t_0}v').
\end{eqnarray*}
Hence the restriction of $Q$ to $\Ker(\psi_{t_0}-\Id)$ is given by
\begin{equation*}
	Q \Bigl( \Gr\psi_{t_0},\frac{d}{dt}\Gr\psi_{t_0} \Bigr) \bigl( ( v,\psi_{t_0}v),(v',\psi_{t_0}v' ) \bigr) = v^\tau S_{t_0}v' \quad \forall v,v' \in \Ker(\psi_{t_0}-\Id)
\end{equation*}
A crossing $t_0\in [0,1]$ is thus regular for the smooth path $\Gr\psi$ if and only if the restriction of $S_{t_0}$ to $\Ker(\psi_{t_0}-\Id)$ is nondegenerate.
We thus give the following definition
\begin{definition}\cite{RobbinSalamon}\label{cross}
	Let $\psi :  [0,1] \rightarrow \Sp(\R^{2n},\Omega_0) : t\mapsto \psi_t$ be a smooth path of symplectic matrices. Write $\dot\psi_t=J_0S_t \psi_t$ with $t\mapsto S_t$ a path of symmetric matrices.
	
	A number $t \in \left[0,1\right]$ is called a \emph{crossing} if $\det (\psi_t-\Id) = 0.$
	
	For $t\in [0,1]$,  the \emph{crossing form} $\Gamma(\psi,t)$ is defined as the quadratic form which is the restriction of $S_t$ to $\Ker (\psi_t-\Id).$
	
	A crossing $t_0$ is called \emph{regular} if the crossing form $\Gamma(\psi,t_0)$ is nondegenerate.		
\end{definition}
\begin{proposition}\label{murs}
	For a smooth path $\psi:  [0,1] \rightarrow \Sp(\R^{2n},\Omega_0) : t\mapsto \psi_t$ having only regular crossings, the Robbin-Salamon index introduced in definition \ref{RS} is given by
	\begin{equation}\label{muCZRS}
		\mu_{RS}(\psi) = \half \sign \Gamma(\psi,0) + \sum_{{\stackrel{t \textrm{ crossing,}}{\mbox{\tiny{$t\in ]0,1[$}}}}} \sign \Gamma(\psi,t)+\half \sign \Gamma(\psi,1).
	\end{equation}
\end{proposition}
\begin{proposition}
	Let $\psi : [0,1] \rightarrow \Sp(\R^{2n},\Omega_0)$ be a continuous path of symplectic matrices such
	that $\psi(0)=\Id$ and such that $1$ is not an eigenvalue of $\psi(1)$ (i.e. $\psi \in \textrm{SP}(n)$).
	The Robbin-Salamon index of $\psi$ defined by \eqref{hungry} coincides with the Conley-Zehnder index of $\psi$ as in definition \ref{def1}.
	In particular, for a smooth path  $\psi \in \textrm{SP}(n)$ having only regular crossings, the Conley-Zehnder index is given by
	\begin{equation}
		\mu_{CZ}(\psi) = \half\sign(S_0)+ \sum_{{\stackrel{t \textrm{ crossing,}}{\mbox{\tiny{$t\in ]0,1[$}}}}} \sign \Gamma(\psi,t) 
	\end{equation}
	with  $S_0=-J_0\dot\psi_0.$
\end{proposition}
\begin{proof}
	Since the Robbin-Salamon index for paths of Lagrangians is invariant under homotopies with fixed end points, the Robbin-Salamon index  for paths of symplectic matrices is also invariant under homotopies  with fixed endpoints.
	
	Its restriction to $\textrm{SP}(n)$ is actually invariant under homotopies of paths in $\textrm{SP}(n)$ since for any path in $\textrm{SP}(n)$, the starting point $\psi_0=\Id$  is fixed and the endpoint $ \psi_1$ can only move
	in a connected component of $\Sp^*(\R^{2n},\Omega_0)$ where no matrix has $1$ as an eigenvalue.\\

	To show that this index coincides with the Conley-Zehnder index, it is enough, in view of proposition \ref{caract}, to show that it satisfies the loop and signature properties.\\

	Let us prove the signature property. Let $\psi_t=\exp(tJ_0S)$ with $S$ a symmetric nondegenerate matrix with all eigenvalues of absolute value $ <2\pi$, so that 
	$\Ker(\exp(tJ_0S)-\Id)=\{0\}$ for all $t\in ]0,1] .$ Hence the only crossing is at $t=0$, where $\psi_0=\Id$ and $\dot\psi_t= J_0 S \psi_t$ so that $S_t=S$ for all $t$ and 
	\begin{equation*}
		\mu_{CZ}(\psi)=\half\sign S_0=\half\sign S.
	\end{equation*}

	To prove the loop property, note that $\mu_{RS}$ is additive for catenation and invariant under homotopies with fixed endpoints.
	Since we have seen that $(\phi\psi)$ is homotopic to the catenation of $\phi$ and $\psi$,
	it is enough to show that the Robbin-Salamon index of a loop is equal to $2\deg (\rho\circ \phi).$
	Since two loops $\phi$ and $\phi'$ are homotopic if and only if $\deg(\rho\circ\phi)=\deg(\rho\circ\phi'),$ it is enough to consider the loops $\phi_n$ defined by
	\begin{equation*}
		\phi_n(t):=
		\left(\begin{array}{cccc} 
			\cos 2\pi nt & -\sin 2\pi nt &0&0\\
			\sin 2\pi nt & \cos 2\pi nt &0&0\\
			0&0& a(t)\Id &0\\
			0&0&0& a(t)^{-1}\Id
		 \end{array}\right)
		 \in \Sp(\R^2\oplus \R^{2n-2})
	\end{equation*}
	with $a: [0,1] \rightarrow \R^+$ a smooth curve with $a(0)=a(1)=1$ and $a(t)\neq 1$ for $t\in ]0,1[.$
	Since $\rho\bigl(\phi_n(t)\bigr)=e^{2\pi int}$, we have $\deg(\phi_n)=n$.\\
	The crossings of $\phi_n$ arise at $t=\frac{m}{n}$ with $m$ an integer between $0$ and $n.$ At such a crossing, $\Ker\bigl(\phi_n(t)\bigr)$ is $\R^2$ for $0<t<1$ and is $\R^{2n}$ for $t=0$ and $t=1.$ We have
	\begin{equation*}
		\dot\phi_n(t)=
		\left(\begin{array}{cccc} 
			0 & -2\pi n & 0 & 0\\
			2\pi n&0 & 0 &0 \\
			0 & 0 & \frac{\dot a(t)}{a(t)}\Id & 0\\
			0&0&0& -\frac{\dot a(t)}{a(t)}\Id
		 \end{array}\right)
		 \phi_n(t)
	\end{equation*}
	so that 
	\begin{equation*}
		S(t)=
		\left(\begin{array}{cccc} 
			2\pi n &0&0&0\\
			0&2\pi n&0 & 0\\
			0&0&0&- \frac{\dot a(t)}{a(t)}\Id \\
			0&0& -\frac{\dot a(t)}{a(t)}\Id&0
		\end{array}\right).
	\end{equation*}
	Thus  $\sign \Gamma(\phi_n,t)=2$ for all crossings $t=\frac{m}{n} ,\ 0 \le m \le n$.
	From equation \eqref{muCZRS} we get
	\begin{eqnarray*}
		\mu_{RS}(\phi_n)&=&\half\sign \Gamma(\phi_n,0)+\sum_{0<m<n} \sign \Gamma\bigl(\phi_n,\tfrac{m}{n}\bigr)+\half\sign \Gamma(\phi_n,1)\\
		&=&1+2(n-1)+1=2n=2\deg(\rho\circ \phi_n)
	\end{eqnarray*}
	and the loop property is proved.
	Thus the Robbin-Salamon index for paths in $\textrm{SP}(n)$ coincides with the Conley-Zehnder index.\\

	The formula for the Conley-Zehnder index of a path $\psi \in \textrm{SP}(n)$ having only regular crossings, follows then from \eqref{muCZRS}.
	Indeed, we have $\Ker(\psi_1-\Id)=\{0\}$, while $\Ker(\psi_0-\Id)=\R^{2n}$ and $\Gamma(\psi,0)=S_0.$
\end{proof}
\begin{definition}
A \emph{symplectic shear} is a path of symplectic matrices of the form $\psi_t =
	\left(\begin{array}{cc}
		\Id & B(t) \\
		0 & \Id
	\end{array} \right)$
	with $B(t)$ symmetric.
\end{definition}
\begin{proposition}\label{RSforshears}
	The Robbin-Salamon index $\mu_{\textrm{RS}}$ on a symplectic shear $\psi_t =
	\left(\begin{array}{cc}
		\Id & B(t) \\
		0 & \Id
	\end{array} \right)$
	(with $B(t)$ symmetric) is equal to 
	$$
		\mu_{\textrm{RS}}(\psi)=\half \sign B(0) - \half \sign B(1).
	$$
\end{proposition}
\begin{proof}
		Since any symmetric matrix $B(t)$ is diagonalisable, we  write $B(t)= A(t)^\tau D(t)A(t)$ with $A(t)\in O(\R^{n})$  and  $D(t)$ a diagonal matrix.\\
          The matrix $\phi_t=
          \left(\begin{array}{cc}
		A(t)^\tau & 0 \\
		0 & A(t)
	\end{array} \right)$
	is in $\Sp(\R^{2n},\Omega_0)$ and 
         \begin{equation*}
		\psi'_t:=\phi_t\psi_t\phi_t^{-1}=
		\left(\begin{array}{cc}
			\Id & D(t) \\
			0 & \Id
		\end{array} \right).
        \end{equation*}
	By proposition \ref{prop:invariance} $\mu_{\textrm{RS}}(\psi)=\mu_{\textrm{RS}}(\psi');$ by the product property it is enough to show that $\mu_{\textrm{RS}}(\psi)=\half \sign d(0)-\half \sign d(1)$ for the path 
	$$
		\psi: [0,1]\rightarrow \Sp(\R^2,\Omega_0) : t \mapsto \psi_t=
		\left(\begin{array}{cc}
			1 & d(t) \\
			0 & 1
		\end{array} \right).
	$$
	Since $\mu_{\textrm{RS}}$ is invariant under homotopies with fixed end points, we may assume $\psi_t=
	\left(\begin{array}{cc}
		a(t)& d(t) \\
		c(t) &a(t)^{-1}\bigl(1+ d(t)c(t)\bigr)
	\end{array} \right)$
	with $a,c$  smooth functions such that :\\
	$a(0)=a(1)=1, ~\dot a(0)\neq 0,\, \dot a(1)\neq 0$ and $ a(t)>1$ for $0<t<1;$\\
	$c(0)=c(1)=0,~c(t)d(t)\ge 0~ \forall t$ and $\dot c(t)\neq 0 $ (resp. $=0$) when $d(t)\neq 0 $ (resp. $=0$) for $t=0$ or $1$.\\
	The only crossings are $t=0$ and $t=1$  since the trace of $\psi(t)$ is $>2$ for $0<t<1.$
	Now, at those points ( $t=0$ and $t=1$)  $\dot\psi_t=
	\left(\begin{array}{cc}
		\dot a(t)&  \dot d(t)\\
		\dot c(t) & -\dot a(t)+d(t)\dot c(t)
	\end{array} \right)$
	so that $S_t= -J_0 \dot \psi_t\psi_t^{-1}=
	\left(\begin{array}{cc}
		\dot c(t)& -\dot a(t) \\
		-\dot a(t) & \dot a(t)d(t)-\dot d(t)
	\end{array} \right).$
 
	Clearly, at the crossings, we have  $ \Ker\psi_t =\R^2$ iff $d(t)=0$ and $ \Ker\psi_t $ is spanned by the first basis element iff $d(t)\neq 0,$ so that  from definition \ref{cross}
	$\Gamma(\psi,t)=\bigl( \dot c(t) \bigr)$
	when $d(t)\neq 0$ and $\Gamma(\psi,t)=
	\left(\begin{array}{cc}
		0& -\dot a(t) \\
		-\dot a(t) &0
	\end{array} \right)$
	when $d(t)= 0.$
	Hence both crossings are regular and $\sign \Gamma(\psi,t)=\sign \dot c(t)$ when $d(t)\neq 0$ and $\sign \Gamma(\psi,t)=0$ when $d(t)=0.$
	Since $d(t)c(t)\ge 0$ for all $t,$ we clearly have $\sign \dot c(0)=\sign d(0)$ and $\sign \dot c(1)= -\sign d(1)$.
	Proposition \ref{murs} then gives $\mu_{\textrm{RS}}(\psi)=\half \Gamma(\psi,0)+\half \sign \Gamma(\psi,1)=\half \sign d(0)-\half\sign d(1).$
 \end{proof}
 
 \subsection{Characterization of the Robbin-Salamon index}
 In this section, we prove the following characterization of the Robbin-Salamon index.
 \begin{theorem}\label{thm:Thechar}
	The Robbin-Salamon index for a path of symplectic matrices is characterized by the following properties:
	\begin{itemize}
		\item\emph{(Homotopy)} it is invariant under homotopies with fixed end points;
		\item\emph{(Catenation)} it is additive under catenation of paths;
		\item\emph{(Zero)} it vanishes on any path $\psi:[a,b]\rightarrow\Sp(\R^{2n},\Omega)$ of matrices such that 
		         $\dim\Ker \bigl(\psi(t)-\Id\bigr) =k   $ is constant for all $t\in [a,b]$;
		\item\emph{(Normalization)} if $S=S\tr \in \R^{2n\times 2n}$ is a symmetric matrix with all eigenvalues of absolute value
			$<2\pi $ and if $\psi(t) = \operatorname{exp}(J_0St)$ for $t\in \left[0,1\right],$ then
			$\mu_{\textrm{RS}}(\psi) = \half \sign S$ where $\sign S$ is the signature of $S$;
	\end{itemize}			
\end{theorem}
Before proving this theorem, we show that the Robbin-Salamon index is characterized by the fact that it extends Conley-Zehnder index and has all the properties stated in the previous section. This is made explicit in  Lemma \ref{charRS1}. We then use the characterization of the Conley-Zehnder index given in
Proposition \ref{caract} to give in Lemma \ref{charRS2} a characterization of the Robbin-Salamon index in terms of six properties. We then prove the theorem stated above.
\begin{lemma}\label{charRS1}
The  Robbin-Salamon index is characterized by the following properties:
\begin{enumerate}
\item{(Generalization)} it is a correspondence $\mu_{\textrm{RS}}$ which associates a half integer to any continuous path $\psi:[a,b]\rightarrow \Sp(\R^{2n},\Omega_0)$ of symplectic matrices and it coincides with
$\mu_{\textrm{CZ}}$ on paths starting from the identity matrix and ending at a matrix for which $1$ is not an eigenvalue;
\item \emph{(Naturality)} if $\phi,\psi : \left[0,1\right] \rightarrow \textrm{Sp}(\R^{2n},\Omega_0)$, we have
		 	$\mu_{\textrm{RS}}(\phi\psi\phi^{-1}) = \mu_{\textrm{RS}}(\psi)$;
\item\emph{(Homotopy)} it is invariant under homotopies with fixed end points;
\item\emph{(Catenation)} it is additive under catenation of paths;
\item\emph{(Product)} it has  the product property $\mu_{\textrm{RS}}(\psi'\oplus \psi'')=\mu_{\textrm{RS}}(\psi')+\mu_{\textrm{RS}}(\psi'')$;
\item\emph{(Zero)} it vanishes on any path $\psi:[a,b]\rightarrow\Sp(\R^{2n},\Omega)$ of matrices such that $\dim\Ker (\psi(t)-\Id) =k$ is constant for all $t\in [a,b]$;
\item\emph{(Shear)}on a symplectic shear $,\psi : \left[0,1\right] \rightarrow \textrm{Sp}(\R^{2n},\Omega_0)$ of the form $$\psi_t =
	\left(\begin{array}{cc}
		\Id & -tB \\
		0 & \Id
	\end{array} \right)=\exp t \left(\begin{array}{cc}
		0 &- B \\
		0 & 0
	\end{array} \right)=\exp t J_0  \left(\begin{array}{cc}
		0 &0 \\
		0 & B
	\end{array} \right)$$
	with $B$ symmetric, it is equal to $
		\mu_{\textrm{RS}}(\psi)= \half \sign B.$
\end{enumerate}
\end{lemma}
\begin{proof}
We have seen in the previous section that the index $\mu_{\textrm{RS}}$ defined by Robbin and Salamon satisfies all the above properties.
To see that those properties characterize this index, it is enough to show (since the group $\Sp(\R^{2n},\Omega_0)$ is connected and since we have the catenation property)
that those properties determine the index of any path starting from the identity. Since it must be a generalization of the Conley-Zehnder index
and must be additive for catenations of paths, it is enough to show that any symplectic matrix $A$ which admits $1$ as an eigenvalue can be linked to a matrix $B$ which does not admit 
$1$ as an eigenvalue by a continuous path whose index is determined by the properties stated. We have seen in 
Theorem \ref{normalforms1} that there is a basis of $\R^{2n}$ such that $A$ is the direct symplectic sum of a matrix which does not admit $1$ as eigenvalue and 
matrices of the form 
$$
A^{(1)}_{r_j,d_j}:=
\left(\begin{array}{cc}
J(1,r_j)& \operatorname{diag}(0,\ldots, 0,d_j)\bigl( J(1,r_j)^{-1}\bigr)^{\tau}\\
		0& \bigl( J(1,r_j)^{-1}\bigr)^{\tau}
		\end{array}\right);
$$
with $d_j$ equal to $0,1$ or $-1$. The dimension of the eigenspace of eigenvalue $1$  for 
$A^{(1)}_{r_j,d_j}$ is equal to $1$ if $d_j\neq 0$ and is equal to $2$ if $d_j=0$.
 In view of the naturality and the product property of the index, we can consider a direct sum of paths with the constant path on the symplectic subspace where $1$ is not an eigenvalue and we just have to build a path  in $\Sp(\R^{2r_j},\Omega_0)$ from $A^{(1)}_{r_j,d_j}$
 to a matrix which does not admit $1$ as eigenvalue and whose index is determined by the properties given in the statement.
 This we do by the catenation of three paths : we first build the path $\psi_1: [0,1] \rightarrow \Sp(\R^{2r_j},\Omega_0)$
defined by 
$$ \psi_1(t):=\left(\begin{array}{cc}
D(t,r_j)& \operatorname{diag}\bigl(c(t),0,\ldots,0,d(t)\bigr)\bigl( D(t,r_j)^{-1}\bigr)^{\tau}\\
		0& \bigl( D(t,r_j)^{-1}\bigr)^{\tau}
		\end{array}\right)
 $$
 with  $D(t,r_j)=\left(\begin{array}{cccccc} 1&1-t&0&\ldots &\ldots &0\\
                                                                      0&e^t&1-t&0&\ldots &0\\
                                                                      \vdots &0&\ddots &\ddots &0&\vdots\\
                                                                      0&\ldots &0&e^t&1-t &0\\
                                                                        0&\ldots &\ldots &0&e^t&1-t\\
                                                                          0&\ldots &\ldots &\ldots &0&e^t
                                                                      \end{array}\right)$, \\
                                                                      and 
 with $c(t)=td_j, ~d(t)=(1-t)d_j$. 
 Observe  that $\psi_1(0)=A^{(1)}_{r_j,d_j}$ and $\psi_1(1)$ is the symplectic direct sum of 
$\left(\begin{array}{cc} 1&c(1)=d_j\\
                                          0& 1
                                           \end{array}\right)$ and $\left(\begin{array}{cc}e\Id_{r_j-1}  &0\\0&e^{-1}\Id_{r_j-1}                                                                    \end{array}\right)$ and this last matrix does not admit $1$ as eigenvalue.\\

Clearly  $\dim\ker\bigl(\psi_1(t)-\Id\bigr)=2$ for all $t\in [0,1]$ when $d_j=0$; we now prove that $\dim\ker(\psi_1(t)-\Id)=1$ for all $t\in [0,1]$ when $d_j\neq0$. Hence the index of $\psi_1$ must always be zero by the zero property.\\
To prove that $\dim\ker(\psi_1(t)-\Id)=1$ we have to show the non vanishing of  the determinant of the $2r_j-1\times 2r_j-1$ matrix
$$
{\left(\begin{array}{ccccccccc} 1-t&0&\ldots &0&c(t)&0&\ldots&\ldots&0\\
                                                                      e^t-1&1-t&\ddots &0&0&0&\ldots&\ldots&0\\
                                                        0 &\ddots&\ddots &0&\vdots&\vdots&&&\vdots\\
                                                                         0&\ddots  &e^t-1&1-t&0&0&\ldots&\ldots&0\\
                                                                         \vdots&\ldots &0 &e^t-1&E^{r_j}_1(t)d(t)&E^{r_j}_2(t)d(t)&\ldots&\ldots&E^{r_j}_{r_j}(t)d(t)\\
                                                                         \vdots&\ldots&\ldots&0&E^2_{1}(t)&E^2_2(t)-1&\ldots&\ldots&E^{1}_{r_j}(t)\\
                                                                          0&\ldots&\ldots&0&\vdots&&&&\vdots\\
                                                                          \vdots&&&\vdots&\vdots&\vdots&\ddots&\ddots&\vdots\\
                                                                            0&\ldots&\ldots  &0&E^{r_j}_1(t)&E^{r_j}_2(t)&\ldots&\ldots&E^{r_j}_{r_j}(t)-1\\
  \end{array}\right) }    
 $$
where $E(t)$ is the  transpose of the inverse of $D(t,r_j)$ so is lower triangular with $E^{i}_i=e^{-t}$ for all $i>1$. 
This determinant is equal to
$$
(-1)^{r_j-1}c(t)(e^t-1)^{r_j-1}(e^{-t}-1)^{r_j-1}+(-1)^{r_j-1}d(t)(1-t)^{r_j-1}\det E'(t)
$$
where $E'(t)$ is obtained by deleting the first line and the last column in $E(t)-\Id$ so is equal to 
$$
{{ \left(\begin{array}{cccccc}
                                  (t-1)e^{-t}&e^{-t}-1&0&\ldots&0\\
                                    (t-1)^2e^{-2t}&(t-1)e^{-2t}&e^{-t}-1&\ddots&0\\[6mm]
                                     \vdots&&\ddots&\ddots&0\\[6mm]
                                          (t-1)^{r_j-2}e^{-(r_j-2)t}&(t-1)^{r_j-3}e^{-(r_j-2)t}&\ldots&(t-1)e^{-2t}&e^{-t}-1\\
                                       (t-1)^{r_j-1}e^{-(r_j-1)t}&(t-1)^{r_j-2}e^{-(r_j-1)t}&\ldots&(t-1)^2e^{-3t}&(t-1)e^{-2t}
  \end{array}\right) }}                                                                
$$
hence $\det E'(t)=(t-1) \det F_{r_j-2}(t)$ where
$$
F_m(t):={{ \left(\begin{array}{ccccc}
                                    (t-1)e^{-2t}&e^{-t}-1&0&\ddots&0\\
                                    (t-1^2)e^{-3t}&(t-1)e^{-2t}&e^{-t}-1&\ddots&\vdots\\[6mm]
                                     \vdots&\ddots&\ddots&\ddots&0\\[6mm]
                                      (t-1)^{m-1}e^{-(m-2t}&\ddots&(t-1)e^{-2t}&e^{-t}-1\\
                                      (t-1)^{m}e^{-(m-1)t}&\ldots&(t-1)^2e^{-3t}&(t-1)e^{-2t}
  \end{array}\right) }}                                                                
$$
and  we have $\det F_m(t)= (t-1)e^{-2t}\det F_{m-1}(t)-(e^{-t}-1)(t-1)e^{-t}\det F_{m-1}(t)=(t-1)e^{-t}\det F_{m-1}(t)$
so that, by induction on $m$, $\det F_m(t)=(t-1)^me^{-(m+1)t}$ hence the determinant we have to study is\\
$
(-1)^{r_j-1}c(t)(2-e^t-e^{-t})^{r_j-1}-(-1)^{r_j-1}d(t)(1-t)^{r_j} \det F_{r_j-2} (t)$\\
which is equal to \\
$
(-1)^{r_j-1}c(t)(2-e^t-e^{-t})^{r_j-1}+(-1)^{r_j}d(t)(1-t)^{r_j}(t-1)^{r_j-2}e^{-(r_j-1)t}$\\
hence to 
$$
c(t)(e^t+e^{-t}-2)^{r_j-1}+d(t)(1-t)^{2r_j-2}e^{-(r_j-1)t}
$$
 which never vanishes if $c(t)=td_j$ and
$d(t)=(1-t)d_j$ since $e^t+e^{-t}-2$ and $(1-t)$ are $\ge 0$.\\

We then  construct a path $\psi_2: [0,1] \rightarrow \Sp(\R^{2r_j},\Omega_0)$ which is constant on the symplectic subspace where $1$ is not an eigenvalue  and which
 is a symplectic shear on the first two dimensional symplectic vector space, i.e. 
 $$
 \psi_2(t):= \left(\begin{array}{cc} 1&(1-t)d_j\\
                                          0& 1
                                           \end{array}\right) \oplus    \left(\begin{array}{cc}e\Id_{r_j-1}  &0\\0&e^{-1}\Id_{r_j-1}                                                                    \end{array}\right);                                           
 $$  
then the index of $\psi_2$ is equal to $\half\sign d_j$. 
Observe that $\psi_2$ is constant if $d_j=0$; then the index of $\psi_2$ is zero. Observe that in all cases $\psi_2(1)=\Id_2\oplus  \left(\begin{array}{cc}e\Id_{r_j-1}  &0\\0&e^{-1}\Id_{r_j-1}\end{array}\right)$. \\
We then build $\psi_3: [0,1] \rightarrow \Sp(\R^{2r_j},\Omega_0)$ given by
$$
\psi_3(t):=\left(\begin{array}{cc} e^t & 0 \\ 
                                                        0 & e^{-t}
                  \end{array}\right)       \oplus    
                  \left(\begin{array}{cc}  e\Id_{r_j-1}  &0\\
                                                                             0&e^{-1}\Id_{r_j-1}  
                  \end{array}\right)
$$
which is the direct sum of a path whose Conley-Zehnder index is known and a constant path whose index is zero.
Clearly $1$ is not an eigenvalue of $\psi_3(1)$.
\end{proof}

Combining the above with the characterization of the Conley-Zehnder index, we now prove:
\begin{lemma}\label{charRS2}
	The Robbin-Salamon index for a path of symplectic matrices is characterized by the following properties:
	\begin{itemize}
		\item\emph{(Homotopy)} it is invariant under homotopies with fixed end points;
		\item\emph{(Catenation)} it is additive under catenation of paths;
		\item\emph{(Zero)} it vanishes on any path $\psi:[a,b]\rightarrow\Sp(\R^{2n},\Omega)$ of matrices such that $\dim\Ker (\psi(t)-\Id) =k$ is constant for all $t\in [a,b]$;
		\item\emph{(Product)} it has  the product property $\mu_{\textrm{RS}}(\psi'\oplus \psi'')=\mu_{\textrm{RS}}(\psi')+\mu_{\textrm{RS}}(\psi'')$;
		\item\emph{(Signature)} if $S=S\tr \in \R^{2n\times 2n}$ is a symmetric non degenerate matrix with all eigenvalues of absolute value
			$<2\pi $ and if $\psi(t) = \operatorname{exp}(J_0St)$ for $t\in \left[0,1\right],$ then
			$\mu_{\textrm{RS}}(\psi) = \half \sign S$ where $\sign S$ is the signature of $S$;
		\item\emph{(Shear)} if $\psi_t = \exp t J_0  \left(\begin{array}{cc}
		0 &0 \\
		0 & B
	\end{array} \right)$ for $t\in \left[0,1\right],$
	with $B$ symmetric, then $
		\mu_{\textrm{RS}}(\psi)= \half \sign B.$
	\end{itemize}
\end{lemma}

\begin{proof}
Remark first that the invariance by homotopies with fixed end points, the additivity under catenation and the zero property
imply the naturality; they also imply the constancy on the components of $ \textrm{SP}(n).$ 
The signature property stated above is the signature property which arose in the characterization of the Conley-Zehnder index given in proposition \ref{caract}.
To be sure that our index is a generalization of the Conley-Zehnder index, there remains just  to prove the loop property.
Since the product of a loop $\phi$ and a path $\psi$ starting at the identity is homotopic to the catenation of $\phi$ and $\psi$,
it is enough to prove that the index of a loop $\phi$  with $\phi(0) = \phi(1) = \Id$ is given by
$ 2 \operatorname{deg}(\rho \circ \phi)$.
Since two loops $\phi$ and $\phi'$ are homotopic if and only if $\deg(\rho\circ\phi)=\deg(\rho\circ\phi'),$ it is enough to consider the loops $\phi_n$ defined by $\phi_n(t):=
		\left(\begin{array}{cc} 
			\cos 2\pi nt & -\sin 2\pi nt \\
			\sin 2\pi nt & \cos 2\pi nt \end{array}\right)\oplus \Id ~
		 \in \Sp(\R^2\oplus \R^{2n-2})$ and since $\phi_n(t)=(\phi_1(t))^n$ it is enough to show, using the homotopy, catenation, product and zero properties that 	the index of the loop given by $\phi(t)=	\left(\begin{array}{cc} 
			\cos 2\pi t & -\sin 2\pi t \\
			\sin 2\pi t & \cos 2\pi t \end{array}\right)$ for $t\in [0,1]$ is equal to $2$. This is true, using the signature property,			writing $\phi$ as the catenation of the path $\psi_1(t):=\phi(\frac{t}{2})=\exp tJ_0\left(\begin{array}{cc} 
			\pi  & 0\\
			0 & \pi  \end{array}\right)$ for $t\in [0,1]$  whose index is $1$ and the path $\psi_2(t):=\phi(\frac{t}{2})=\exp tJ_0\left(\begin{array}{cc} 
			\pi  & 0\\
			0 & \pi  \end{array}\right)$ for $t\in [1,2]$. 
			We introduce the  path in the reverse direction  $\psi^-_2(t):=\exp -tJ_0\left(\begin{array}{cc} 
			\pi  & 0\\
			0 & \pi  \end{array}\right)$ for  $t\in [0,1]$ whose index is $-1$; since the catenation of $\psi^-_2$ and $\psi_2$ is homotopic to
			      the constant path whose index is zero,  the index of  $\phi_1$ is given by the index of $\psi_1$ minus the index of $\psi^-_2$
			      hence is equal to $2$.
\end{proof}

\begin{proof}[ of theorem \ref{thm:Thechar}] 
Observe that any symmetric matrix can be written as the symplectic direct sum of a non degenerate symmetric matrix $S$ and a matrix $S'$ of the form
$ \left(\begin{array}{cc}
		0 &0 \\
		0 & B
	\end{array} \right)$ where $B$ is symmetric and may be degenerate. The index of the path $\psi_t=\exp t J_0S'$ is equal to the index of the path $\psi'_t=\exp t \lambda J_0S'$
	for any $\lambda>0$. Hence the signature and shear conditions, in view of the product condition, can be simultaneously written as: 
	if $S=S\tr \in \R^{2n\times 2n}$ is a symmetric  matrix with all eigenvalues of absolute value
			$<2\pi $ and if $\psi(t) = \operatorname{exp}(J_0St)$ for $t\in \left[0,1\right],$ then
			$\mu_{\textrm{RS}}(\psi) = \half \sign S$. This is the normalization condition stated in the theorem.\\
			
From Lemma \ref{charRS2}, we just have to prove that the product property is a consequence of the other properties. We prove it for paths with values 
in $\Sp(\R^{2n},\Omega_0)$ by induction on $n$, the case $n=1$ being obvious.
Since $\psi'\oplus\psi''$ is homotopic with fixed endpoints to the catenation 
of $\psi'\oplus\psi''(0)$ and $\psi'(1)\oplus\psi''$, it is enough to show that the
index of $A\oplus\psi$ is equal to the index of $\psi$ for any fixed $A\in \Sp(\R^{2n'},\Omega_0)$ with $n'<n$ and any
continuous path $ \psi : [0,1] \rightarrow  \Sp(\R^{2n''},\Omega_0)$ with $n''<n$.

Observe also (using subsection \ref{subsecnormalforms} and proposition \ref{charRS1}) that any symplectic matrix $A$ can be linked by a path $\phi(s)$
with constant dimension of the $1$-eigenspace to a matrix of the form $\operatorname{exp}(J_0S')$ with $S'$ a 
symmetric  $n'\times n'$ matrix with all eigenvalues of absolute value $<2\pi $. The index of $A\oplus\psi$
is equal to the index of $\operatorname{exp}(J_0S')\oplus \psi$; indeed  $A\oplus\psi$ is homotopic with fixed endpoints
to the catenation of the three paths $\phi_s\oplus\psi(0)$, $\operatorname{exp}(J_0S')\oplus \psi$ and the path $\phi_s\oplus\psi (1)$
in  the reverse order, and the index of the first and third paths are zero since the dimension of the $1$-eigenspace does not vary along those paths.

Hence it is enough to show that the index of $\operatorname{exp}(J_0S')\oplus \psi$ is the same as the index of $\psi$. This is true because  the map 
$\mu$
sending  a path $\psi$  in $\Sp(\R^{2n''},\Omega_0)$ (with $n''<n$) to the index of $\operatorname{exp}(J_0S')\oplus \psi$ has the four properties stated in the theorem,
and these characterize the Robbin-Salamon index for those paths by induction hypothesis. 
It is clear that $\mu$ is invariant under homotopies, additive for catenation and equal to
zero on paths $\psi$ for which the dimension of the $1$-eigenspace is constant.  Furthermore
$\mu(\exp t(J_0S))$ which is the index of $\exp(J_0S')\oplus \exp t(J_0S)$ is equal to $\half \sign S$, because the path
$\exp tJ_0(S'\oplus S)$ whose index is $\half \sign (S'\oplus S)= \half \sign S' + \half \sign S$ is homotopic with fixed 
endpoints with the catenation of $\exp t(J_0S')\oplus \Id=\exp tJ_0(S'\oplus 0)$, whose index is $\half \sign S'$,
and the path $\exp(J_0S')\oplus \exp t(J_0S)$.
\end{proof}			
	
\subsection{Another Robbin-Salamon index for paths of symplectic matrices}
In \cite{RobbinSalamon} Robbin and Salamon associate to a path of symplectic matrices $\psi : [0,1]\rightarrow \Sp(\R^{2n},\Omega_0)$
the index
\begin{equation*}
	\mu_{\textrm{RS}2}(\psi) := \mu_{\textrm{RS}}(\psi V,V)
\end{equation*}
where $V = \{0\} \times \R^n$ is a fixed Lagrangian in $\R^{2n}$ and $\psi V$ is the path of Lagrangians defined by
$t \mapsto \psi_t V$.

The properties of theorem \ref{proplag} imply that \cite{RobbinSalamon}
\begin{itemize}
	\item $\mu_{RS2}$ is invariant under homotopies with fixed endpoints and two paths with the same endpoints are homotopic with fixed endpoints if and only if they have the same $\mu_{RS2}$ index,
	\item $\mu_{RS2}$ is additive under catenation of paths,
	\item $\mu_{RS2}$ has the product property $\mu_{RS}(\psi'\oplus\psi'') = \mu_{RS}(\psi')+\mu_{RS}(\psi'')$ as in proposition \ref{proprietescz},
	\item vanishes on a path whose image lies in $$\Sp_k(\R^{2n},\Omega_0,V) = \{ A \in \Sp(\R^{2n},\Omega_0) \, \vert \, \dim AV\cap V = k \}$$
		for a given $k \in \{0, \ldots, n \}$,
	\item has value $\half \sign B(0) - \half \sign B(1)$ when $\psi_t =
		\left(\begin{array}{cc}
			\Id & B(t) \\
			0 & \Id
		\end{array} \right)$.
\end{itemize}
Robbin and Salamon also prove in \cite{RobbinSalamon} that those properties characterize this index.
\begin{proposition}
	The two indices $\mu_{\textrm{RS}}$ and $\mu_{\textrm{RS}2}$ associated to paths of symplectic matrices do not coincide in general.
\end{proposition}
\begin{proof}
	Consider the path $\psi : [0,1]\rightarrow \Sp(\R^{2n},\Omega_0) : t\mapsto \psi_t=
	\left(\begin{array}{cc}
			\Id & 0\\
			C(t) & \Id
	\end{array} \right)$.
	Since $\psi_t V \cap V= V \quad \forall t , \quad \psi_t$ lies in $\Sp_n(\R^{2n},\Omega_0,V) \quad \forall t$ and
	$\mu_{\textrm{RS}2}(\psi) = 0$.
	
	Define $\phi =
	\left(\begin{array}{cc}
		0 & \Id \\
		-\Id & 0
	\end{array}\right)$
	and $\psi' = \phi\psi \phi^{-1}$ so that $\psi'_t =
	\left(\begin{array}{cc}
		\Id & -C(t) \\
		0 & \Id
	\end{array}\right)$.
	Then
	\begin{equation*}
		\mu_{\textrm{RS}2}(\psi') = \half \sign C(1) - \half \sign C(0)
	\end{equation*}
	which is in general different from $\mu_{\textrm{RS}2}(\psi)$.

	On the other hand, by \eqref{prop:invariance}, $\mu_{\textrm{RS}}(\psi) = \mu_{\textrm{RS}}(\psi')$.
\end{proof}
\begin{remarque}
	The index $\mu_{\textrm{RS}2}$ vanishes on a path whose image lies into one of the $(n+1)$ strata defined by
	$\Sp_k(\R^{2n},\Omega_0,V) = \{ A \in \Sp(\R^{2n},\Omega_0) \, \vert \, \dim AV\cap V = k \}$ for $0\le k\le n$,
	whereas the index $\mu_{\textrm{RS}}$ vanishes on a path whose image lies into one of the $(2n+1)$ strata defined by the set of symplectic matrices
	whose eigenspace of eigenvalue $1$ has dimension $k$ (for $0\le k\le 2n)$.
\end{remarque}
\begin{proposition}
	The two indices $\mu_{\textrm{RS}}$ and $\mu_{\textrm{RS}2}$ coincide on symplectic shears.
	\end{proposition}
\begin{proof}
	Indeed, Robbin and Salamon have shown that if $\psi_t =
	\left(\begin{array}{cc}
		\Id & B(t) \\
		0 & \Id
	\end{array} \right)$ then
	$
		\mu_{\textrm{RS}2}(\psi)=\half \sign B(0) - \half \sign B(1).
	$
	We have proven in proposition \ref{RSforshears} that $\mu_{\textrm{RS}}(\psi)=\half \sign B(0) - \half \sign B(1)$.
 \end{proof}
 


\end{document}